\documentclass[a4paper,10pt, oneside]{article}
\usepackage{amssymb,amsthm,bm,bbm,mathrsfs,mathtools}
\usepackage{authblk}
\usepackage[font=small,labelfont=md,textfont=it]{caption}
\usepackage{enumitem}
\usepackage[colorlinks,linkcolor=black,citecolor=black]{hyperref}
\usepackage{etoolbox}
\usepackage{longtable}
\usepackage{diagbox}
\usepackage{booktabs,makecell,multirow}
\usepackage{cases}          
\usepackage[most]{tcolorbox}

\usepackage[capitalise]{cleveref}
\usepackage{graphicx, tikz, pgfplots}
\pgfplotsset{
	my_log_axis/.style={
		xmode=log, ymode=log,
		log ticks with fixed point,
		xticklabel style={font=\small},
		yticklabel style={font=\small},
		legend style={font=\small, draw=none, fill=none, inner sep=2pt},
		grid=both,
		minor grid style={dashed, gray!20},
		major grid style={dashed, gray!40}
	}
}

\usepackage{subcaption} 

\crefname{assumption}{Assumption}{Assumptions}
\crefname{hypothesis}{Hypothesis}{Hypotheses}
\crefname{lemma}{Lemma}{Lemmas}
\crefname{theorem}{Theorem}{Theorems}

\crefformat{equation}{\textup{#2(#1)#3}}
\crefrangeformat{equation}{\textup{#3(#1)#4--#5(#2)#6}}
\crefmultiformat{equation}{\textup{#2(#1)#3}}{ and \textup{#2(#1)#3}}
{, \textup{#2(#1)#3}}{, and \textup{#2(#1)#3}}
\crefrangemultiformat{equation}{\textup{#3(#1)#4--#5(#2)#6}}%
{ and \textup{#3(#1)#4--#5(#2)#6}}{, \textup{#3(#1)#4--#5(#2)#6}}{, and \textup{#3(#1)#4--#5(#2)#6}}

\Crefformat{equation}{#2Equation~\textup{(#1)}#3}
\Crefrangeformat{equation}{Equations~\textup{#3(#1)#4--#5(#2)#6}}
\Crefmultiformat{equation}{Equations~\textup{#2(#1)#3}}{ and \textup{#2(#1)#3}}
{, \textup{#2(#1)#3}}{, and \textup{#2(#1)#3}}
\Crefrangemultiformat{equation}{Equations~\textup{#3(#1)#4--#5(#2)#6}}%
{ and \textup{#3(#1)#4--#5(#2)#6}}{, \textup{#3(#1)#4--#5(#2)#6}}{, and \textup{#3(#1)#4--#5(#2)#6}}

\crefdefaultlabelformat{#2\textup{#1}#3}

\apptocmd{\sloppy}{\hbadness 10000\relax}{}{}

\newcommand{\triplenorm}[1]{
  \vert\kern-0.25ex\vert\kern-0.25ex\vert #1
  \vert\kern-0.25ex\vert\kern-0.25ex\vert
}

\makeatletter
\def\spher@harm#1{%
  \vbox{\offinterlineskip
    \halign{&\hb@xt@2\p@{\hss$##$\hss}\cr
      #1\crcr
    }%
    \vskip-0.36ex
  }%
}
\def\gshone{\spher@harm{.}}
\def\gshtwo{\spher@harm{.&.}}
\def\gshthree{\spher@harm{.&.&.}}
\let\gsh\spher@harm
\makeatother

\newtheorem{proposition}{Proposition}[section]

\newtheorem{definition}{Definition}[section]

\newtheorem{lemma}{Lemma}[section]
\newtheorem{remark}{Remark}[section]
\newtheorem{theorem}{Theorem}[section]

\numberwithin{equation}{section}

\begin{document}
\title{Convergence of finite element approximations for the one-dimensional stochastic Burgers equation with additive trace-class noise}


\author[1]{Binjie Li\thanks{libinjie@scu.edu.cn}}
\author[1]{Xiaoping Xie\thanks{xpxie@scu.edu.cn}}
\author[2]{Qin Zhou\thanks{zqmath@cwnu.edu.cn}}
\affil[1]{School of Mathematics, Sichuan University, Chengdu 610064, China}
\affil[2]{School of Mathematics, China West Normal University, Nanchong 637002, China}

\date{}
\maketitle

\begin{abstract}
 This paper investigates finite element approximations of the one-dimensional viscous
  stochastic Burgers equation with additive trace-class noise.
 For the $P_2$ finite element spatial semi-discretization,
 we derive strong error estimates that are optimal with respect to regularity in
 \(L^p([0,T] \times \Omega;H_{D}^{\alpha,q})\) for \(p,q\in[2,\infty)\) and \(\alpha\in[-1,0]\),
 as well as an almost regularity-optimal estimate in \(L^p(\Omega;C([0,T];L^\infty(\mathcal{O})))\).
 Furthermore, we derive weak error estimates for moments of both terminal $L^q$-norms and
 space-time $L^p(0,T;L^q(\mathcal{O}))$-norms, with weak convergence rates (nearly) twice the corresponding strong ones.
 For the fully discrete scheme,
 which combines the \(P_2\) finite element method in space with a drift-implicit Euler--Maruyama scheme in time,
 we establish a strong temporal convergence rate of order \(\tau^{1/2-\varepsilon}\) in a discrete analogue of
 \(L^p(\Omega;C([0,T];L^\infty(\mathcal{O})))\), under the condition \(\tau \leqslant h^2\).
 Numerical experiments are presented to illustrate the theoretical convergence rates.
\end{abstract}


\section{Introduction}
Let $\mathcal{O} := (0,1)$ and fix a time horizon $T > 0$. We consider the one-dimensional viscous stochastic
Burgers equation driven by additive trace-class noise, taking values in $L^2(\mathcal{O})$, given by
\begin{equation} \label{eq:burgers}
  \mathrm{d}u(t) = \Bigl( \partial_x^2 u(t) - \frac{1}{2} \partial_x (u^2(t)) \Bigr)\, \mathrm{d}t + Q\,\mathrm{d}W(t), \quad t \in [0,T],
\end{equation}
subject to homogeneous Dirichlet boundary conditions $u(t,0) = u(t,1) = 0$ for $t \in (0,T]$, and the initial condition $u(0) = u_0$. 
Here, $u_0$ is a given initial datum, $(W(t))_{t \geqslant 0}$ is an $\ell^2$-cylindrical Wiener process,
and $Q$ is a $\gamma$-radonifying operator from $\ell^2$ to $L^2(\mathcal{O})$ to be specified in \cref{sec:regularity}.

Originally introduced to model turbulence and nonequilibrium phenomena \cite{bec2007burgers,bertini1994stochastic},
the stochastic Burgers equation has become a central example in stochastic analysis.
Its study has driven progress in regularity theory for nonlinear SPDEs
\cite{daPratoDebusscheTemam1994,daprato1995stochastic,gyongy1999stochasticburgers},
invariant measures \cite{e2000invariant}, ergodicity \cite{goldys2005exponential},
and the modern frameworks of regularity structures \cite{Hairer2014theory} and
paracontrolled distributions \cite{Gubinelli2017kpz}.

Numerical approximations of the stochastic Burgers equation have been widely studied.
Hairer and Voss \cite{Hairer2011approximations} noted that different finite-difference schemes converge
to distinct limiting processes.
Bl\"omker and Jentzen \cite{Blomker2013} established pathwise convergence of semi-discrete spectral Galerkin
approximations in the space-time maximum norm at a rate of almost $1/2$ in the case of additive space-time white noise.
This result was subsequently extended to a fully discrete scheme in \cite{Blomker2013full}, wherein a pathwise temporal
convergence rate of order $1/4$ was derived.
Regarding the semi-discrete spectral Galerkin approximations in the case of
the multiplicative trace-class noise, Hutzenthaler and Jentzen
\cite{Hutzenthaler2020perturbation} provided strong convergence rates.
For the fully discrete spectral Galerkin approximations, Jentzen et al.~\cite{JentzenSalimovaWelti2019} proved the
strong convergence of an explicit scheme under additive space-time white noise, albeit without providing explicit rates.
Hutzenthaler et al.~\cite{Hutzenthaler2020strong,HutzenthalerLink2022strong} achieved a significant
advancement by establishing the first strong convergence rates for fully discrete approximations under
additive and multiplicative trace-class noise,
utilizing a spectral Galerkin method in space coupled with a tamed exponential Euler method in time.
More recently, Br\'ehier et al.~\cite{BrehierCoxMillet2026} derived weak error estimates for the spectral
Galerkin discretization of \cref{eq:burgers} in the $L^2$-framework.
Additionally, Wang and Cao \cite{Wang2025strong} established the strong convergence of a fully
discrete scheme for the stochastic Burgers equation driven by an additive cylindrical fractional Brownian motion,
although no explicit strong convergence rate was obtained. We also refer to \cite{Li2027strong,Wang2025numerical} for related studies on the stochastic Burgers--Huxley equation.

Despite these advances, a systematic treatment of both the strong and weak errors of finite element
approximations for the stochastic Burgers equation is still lacking.
Moreover, existing numerical analyses that provide explicit strong convergence rates are essentially
restricted to the \(L^2\)-framework;
to the best of our knowledge, no strong convergence rates have been derived in the space-time maximum norm,
more precisely, in \(L^p(\Omega; C([0,T]; L^\infty(\mathcal{O})))\) or its discrete analogue.
In this paper, we partially fill these gaps with the following contributions:

\begin{itemize}
\item For the $P_2$ finite element semi-discretization of \cref{eq:burgers},
we establish a regularity-optimal strong convergence rate of order $h^{1-\alpha}$ in
$L^p([0,T] \times \Omega;H_{D}^{\alpha,q})$, and an almost regularity-optimal strong
convergence rate of order $h^{1-\alpha-\varepsilon}$ in $L^p(\Omega;C([0,T];H_{D}^{\alpha,q}))$,
for all $p,q\in[2,\infty)$, $\alpha\in[-1,0]$, and $\varepsilon>0$;
see Subsection~\ref{subsec:functional-analytic-framework} for the definition of $H_{D}^{\alpha,q}$.
Moreover, we prove strong convergence in  
\[
L^{p}\bigl(\Omega;C([0,T];L^{\infty}(\mathcal{O}))\bigr),\qquad p\in[2,\infty),
\]
of order \(h^{1-\varepsilon}\) for arbitrary \(\varepsilon>0\); see \cref{thm:uh-strong} for precise statements.
These appear to be the first strong error estimates in the space-time maximum norm
for spatial semi-discretizations of the stochastic Burgers equation.
They extend previous results on pathwise convergence in  
\(C([0,T];L^{\infty}(\mathcal{O}))\)~\cite{Blomker2013} and
on strong error estimates in  
\(L^{\infty}(0,T;L^{p}(\Omega;L^{2}(\mathcal{O})))\)~\cite{Hutzenthaler2020perturbation},
both obtained for semi-discrete spectral Galerkin approximations.

\item For this semi-discretization, we establish the following weak error
  estimates, for all $2 \leqslant q \leqslant p < \infty$:
  \begin{align*}
    \Bigl| \mathbb{E}\bigl[\|u(T)\|_{L^q(\mathcal{O})}^{p}\bigr]
           - \mathbb{E}\bigl[\|u_h(T)\|_{L^q(\mathcal{O})}^{p}\bigr] \Bigr|
      & \leqslant C h^{2-\varepsilon}, \\
    \Bigl| \mathbb{E}\bigl[\|u\|_{L^p(0,T;L^q(\mathcal{O}))}^{p}\bigr]
           - \mathbb{E}\bigl[\|u_h\|_{L^p(0,T;L^q(\mathcal{O}))}^{p}\bigr] \Bigr|
      & \leqslant C h^{2},
  \end{align*}
  where $\varepsilon > 0$ can be chosen arbitrarily small.
  In both cases, the weak convergence order is (nearly) twice the
  corresponding strong convergence order; see \cref{thm:uh-weak} for precise statements.
  Our analysis employs a standard duality argument
 (cf.~\cite{Geissert2009rate}, \cite[Chapter~5]{kruse2014book},
 and \cite{AnderssonKruseLarsson2016Duality}),
 which crucially exploits the strong error estimates in negative-order Sobolev spaces established above.
  To the best of our knowledge, these are the first weak error estimates for finite element approximations of
  the stochastic Burgers equation and, more generally,
  the first weak error analysis of spatial semi-discretizations in the $L^q$-setting.
  In contrast, the weak error analysis of semi-discrete spectral Galerkin approximations
  in \cite{BrehierCoxMillet2026} is based on Kolmogorov equations and carried out in the $L^2$-setting.
  Within the broader numerical analysis literature for SPDEs, related results include
  \cite{Hefter2016weak}, where essentially sharp weak convergence rates for
  exponential Euler approximations of semilinear SPDEs are established in the UMD Banach space setting,
  and \cite{Brehier2018weak}, where weak error estimates for path-dependent functionals are
  studied in the $L^2$-setting.

\item For the fully discrete approximation of \cref{eq:burgers} based on the $P_2$
finite element method in space and a drift-implicit Euler--Maruyama scheme in time, we establish strong error estimates.
Under the condition $\tau \leqslant h^{2}$, we prove, for all $p\in[2,\infty)$,
\[
  \Bigl(\mathbb{E}\Bigl[\max_{1 \leqslant j \leqslant J}
    \|u_h(t_j) - U_j\|_{L^{\infty}(\mathcal{O})}^{p}\Bigr]\Bigr)^{1/p}
    \leqslant C\,\tau^{1/2-\varepsilon},
\]
where $U_j$ denotes the fully discrete approximation at $t_j=j\tau$, $\varepsilon>0$ may be chosen arbitrarily small,
and $C$ is independent of $h$ and $\tau$; see \cref{thm:uh-U} for the precise statement.
These results strengthen the convergence topology in
\cite{Hutzenthaler2020strong,HutzenthalerLink2022strong}, 
where strong convergence rates were obtained in the $L^{\infty}(0,T;L^{p}(\Omega;L^{2}(\mathcal{O})))$
norm for spectral Galerkin approximations in space coupled with tamed exponential Euler schemes in time.
To the best of our knowledge, the present estimates are the first strong error estimates in the
discrete space-time maximum norm for fully discrete approximations of the stochastic Burgers equation.
For comparison, even for one-dimensional stochastic parabolic equations with Lipschitz
nonlinearities \cite{Anton2020fully,Gyongy1998latticeI,Gyongy1999latticeII},
fully discrete strong error estimates have been established only in weaker norms such as
$\sup_{x,t}\|\cdot\|_{L^{p}(\Omega)}$.
\end{itemize}

From a methodological standpoint in numerical analysis, one main novelty of this work, compared with the existing
literature, is that we combine the classical techniques of the \(L^2\)-framework for the numerical analysis of
SPDEs with the discrete stochastic maximal \(L^p\)-regularity estimates \cite{li2025stability, LiZhouLp2026} to
present, for the first time, an \(L^q\)-framework for the numerical analysis of finite element approximations of
the stochastic Burgers equation.
Beyond advancing the numerical analysis of the stochastic Burgers equation,
the analytical framework developed herein extends naturally to other stochastic 
parabolic equations, such as the stochastic Allen-Cahn equation.


Another main novelty of this work is the derivation of strong convergence rates for a fully discrete finite element method
in the discrete space-time maximum norm. In the numerical analysis of the stochastic Burgers equation,
as well as the stochastic Navier-Stokes equations, the convection term's dependence on the solution gradient and quadratic
growth in the solution itself render the standard discrete Gronwall argument inapplicable.
To circumvent this for the two-dimensional stochastic Navier-Stokes equations with periodic boundary conditions,
Bessaih and Millet \cite{Bessaih2022spacetime} combined exponential moment estimates with a localization technique to
extend the discrete Gronwall argument, establishing strong convergence rates for fully discrete finite element approximations.
A key ingredient in their proof is the identity
\[
\int_{D} \bigl((u\cdot\nabla)u\bigr)\cdot (Au)\,\mathrm{d}x = 0 \qquad \text{for all } u\in\operatorname{Dom}(A),
\]
where $A$ represents the $L^2(D;\mathbb{R}^2)$-realization of the Stokes operator subject to periodic boundary
conditions on the torus $D$. This identity ensures that, under appropriate assumptions,
\[
\mathbb{E} \biggl[ \exp\!\Big(\kappa\sup_{t \in [0,T]} \|\nabla u(t)\|_{L^2(D;\mathbb{R}^{2\times2})}^2\Big) \biggr] < \infty
\]
holds for some $\kappa > 0$. However, this exponential moment bound 
is not available from the known results for the stochastic Burgers equation.
Consequently, the methodology of Bessaih and Millet \cite{Bessaih2022spacetime} cannot be directly applied to the stochastic Burgers equation.

To the best of our knowledge, explicit strong convergence rates for fully discrete approximations of the stochastic Burgers equation have been established only in the aforementioned studies \cite{Hutzenthaler2020strong,HutzenthalerLink2022strong}.
Specifically, \cite{Hutzenthaler2020strong} integrates the classical Alekseev--Gr\"obner formula with uniform exponential moment bounds under additive trace-class noise, whereas \cite{HutzenthalerLink2022strong} couples the perturbation theory of \cite{Hutzenthaler2020perturbation} with analogous estimates for multiplicative trace-class noise.
Both works utilize a spectral Galerkin spatial discretization combined with a tamed exponential Euler scheme.
In contrast, the present paper develops a direct approach to deriving strong convergence rates for a fully discrete finite element method based on the implicit Euler--Maruyama scheme.
Our strategy hinges on lifting the discrete numerical solution to a continuous-time process and applying a Gronwall argument
to the error between this lifted process and the spatially semi-discrete solution.
The analysis further exploits exponential moment bounds for the numerical approximations;
crucially, the sharp stability estimates furnished by the discrete stochastic maximal $L^p$-regularity
form the cornerstone of the proof.

The remainder of the paper is organized as follows.
Section~\ref{sec:pre} establishes the functional framework (Sobolev and interpolation spaces, Dirichlet Laplacian and its analytic semigroup)
and the probabilistic setting ($\ell^2$-cylindrical Wiener process, $\gamma$-radonifying operators and stochastic integral estimates).
Section~\ref{sec:regularity} investigates the regularity of the mild solution to the stochastic Burgers equation.
Section~\ref{sec:spatial-semi-discretization} analyzes the strong and weak convergence of the spatial semi-discretization,
while Section~\ref{sec:full-discr} derives the strong convergence rate for the fully discrete scheme.
Section~\ref{sec:numerical} presents numerical experiments that verifies the theoretical results,
and Section~\ref{sec:concluding} offers concluding remarks.

\section{Preliminaries}
\label{sec:pre}

\subsection{Functional analytic framework}
\label{subsec:functional-analytic-framework}

Unless otherwise specified, all Banach spaces in this paper are assumed to be real.
We denote by $\mathcal{L}(E_1, E_2)$ the Banach space of bounded linear operators from $E_1$ to $E_2$, endowed with the operator norm.
The symbol $I$ stands for the identity operator. For Banach spaces $E_1$ and $E_2$,
the notation $E_1 \hookrightarrow E_2$ indicates that $E_1$ is continuously embedded into $E_2$.
For $\theta \in (0,1)$ and $p \in (1,\infty)$, let $(\cdot, \cdot)_{\theta,p}$ denote the real interpolation space defined by the $K$-method (see \cite[Chapter~1]{Lunardi2018}).
Moreover, given a couple $(X_1,X_2)$ of real Banach spaces with embedding $X_2 \hookrightarrow X_1$,
for each $\theta\in(0,1)$, the complex interpolation space $[X_1,X_2]_\theta$
is defined as the space of real parts of elements in the standard complex interpolation space
$[X_1^{\mathbb{C}},X_2^{\mathbb{C}}]_\theta$ (cf.~\cite[Chapter~2]{Lunardi2018}),
endowed with the norm inherited from $[X_1^{\mathbb{C}},X_2^{\mathbb{C}}]_\theta$.
Here, $X_i^{\mathbb{C}}$ denotes the complexification of $X_i$ for $i=1,2$ (see \cite[Section~B.4]{HytonenWeis2016}).

Fix $T > 0$. Given a Banach space $E$ and $p \in [1,\infty]$, we denote by $L^p(0,T;E)$ the Bochner space
(see \cite[Definition~1.2.15]{HytonenWeis2016}), and by $C([0,T];E)$ the Banach space of continuous functions
$v \colon [0,T] \to E$ endowed with the norm $\|v\|_{C([0,T];E)} := \sup_{t \in [0,T]} \|v(t)\|_E$.

Let $\mathcal{O} := (0,1)$. For $q \in [1,\infty]$, let $L^q(\mathcal{O})$ denote the standard Lebesgue space,
and let $W^{1,q}(\mathcal{O})$, $W_0^{1,q}(\mathcal{O})$,
and $W^{2,q}(\mathcal{O})$ denote the standard Sobolev spaces (see, e.g., \cite[Chapter~7]{Gilbarg2001}).
For $q \in (1,\infty)$, we define $A_q$ as the realization of the operator $\partial_x^2$ in $L^q(\mathcal{O})$ subject to homogeneous Dirichlet boundary conditions,
with domain $D(A_q) = W_0^{1,q}(\mathcal{O}) \cap W^{2,q}(\mathcal{O})$. It is standard that $-A_q$ is a
densely defined sectorial operator with a bounded inverse \cite[Theorem~2.12]{Yagi2010} and admits a
bounded $H^\infty$-calculus \cite[Theorem~2]{Duong1996functional}.

For $\theta \geqslant 0$ and $q \in (1,\infty)$, we define $H_{D}^{\theta,q}$ as the domain of
the fractional power $(-A_q)^{\theta/2}$,
equipped with the norm $\|v\|_{H_{D}^{\theta,q}} := \| (-A_q)^{\theta/2}v \|_{L^q(\mathcal{O})}$,
where $(-A_q)^{\theta/2} \colon H_{D}^{\theta,q} \to L^q(\mathcal{O})$ is an isometric isomorphism;
see \cite[Theorem~15.2.5]{HytonenNeervenVeraarWeis2023Vol3}
for the related theory of fractional powers of sectorial operators.
In particular, we have $H_{D}^{2,q} = W_0^{1,q}(\mathcal{O}) \cap W^{2,q}(\mathcal{O})$
and $H_{D}^{1,q} = W_0^{1,q}(\mathcal{O})$, with equivalent norms (cf.~\cite[Theorem~16.15]{Yagi2010}).

For $\theta \in [-2,0)$ and $q \in (1,\infty)$, let $H_{D}^{\theta,q}$ denote the dual space of $H_{D}^{-\theta,q'}$,
where $q' := q/(q-1)$ denotes the conjugate exponent of $q$.
Each $f \in L^q(\mathcal{O})$ is identified as an element of $H_{D}^{\theta,q}$ by
\[
  \langle f, v \rangle_{H_{D}^{\theta,q},H_{D}^{-\theta,q'}} :=
  \int_{\mathcal{O}} f v \, \mathrm{d}x,
  \quad \forall v \in H_{D}^{-\theta,q'}.
\]
With this identification, $L^q(\mathcal{O})$ is dense in $H_{D}^{\theta,q}$.
For all $-2 \leqslant \theta_1 < \theta_2 \leqslant 2$, $ s \in (0,1)$, and $q \in (1,\infty)$,
the interpolation identity
\begin{equation}
  \label{eq:interpolation-identity}
  [H_{D}^{\theta_1,q}, H_{D}^{\theta_2,q}]_{s} = H_{D}^{(1-s)\theta_1 + s \theta_2,q}
\end{equation}
holds with equivalent norms (see \cite[Theorem~1.5.4 of Chapter V]{Amann1995linear}).

Fix $\theta \in [-2,0)$ and $q \in (1,\infty)$.
Define $ A_{q,\theta} \colon H_{D}^{\theta+2,q} \to H_{D}^{\theta,q}$ by
\[
\langle A_{q,\theta}v, w \rangle_{H_{D}^{\theta,q}, H_{D}^{-\theta,q'}}
:= -\int_{\mathcal{O}} \bigl((-A_q)^{\theta/2+1}v\bigr) \bigl((-A_{q'})^{-\theta/2}w\bigr) \, \mathrm{d}x
\]
for $v \in H_{D}^{\theta+2,q}$ and $w \in H_{D}^{-\theta,q'}$.
Since $(-A_q)^{\theta/2+1}\colon H_{D}^{\theta+2,q}\to L^q(\mathcal{O})$ and
$(-A_{q'})^{-\theta/2}\colon H_{D}^{-\theta,q'}\to L^{q'}(\mathcal{O})$ are isometric
isomorphisms (see \cite[Theorem~15.2.5]{HytonenNeervenVeraarWeis2023Vol3}),
the operator $A_{q,\theta} \colon  H_{D}^{\theta+2,q}\to H_{D}^{\theta,q}$ is
well defined and an isometric isomorphism.
The operator $A_{q,\theta}$ is consistent with $A_q$ in the sense that
\[
  A_{q,\theta} v = A_q v
  \quad \text{in } H_{D}^{\theta,q}
  \qquad \text{for all } v \in H_{D}^{2,q}.
\]
Furthermore, by \cite[Theorem~2.1.3 of Chapter~V]{Amann1995linear},
the analytic semigroups $S_q(t)$ and $S_{q,\theta}(t)$ generated by $A_q$ and $A_{q,\theta}$,
respectively, satisfy the consistency relation $S_{q,\theta}(t)v=S_q(t)v$ in $H_D^{\theta,q}$
for all $v\in L^q(\mathcal{O})$ and $t\geqslant0$.
Hereafter, whenever no confusion can arise,
we write $A$ for either $A_q$ or $A_{q,\theta}$ and $S(t)$ for either $S_q(t)$ or $S_{q,\theta}(t)$.

For $g \in L^{1}(0,T;H_{D}^{\theta,q})$ with $\theta \in [-2,0]$ and $q \in (1,\infty)$, we define the convolution
\[
  (S \ast g)(t) := \int_{0}^{t} S(t-s)\, g(s)\, \mathrm{d}s,
  \qquad t \in [0,T].
\]
The following lemma collects some basic properties of the semigroup $S$ and of the
convolution operator $g \mapsto S \ast g$ that will be used in the sequel.

\begin{lemma}
  \label{lem:parabolic-regu}
  The following assertions hold:
  \begin{itemize}
    \item[\textup{(i)}] For any $v \in H_{D}^{\theta,q}$ with
    $\theta \in [-2,2]$ and $ q \in (1,\infty)$,
      \[
        \|S(\cdot)v\|_{C([0,T];H_{D}^{\theta,q})}
        \leqslant C \|v\|_{H_{D}^{\theta,q}},
      \]
      where $ C = C(q,\mathcal{O})>0$.
    \item[\textup{(ii)}] Let $\theta_1 \in [-2,0]$,
    $\theta_2 \in [\theta_1, \theta_1+2)$, and $q \in (1,\infty)$.
      For any $g \in L^\infty(0,T;H_{D}^{\theta_1,q})$,
      \[
        \|S \ast g\|_{C([0,T];H_{D}^{\theta_2,q})}
        \leqslant C \|g\|_{L^\infty(0,T;H_{D}^{\theta_1,q})},
      \]
      where $C = C(\theta_2-\theta_1,q,T,\mathcal{O})>0$.
    \item[\textup{(iii)}]
    For any $v \in (H_{D}^{\theta,2}, H_{D}^{\theta+2,2})_{1-1/p,p}$
    and $g \in L^p(0,T;H_{D}^{\theta,2})$ with $\theta \in [-2,0]$ and $p \in (1,\infty)$,
    \[
      \|S(\cdot)v + S \ast g\|_{L^p(0,T;H_{D}^{\theta+2,2})} \leqslant
      C \big(
        \|v\|_{(H_{D}^{\theta,2}, H_{D}^{\theta+2,2})_{1-1/p,p}}
        + \|g\|_{L^p(0,T;H_{D}^{\theta,2})}
      \big),
    \]
    where $C = C(p,\mathcal{O})>0$.
  \end{itemize}
\end{lemma}

\begin{proof}
  Assertions \textup{(i)} and \textup{(ii)} are classical results in the theory of
  analytic semigroups (see, e.g., \cite[Proposition~4.2.1 and Corollary~4.2.4]{Lunardi1995analytic}), while
  assertion \textup{(iii)} is a standard result of the maximal $L^{p}$-regularity theory
  (see \cite[p.~114]{Lunardi2018} and \cite[Theorem~3.3]{Neerven2012b}).
\end{proof}

\subsection{Stochastic framework}
Let $(\Omega, \mathcal{F}, \mathbb{F} = (\mathcal{F}_t)_{t \geqslant 0}, \mathbb{P})$ be a
fixed filtered probability space satisfying the usual conditions.
Let $(W_n)_{n \geqslant 1}$ be a sequence of mutually independent real-valued $\mathbb{F}$-Brownian motions,
and let $\ell^2$ denote the Hilbert space of square-summable real sequences. The associated cylindrical Wiener process on $\ell^2$,
denoted by $(W(t))_{t \geqslant 0}$, is the family of bounded linear operators $W(t) \colon \ell^2 \to L^2(\Omega, \mathcal{F}_t, \mathbb{P})$ defined, for each $t \geqslant 0$, by
\[
  W(t)l := \sum_{n=1}^\infty l_n W_n(t), \quad l = (l_n)_{n \geqslant 1} \in \ell^2,
\]
where the series converges in $L^2(\Omega, \mathcal{F}_t, \mathbb{P})$.
Throughout this paper, the stochastic integration with respect to $W$ is understood in the It\^o sense 
as defined in \cite{Neerven2007}.

For any separable Banach space $E$, we denote by $C_{\mathbb{F}}([0,T];E)$
the space of $\mathbb{F}$-adapted $E$-valued processes with $\mathbb{P}$-a.s.~continuous
paths, where indistinguishable processes are identified.
As usual, each $f \in C_{\mathbb{F}}([0,T];E)$ is regarded as a
strongly $\mathcal{F}_T$-measurable random variable with values in
$C([0,T];E)$ (see \cite[Proposition~3.18]{DaPrato2014}).

For any Banach space $E$ and $p \in [1,\infty)$, 
the Bochner space $L^p(\Omega; E)$ consists of all (equivalence classes of)
strongly $\mathcal{F}$-measurable functions $v \colon \Omega \to E$ satisfying $\mathbb{E}[\|v\|_E^p] < \infty$,
endowed with the norm $\|v\|_{L^p(\Omega;E)} := \big(\mathbb{E}[\|v\|_E^p]\big)^{1/p}$. 
Analogously, $L^p([0,T] \times \Omega; E)$ consists of (equivalence classes of) strongly
$\mathcal{B}([0,T]) \otimes \mathcal{F}$-measurable functions $v \colon [0,T] \times \Omega \to E$ such that 
\[
\mathbb{E}\left[ \int_0^T \|v(t,\cdot)\|_E^p \, \mathrm{d}t \right] < \infty,
\]
endowed with the corresponding canonical norm. 
We denote by $L^p_{\mathbb{F}}([0,T] \times \Omega; E)$ the subspace of $L^p([0,T] \times \Omega; E)$
consisting of all elements that admit an $\mathbb{F}$-progressively measurable representative.

For a Banach space $E$, we denote by $\gamma(\ell^2,E)$ the space of $\gamma$-radonifying operators
from $\ell^2$ to $E$ \cite[Definition~9.1.4]{HytonenWeis2017}.
This space satisfies the ideal property \cite[Proposition~9.1.10]{HytonenWeis2017}: for any Banach space $F$, any $U \in \gamma(\ell^2,E)$, and any $B \in \mathcal{L}(E,F)$,
\begin{equation}
  \label{eq:ideal}
  \|BU\|_{\gamma(\ell^2,F)} \leqslant \|B\|_{\mathcal{L}(E,F)} \|U\|_{\gamma(\ell^2,E)}.
\end{equation}
Furthermore, if $E$ is a Hilbert space, $\gamma(\ell^2,E)$ is isometrically isomorphic to the space of Hilbert--Schmidt operators from $\ell^2$ to $E$ \cite[Proposition~9.1.9]{HytonenWeis2017}.

We recall the following fundamental estimate for stochastic integrals with respect to the
$\ell^2$-cylindrical Wiener process $W$ (see \cite[Equation~(2.1)]{Neerven2012}).
\begin{proposition}
  \label{prop:stoch-int}
  For any $p \in (1,\infty)$, $q \in [2,\infty)$, $t \in (0,\infty)$, and any
  $\mathbb{F}$-progressively measurable process $g\colon [0,t] \times \Omega \to \gamma(\ell^2,L^q(\mathcal{O}))$,
  there exists a constant $C = C(p,q,\mathcal{O}) > 0$, such that
  \[
    \mathbb{E} \left[ \Big\| \int_0^t g(s) \, \mathrm{d}W(s) \Big\|_{L^q(\mathcal{O})}^p \right]
    \leqslant C \, \mathbb{E} \left[ \Big( \int_0^t \|g(s)\|_{\gamma(\ell^2,L^q(\mathcal{O}))}^2 \, \mathrm{d}s \Big)^{p/2} \right],
  \]
  provided the right-hand side is finite.
\end{proposition}

\section{Regularity}
\label{sec:regularity}
Throughout this paper, $Q$ denotes the $\gamma$-radonifying operator from \(\ell^2\) to \(L^2(\mathcal{O})\) defined by
\begin{equation}
  \label{eq:Q-def}
  Ql := \sum_{n=1}^\infty \sqrt{\lambda_n} \, l_n e_n, \quad l = (l_n)_{n \geqslant 1} \in \ell^2,
\end{equation}
where $(\lambda_n)_{n \geqslant 1}$ is a sequence of nonnegative coefficients satisfying
$\sum_{n=1}^\infty \lambda_n < \infty$,
and $(e_n)_{n \geqslant 1}$ is the orthonormal basis of $L^2(\mathcal{O})$
given by $e_n(x) = \sqrt{2}\sin(n\pi x)$ for \(x \in \mathcal{O}\).
Since $\sup_{x \in \mathcal{O}} |e_n(x)| = \sqrt{2}$ for all $n \geqslant 1$,
it follows from the summability of $(\lambda_n)_{n \geqslant 1}$ and \cite[Lemma~2.1]{Neerven2008} that
\begin{equation}
  \label{eq:Q-property}
  Q \in \gamma(\ell^2,L^q(\mathcal{O})) \quad\text{for all } q \in [1,\infty).
\end{equation}

We introduce the stochastic convolution
\begin{equation}\label{eq:G-def}
  G(t) := \int_0^t S(t-s)\, Q \, \mathrm{d}W(s), \qquad t \in [0,T],
\end{equation}
where $S(\cdot)$ denotes the analytic semigroup generated by the Dirichlet
Laplacian $A$ (see Subsection~\ref{subsec:functional-analytic-framework}).
In view of \cref{eq:Q-property},
\cite[Theorem~1.2]{Neerven2012} guarantees that $G$ admits a continuous modification
satisfying
\begin{equation}\label{eq:G-max}
  G \in L^{p}\Bigl(\Omega;\, C\bigl([0,T];\, (L^{q}(\mathcal{O}), H_{D}^{2,q})_{\frac{1}{2}-\frac{1}{p},\,p}\bigr)\Bigr),
\end{equation}
for all $p \in (2,\infty)$ and $q \in [2,\infty)$.
Here and in the sequel, $G$ is always identified with this continuous modification.

Now fix any $p,q\in[2,\infty)$ and $\theta\in[0,1)$, and choose $p_*>\max\{p,\,2/(1-\theta)\}$.
Applying \eqref{eq:G-max} with $p_*$ in place of $p$ and using the embedding
$(L^q(\mathcal{O}),H_{D}^{2,q})_{1/2-1/p_*,p_*}\hookrightarrow H_{D}^{\theta,q}$ (cf.~\cite[Theorem~C.4.1]{HytonenWeis2016} and \cite[Proposition~1.4]{Lunardi2018}), we obtain $G\in L^{p_*}(\Omega;C([0,T];H_{D}^{\theta,q}))$. Since $L^{p_*}(\Omega)\hookrightarrow L^p(\Omega)$ (the probability space has finite measure), we conclude that
 \begin{equation}\label{eq:G-LpCHq}
 G\in L^p(\Omega;C([0,T];H_{D}^{\theta,q}))\quad\text{for all }p,q\in[2,\infty)\text{ and }\theta\in[0,1).
 \end{equation}
Furthermore, the stochastic maximal $L^p$-regularity estimate in \cite[Theorem~1.1]{Neerven2012},
in conjunction with \cref{eq:Q-property}, implies $G\in L^p([0,T] \times \Omega;H_{D}^{1,q})$ for
all $p\in(2,\infty)$ and $q\in[2,\infty)$.
The inclusion $L^p([0,T] \times \Omega;H_{D}^{1,q}) \subset L^2([0,T] \times \Omega;H_{D}^{1,q})$ for $p>2$
allows us to extend the result to the endpoint $p=2$, and consequently,
\begin{equation}\label{eq:G-LpH1q}
   G\in L^p([0,T] \times \Omega;H_{D}^{1,q})\quad\text{for all }p,q\in[2,\infty).
\end{equation}

\begin{definition}
  Let $u_0$ be a strongly $\mathcal{F}_0$-measurable random variable with values in $L^2(\mathcal{O})$.
  A process $u \in C_{\mathbb{F}}([0,T];L^2(\mathcal{O}))$
  is called a \emph{mild solution} to equation
  \eqref{eq:burgers} if there exists a $\mathbb{P}$-null set $N$ such that,
  for all $\omega \in \Omega \setminus N$ and $t \in [0,T]$,
  \begin{equation}
    \label{eq:mild-def}
    u(t) = S(t)u_0 - \frac{1}{2}  \big( S \ast (\partial_x (u^2)) \big)(t) + G(t)
    \quad \text{ in } L^2(\mathcal{O}).
  \end{equation}
  Here, $G$ denotes the continuous modification of the stochastic convolution defined in \eqref{eq:G-def},
  and $\partial_x$ represents the distributional derivative.
\end{definition}

The following proposition establishes the regularity of the mild solution that will be needed in the sequel.

\begin{proposition}
  \label{prop:u-regu}
  Let \(u_0 \in L^2(\mathcal{O})\) be deterministic. Then \cref{eq:burgers} admits a unique mild solution~\(u\), and there exists \(\kappa > 0\) such that
  \begin{equation}
    \label{eq:u-exp-moment}
    \mathbb{E}\Bigl[ \exp\bigl( \kappa \|u\|_{L^2(0,T;H_{D}^{1,2})}^2 \bigr) \Bigr] < \infty .
  \end{equation}
  Furthermore, the following additional regularity properties hold:
  \begin{enumerate}
    \item[\textup{(i)}] 
    If \(u_0 \in W_0^{1,\infty}(\mathcal{O})\), then for all \(p,q \in [2,\infty)\) and \(\theta \in (0,1)\),
    \begin{align}
      u   &\in L^p([0,T] \times \Omega; H_{D}^{1,q}), \label{eq:u-LpH1q} \\
      u   &\in L^p(\Omega; C([0,T]; H_{D}^{\theta,q})), \label{eq:u-LpCHq} \\
      \xi &\in L^p(\Omega; C([0,T]; H_{D}^{1,q})), \label{eq:xi-LpCH1q}
    \end{align}
    where \(\xi := u - G\) with \(G\) defined in \eqref{eq:G-def}.
    
    \item[\textup{(ii)}] 
    If \(u_0 \in W_0^{1,\infty}(\mathcal{O}) \cap (L^2(\mathcal{O}),H_{D}^{2,2})_{1-1/p^*,p^*}\)
    for some \(p^* \in (4,\infty)\), then
    \begin{equation}
      \label{eq:xi-LpLp*H2}
      \xi \in L^p(\Omega; L^{p^*}(0,T; H_{D}^{2,2}))
      \qquad\text{for all } p \in [2,\infty).
    \end{equation}
  \end{enumerate}
\end{proposition}

\begin{proof}
Existence and uniqueness of a mild solution to \cref{eq:burgers} are classical;
see, e.g., \cite[Theorem~5.10]{Jentzen2021spatial}. The exponential moment bound
\cref{eq:u-exp-moment} is a direct consequence of \cite[Proposition~2.3]{Daprato1998differentiability}.
Set $\xi := u - G$. Then, $\mathbb{P}$-a.s., $\xi$ satisfies
\begin{equation}
  \label{eq:xi-def}
  \xi(t) = S(t)u_0 - \frac{1}{2} S \ast \big( \partial_x(u^2) \big)(t)
  \quad \text{for all } t \in [0,T].
\end{equation}
We first observe that assertion \textup{(ii)} follows from \textup{(i)}.
Indeed, assume that $u_0 \in (L^2(\mathcal{O}), H_D^{2,2})_{1-1/p^*,p^*}$ for some $p^* \in (4,\infty)$.
Applying \cref{lem:parabolic-regu}\textup{(iii)}, we obtain, $\mathbb{P}$-a.s.,
\[
\|\xi\|_{L^{p^*}(0,T;H_{D}^{2,2})} \leqslant C \Big( \|u_0\|_{(L^2(\mathcal{O}),H_{D}^{2,2})_{1-1/p^*,p^*}}
+ \|\partial_x(u^2)\|_{L^{p^*}(0,T;L^2(\mathcal{O}))} \Big),
\]
where $C = C(p^*,\mathcal{O}) > 0$. Consequently, \cref{eq:xi-LpLp*H2} holds for all $p \in [2,\infty)$ provided that 
\[
\partial_x(u^2) \in L^p(\Omega;L^{p^*}(0,T;L^2(\mathcal{O}))) \quad \text{for all } p \in [2,\infty).
\]
The latter condition is a direct consequence of \cref{eq:u-LpH1q}, which is valid for all $p,q \in [2,\infty)$.
It therefore suffices to establish assertion \textup{(i)}.

The proof of assertion \textup{(i)} is divided into three steps. Throughout
the remainder of the proof, for brevity, we write
$\mathcal{X}_p^{\theta,q} := L^p \big( \Omega; C([0,T];H_{D}^{\theta,q}) \big)$.

\textbf{Step 1: Regularity of $u$ in $\mathcal{X}_p^{0,2}$.}
We prove that
\begin{equation} \label{eq:u-LpCL2}
  u \in \mathcal{X}_p^{0,2} \quad \text{for all } p \in [2,\infty).
\end{equation}
The argument is classical; we include a concise proof for completeness.

For any $\varepsilon \in (0,1/2)$, by the Sobolev embedding $H_{D}^{3/2+\varepsilon,2} \hookrightarrow W_0^{1,\infty}(\mathcal{O})$,
the distributional derivative $\partial_x$ induces a bounded linear
operator from $L^1(\mathcal{O})$ to $H_{D}^{-3/2-\varepsilon,2}$.
Since the Nemytskii operator $v \mapsto v^2$ maps $L^2(\mathcal{O})$ continuously into $L^1(\mathcal{O})$,
we deduce that
\begin{equation}
  \label{eq:partial-x-L1}
 \text{the map $v \mapsto \partial_x(v^2)$
is continuous from $L^2(\mathcal{O})$ to $H_{D}^{-3/2-\varepsilon,2}$
for all $\varepsilon \in (0,1/2)$}.
\end{equation}
Hence, from $u \in C_{\mathbb{F}}([0,T];L^2(\mathcal{O}))$ we
deduce that
\[
  \partial_x(u^2) \in C_{\mathbb{F}}([0,T];H_D^{-8/5,2}).
\]
Combined with the initial regularity $u_0 \in W_0^{1,\infty}(\mathcal{O}) \hookrightarrow H_{D}^{1/4,2}$, this
allows us to apply \cref{lem:parabolic-regu}(i,ii) and conclude that
\[
  S(\cdot)u_0 - \frac{1}{2} S \ast \bigl( \partial_x(u^2) \bigr)
  \in C_{\mathbb{F}}([0,T];H_D^{1/4,2}).
\]
Together with the mild formulation \cref{eq:G-LpCHq,eq:mild-def}, this yields
\[
  u \in C_{\mathbb{F}}([0,T];H_D^{1/4,2}).
\]

Next, the Sobolev embedding $H_{D}^{1/4,2} \hookrightarrow L^4(\mathcal{O})$ yields
$u \in C_{\mathbb{F}}([0,T];L^4(\mathcal{O}))$.
Since the Nemytskii operator $v \mapsto v^2$ maps $L^4(\mathcal{O})$ continuously into $L^2(\mathcal{O})$,
we infer that $u^2 \in C_{\mathbb{F}}([0,T];L^2(\mathcal{O}))$.
Hence $\partial_x(u^2) \in C_{\mathbb{F}}([0,T];H_{D}^{-1,2})$.
This, together with the initial condition $u_0 \in W_0^{1,\infty}(\mathcal{O}) \hookrightarrow (H_{D}^{-1,2}, H_{D}^{1,2})_{1/2,2}$,
enables us to use \cref{lem:parabolic-regu}(iii) to infer that 
$S(\cdot)u_0 - \frac{1}{2} S \ast \big( \partial_x(u^2) \big)
\in L^2(0,T;H_{D}^{1,2})$ $\mathbb{P}$-a.s.
Combining this with \cref{eq:G-LpH1q,eq:mild-def}, we obtain that
$u \in L^2(0,T;H_{D}^{1,2})$ $\mathbb{P}$-a.s.
Consequently, since $u \in C_{\mathbb{F}}([0,T];L^2(\mathcal{O}))$, we conclude that
\[
u \in C([0,T];L^2(\mathcal{O})) \cap L^2(0,T;H_{D}^{1,2}) \quad \mathbb{P} \text{-a.s.}
\]

With this regularity property in hand,
it is standard that (see, e.g., \cite[Proposition~4.4]{Neerven2012b}),
$\mathbb{P}$-a.s., for all \(t \in [0,T]\),
\[
  u(t) = u_0 + \int_0^t \Bigl( Au(s) - \frac{1}{2} \partial_x(u^2(s)) \Bigr) \, \mathrm{d}s
  + \int_0^t  Q \, \mathrm{d}W(s) \quad \text{in } H_{D}^{-1,2}.
\]
Applying Itô's formula (\cite[Lemma~A.5]{Agresti2025nonlinear})
and using the identity
\[
  \Bigl\langle Av - \frac{1}{2} \partial_x(v^2), \, v \Bigr\rangle_{H_{D}^{-1,2},H_{D}^{1,2}}
  = -\|v\|_{H_{D}^{1,2}}^2,
  \qquad v \in H_{D}^{1,2},
\]
where \(\langle \cdot, \cdot \rangle_{H_{D}^{-1,2},H_{D}^{1,2}}\) denotes the duality pairing
between \(H_{D}^{-1,2}\) and \(H_{D}^{1,2}\), we deduce that \(\mathbb{P}\)-a.s.~for all \(t \in [0,T]\),
\begin{align*}
  \|u(t)\|_{L^2(\mathcal{O})}^2
  &= \|u_0\|_{L^2(\mathcal{O})}^2
  - 2 \int_0^t \|u(s)\|_{H_{D}^{1,2}}^2 \, \mathrm{d}s \\
  &\quad + 2 \int_0^t \langle u(s), Q \rangle \, \mathrm{d}W(s)
  + t \|Q\|_{\gamma(\ell^2,L^2(\mathcal{O}))}^2.
\end{align*}
Here, $\langle u, Q\rangle$ is interpreted as an $\mathbb{F}$-adapted
$\gamma(\ell^2,\mathbb{R})$-valued continuous process with the estimate (see \cref{rem:uQ})  
\begin{equation}
  \label{eq:uQ-bound}
  \|\langle u, Q \rangle\|_{\gamma(\ell^2,\mathbb{R})} \leqslant
  \|u\|_{L^2(\mathcal{O})} \|Q\|_{\gamma(\ell^2,L^2(\mathcal{O}))}
  \quad \mathbb{P} \otimes \mathrm{d}t \text{-a.e.}
\end{equation}
Hence, dropping the non-positive drift term, we obtain, \(\mathbb{P}\)-a.s., for all \(t \in [0,T]\),
\[
  \|u(t)\|_{L^2(\mathcal{O})}^2
  \leqslant \|u_0\|_{L^2(\mathcal{O})}^2
  + 2 \int_0^t \langle u(s), Q \rangle \, \mathrm{d}W(s)
  + T\|Q\|_{\gamma(\ell^2,L^2(\mathcal{O}))}^2.
\]
For each \(R\in(1,\infty)\), define the stopping time
\[
t_R:=\inf\{t\in[0,T]:\|u(t)\|_{L^2(\mathcal{O})}>R\},
\]
with the convention \(\inf\varnothing=T\). Since \(u\in C_{\mathbb{F}}([0,T];L^2(\mathcal{O}))\), \(t_R\) is an \(\mathbb{F}\)-stopping time and \(\lim_{R\to\infty}\mathbb{P}(t_R=T)=1\).
For every \(p\in[2,\infty)\), a standard argument based on the Burkholder--Davis--Gundy inequality 
(see, e.g., the proof of \cite[Lemma~2.2]{LiuRoeckner2010SPDE}) yields
\[
\mathbb{E}\Bigl[\sup_{t\in[0,T]}\|u(t\wedge t_R)\|_{L^2(\mathcal{O})}^p\Bigr]
\leqslant C\Bigl(\|u_0\|_{L^2(\mathcal{O})}^p+\|Q\|_{\gamma(\ell^2,L^2(\mathcal{O}))}^p\Bigr),
\]
where \(C=C(p,T)>0\). Since \(t_R\uparrow T\) \(\mathbb{P}\)-a.s.,
letting $R \to \infty$ and using the monotone convergence theorem, we obtain
\[
\mathbb{E}\Bigl[\sup_{t\in[0,T]}\|u(t)\|_{L^2(\mathcal{O})}^p\Bigr]
=\mathbb{E}\Bigl[\lim_{R\to\infty}\sup_{t\in[0,T]}\|u(t\wedge t_R)\|_{L^2(\mathcal{O})}^p\Bigr]
\leqslant C\Bigl(\|u_0\|_{L^2(\mathcal{O})}^p+\|Q\|_{\gamma(\ell^2,L^2(\mathcal{O}))}^p\Bigr).
\]
The right-hand side is finite since
$u_0 \in W_0^{1,\infty}(\mathcal{O}) \hookrightarrow L^2(\mathcal{O})$ is deterministic
and $Q \in \gamma(\ell^2,L^2(\mathcal{O}))$ by \cref{eq:Q-property}.
Since $u$ has $\mathbb{P}$-a.s.~continuous paths in $L^2(\mathcal{O})$,
this establishes \cref{eq:u-LpCL2}.

\textbf{Step 2: Proof of \cref{eq:xi-LpCH1q}.}
We proceed by a bootstrap argument that uses the smoothing properties of the
analytic semigroup \(S\) to successively improve the spatial regularity of \(\xi\).

First, since $u_0 \in W_0^{1,\infty}(\mathcal{O}) \hookrightarrow H_{D}^{1,q}$ for every
$q \in (1,\infty)$, \cref{lem:parabolic-regu}(i) with $\theta=1$ yields
\begin{equation}
  \label{eq:Su0}
  S(\cdot)u_0 \in C([0,T];H_{D}^{1,q}) \quad\text{for all } q \in (1,\infty).
\end{equation}
By \cref{eq:u-LpCL2} (applied with $p$ replaced by $2p$), we infer that
$u \in \mathcal{X}_{2p}^{0,2}$ for all $p \in [2,\infty)$.
For any $\varepsilon \in (0,1/2)$, since
the nonlinear operator $v \mapsto \partial_x(v^2)$ maps $L^2(\mathcal{O})$ 
continuously into $H_{D}^{-3/2-\varepsilon/2,2}$ (see \cref{eq:partial-x-L1})
and is quadratically bounded, it follows that
\[
  \partial_x(u^2) \in \mathcal{X}_p^{-3/2-\varepsilon/2,2}
  \quad\text{for all } p \in [2,\infty) \text{ and } \varepsilon \in (0,1/2).
\]
An application of \cref{lem:parabolic-regu}(ii)
with $(\theta_1,\theta_2,q) = (-3/2-\varepsilon/2,1/2-\varepsilon,2)$ for $\varepsilon \in (0,1/2)$
yields
\[
  S \ast \bigl(\partial_x(u^2)\bigr)
  \in \mathcal{X}_p^{1/2-\varepsilon,2}
  \quad\text{for all } p \in [2,\infty) \text{ and } \varepsilon \in (0,1/2).
\]
In view of \cref{eq:xi-def} and \cref{eq:Su0}, we conclude that
\[
  \xi \in \mathcal{X}_p^{1/2-\varepsilon,2}
  \quad\text{for all } p \in [2,\infty) \text{ and } \varepsilon \in (0,1/2).
\]
By the Sobolev embedding theorem, it follows that
\begin{equation}
  \label{eq:xi-LpCH-pre}
  \xi \in \mathcal{X}_p^{0,q}
  \quad\text{for all } p,q \in [2,\infty).
\end{equation}

The spatial regularity of $u$ can now be improved.
Combining \cref{eq:G-LpCHq} (with $\theta = 0$),
\cref{eq:xi-LpCH-pre},
and the decomposition $u = \xi + G$,
we conclude that
\[
  u \in \mathcal{X}_p^{0,q} \quad\text{for all } p,q \in [2,\infty).
\]
Replacing $p$ by $2p$ and $q$ by $2q$ in the previous estimate, and
since the Nemytskii operator $v \mapsto v^2$ maps $L^{2q}(\mathcal{O})$
continuously into $L^{q}(\mathcal{O})$ and is quadratically bounded,
it follows that $u^2 \in \mathcal{X}_p^{0,q}$ for all $p,q \in [2,\infty)$.
The fact that $\partial_x \in \mathcal{L}(L^q(\mathcal{O}),H_{D}^{-1,q})$ then yields
\begin{equation}
  \label{eq:u2-bound}
  \partial_x(u^2) \in \mathcal{X}_p^{-1,q} \quad \text{for all } p,q \in [2,\infty).
\end{equation}
A further application of \cref{lem:parabolic-regu}(ii) with $(\theta_1,\theta_2) = (-1,4/5)$ gives
\[
  S \ast \big(\partial_x(u^2)\big) \in \mathcal{X}_p^{4/5,q}
  \quad \text{for all } p,q \in [2,\infty),
\]
which, combined with \cref{eq:Su0,eq:xi-def}, yields
\begin{equation}
  \label{eq:xi-LpCH-pre2}
    \xi \in \mathcal{X}_p^{4/5,q} \quad \text{for all } p,q\in [2,\infty).
\end{equation}

We now perform the final refinement.  
Fix $p,q \in [2,\infty)$. Combining \cref{eq:G-LpCHq} (with $p$ replaced by $2p$, and $\theta$ replaced by $4/5$),
\cref{eq:xi-LpCH-pre2} (with $p$ replaced by $2p$),
and the decomposition $u = \xi + G$, we deduce that
$u \in \mathcal{X}_{2p}^{4/5,q}$.
For any $v, w \in H_{D}^{4/5,q}$, the following estimates hold
(see \cite[Chapter~4]{Runst1996sobolev} and
the proof of \cref{lem:composition-embeddings} in Subsection~\ref{ssec:proof-uh-weak} for related techniques):
\[
\|v^2\|_{H_{D}^{3/5,q}} \leqslant C\|v\|_{H_{D}^{4/5,q}}^2,
\qquad
\|v^2 - w^2\|_{H_{D}^{3/5,q}} \leqslant C
\|v+w\|_{H_{D}^{4/5,q}} \|v-w\|_{H_{D}^{4/5,q}},
\]
where $C = C(q,\mathcal{O}) > 0$.
Consequently, the Nemytskii operator \(v\mapsto v^2\) is continuous from \(H_D^{4/5,q}\) to \(H_D^{3/5,q}\)
and quadratically bounded. Since \(u\in\mathcal{X}_{2p}^{4/5,q}\), it follows that
\[
u^2\in\mathcal{X}_p^{3/5,q}.
\]
Since \(\partial_x\) is bounded from \(H_D^{3/5,q}\) to \(H_D^{-2/5,q}\), we obtain
\begin{equation}
  \label{eq:partialxu2}
  \partial_x(u^2)\in\mathcal{X}_p^{-2/5,q}.
\end{equation}
Applying \cref{lem:parabolic-regu}(ii) with \((\theta_1,\theta_2)=(-2/5,1)\), we get
\[
S\ast\bigl(\partial_x(u^2)\bigr)\in\mathcal{X}_p^{1,q}.
\]
Since \(p,q\in[2,\infty)\) were arbitrary, this inclusion holds for all \(p,q\in[2,\infty)\). Combining this with \cref{eq:xi-def} and \cref{eq:Su0}, we infer that \cref{eq:xi-LpCH1q} holds for all \(p,q\in[2,\infty)\).

\textbf{Step 3: Proof of \cref{eq:u-LpH1q,eq:u-LpCHq}.}
Recalling the decomposition $u = \xi + G$, the desired regularity \cref{eq:u-LpH1q} for all $p,q \in [2,\infty)$
follows immediately from \cref{eq:G-LpH1q} and \cref{eq:xi-LpCH1q},
in conjunction with the embedding $ \mathcal{X}_p^{1,q} \hookrightarrow L^p([0,T] \times \Omega;H_{D}^{1,q})$.
Finally, the regularity property \cref{eq:u-LpCHq} for all $p,q \in [2,\infty)$ and $\theta \in (0,1)$ is deduced from \cref{eq:G-LpCHq} and \cref{eq:xi-LpCH1q} by utilizing the embedding $H_{D}^{1,q} \hookrightarrow H_{D}^{\theta,q}$.
This completes the proof of (i), and hence that of \cref{prop:u-regu}.
\end{proof}

\begin{remark} \label{rem:uQ}
  For any $v \in L^2(\mathcal{O})$, the mapping $w \mapsto \langle v, w \rangle$ defines
  a bounded linear functional on $L^2(\mathcal{O})$, where $\langle \cdot, \cdot \rangle$
  denotes the inner product in $L^2(\mathcal{O})$. Consequently, by the ideal property \cref{eq:ideal},
  we have $\langle v, Q \rangle \in \gamma(\ell^2,\mathbb{R})$ and
  \[
    \| \langle v, Q \rangle \|_{\gamma(\ell^2,\mathbb{R})} \leqslant \|v\|_{L^2(\mathcal{O})} \|Q\|_{\gamma(\ell^2,L^2(\mathcal{O}))}.
  \]
\end{remark}

\section{Spatial semi-discretization}
\label{sec:spatial-semi-discretization}
Let $\mathcal{K}_h$ be a uniform partition of the domain $\mathcal{O}$ with mesh size $h$, and define the finite element space
\[
X_h := \Bigl\{ v_h \in C(\overline{\mathcal{O}}): \, v_h|_K \in P_2(K) \ \forall\, K \in \mathcal{K}_h,\; v_h(0) = v_h(1) = 0 \Bigr\},
\]
where $P_2(K)$ denotes the space of polynomials on $K$ of total degree at most two.
The discrete Laplacian \(A_h \colon X_h \to X_h\) is defined through the variational formulation
\[
\int_{\mathcal{O}} (A_h u_h) v_h \, \mathrm{d}x
= -\int_{\mathcal{O}} \partial_x u_h \, \partial_x v_h \, \mathrm{d}x
\qquad \text{for all } u_h, v_h \in X_h.
\]
For \(\theta \in \mathbb{R}\) and \(q \in (1,\infty)\), we denote by \(H_{D,h}^{\theta,q}\) the space \(X_h\) equipped
with the discrete norm \(\|v_h\|_{H_{D,h}^{\theta,q}} := \|(-A_h)^{\theta/2}v_h\|_{L^q(\mathcal{O})}\) (\(v_h \in X_h\)).

Let $P_h$ be the $L^2(\mathcal{O})$-orthogonal projection onto $X_h$.
For any $q \in (1,\infty)$ with H\"older conjugate exponent $q'$, the operator $P_h$ extends by duality to a
linear operator from $H_{D}^{\theta,q}$, $\theta \in [-1,0]$, onto $X_h$, characterized by the relation
\begin{equation}
  \label{eq:Ph-extend}
  \int_{\mathcal{O}} (P_h f)\, v_h \,\mathrm{d}x
  \;=\; \langle f,\, v_h \rangle
  \qquad \text{for all } v_h \in X_h,
\end{equation}
where $\langle \cdot,\cdot \rangle$ denotes the duality pairing between $H_{D}^{\theta,q}$ and $H_{D}^{-\theta,q'}$.
A classical result by Douglas, Dupont, and Wahlbin~\cite{Douglas1975} establishes the uniform $L^q$-stability of this projection:
\begin{equation}
  \label{eq:Ph-stab0}
  \sup_{0<h\leqslant 1} \|P_h\|_{\mathcal{L}(L^q(\mathcal{O}), L^q(\mathcal{O}))} < \infty, \qquad q \in [1,\infty].
\end{equation}
In view of this uniform boundedness and the ideal property of $\gamma$-radonifying operators (cf.~\cref{eq:ideal}), property \cref{eq:Q-property} implies the uniform bound
\begin{equation} \label{eq:PhQ-bound}
\sup_{0 < h \leqslant 1} \|P_h Q\|_{\gamma(\ell^2, H_{D,h}^{0,q})} < \infty \qquad \text{for all } q \in [1,\infty).
\end{equation}

With these preliminaries, the spatial semi-discretization of \cref{eq:burgers} reads
\begin{equation}
  \label{eq:uh}
  \begin{cases}
    \mathrm{d}u_h(t) = \bigl(A_h u_h(t) - \frac{1}{2} P_h \partial_x(u_h^2(t))\bigr) \, \mathrm{d}t
    + P_h Q \, \mathrm{d}W(t), \quad t \in [0,T], \\[4pt]
    u_h(0) = P_h u_0.
  \end{cases}
\end{equation}

\begin{definition}
  Let $u_0$ be a strongly $\mathcal{F}_0$-measurable $L^2(\mathcal{O})$-valued random variable.
  A process $u_h \in C_{\mathbb{F}}([0,T];X_h)$
  is called a \emph{solution} to \cref{eq:uh} if
  there exists a $\mathbb{P}$-null set $N$ such that,
  for all $\omega \in \Omega \setminus N$ and $t \in [0,T]$,
  \begin{equation} \label{eq:uh-def}
    u_h(t) = P_h u_0 + \int_0^t \left( A_h u_h(s) - \frac{1}{2} P_h\partial_x\bigl( u_h^2(s) \bigr) \right) \, \mathrm{d}s
    + \int_0^t P_h Q \, \mathrm{d}W(s).
  \end{equation}
\end{definition}

\begin{proposition} \label{prop:uh-regu}
  Let the initial datum $u_0 \in L^2(\mathcal{O})$ be deterministic.
  Then, \cref{eq:uh} admits a unique solution $u_h$, and there exists $\kappa > 0$ such that
  \begin{equation}
    \label{eq:uh-exp-moment}
    \sup_{0 < h \leqslant 1} \mathbb{E}\left[ \exp\left( \kappa \|u_h\|_{L^2(0,T;H_{D}^{1,2})}^2 \right) \right] < \infty.
  \end{equation}
  Moreover, if $u_0 \in W_0^{1,\infty}(\mathcal{O})$, then for all $p, q \in [2,\infty)$ and $\theta \in (0,1)$,
  \begin{align}
    \sup_{0 < h \leqslant 1} \|u_h\|_{L^p([0,T] \times \Omega; H_{D}^{1,q})} &< \infty,
    \label{eq:uh-LpH1q} \\
    \sup_{0 < h \leqslant 1} \|u_h\|_{L^p(\Omega; C([0,T]; H_{D}^{\theta,q}))} &< \infty.
    \label{eq:uh-LpCHq}
  \end{align}
\end{proposition}

The main results of this section are the following strong and weak error estimates for the spatial semi-discretization \cref{eq:uh}.

\begin{theorem}[Strong error estimates]
  \label{thm:uh-strong}
  Let the initial datum $u_0$ be deterministic and satisfy
  \[
    u_0 \in W_0^{1,\infty}(\mathcal{O}) \cap (L^2(\mathcal{O}), H_{D}^{2,2})_{1-1/p^*,p^*}
  \]
  for some $p^* \in (4,\infty)$.
  Let $u$ and $u_h$ denote the mild solution to \cref{eq:burgers} and the solution to \cref{eq:uh}, respectively.
  Then, for all $\alpha \in [-1,0]$, $p, q \in [2,\infty)$ and $\varepsilon \in (0,1)$, the following estimates hold:
  \begin{align}
    \|u-u_h\|_{L^p([0,T] \times \Omega; \, H_{D}^{\alpha,q})} & \leqslant C_1 h^{1-\alpha}, \label{eq:u-uh-LpLq} \\
    \|u-u_h\|_{L^p(\Omega; \, C([0,T];H_{D}^{\alpha,q}))} & \leqslant C_2 h^{1-\alpha-\varepsilon}, \label{eq:u-uh-LpCLq} \\
    \|u-u_h\|_{L^p(\Omega; \, C([0,T];L^\infty(\mathcal{O})))} & \leqslant C_3 h^{1-\varepsilon}. \label{eq:u-uh-LpCLinf}
  \end{align}
  The constants $C_1, C_2, C_3 > 0$ are independent of the mesh size $h$.
  Specifically, $C_1$ depends on the domain $\mathcal{O}$, the terminal time $T$,
  the operator $Q$, the initial datum $u_0$, and the parameters $\alpha, p, q$,
  while $C_2$ and $C_3$ additionally depend on $\varepsilon$.
\end{theorem}

\begin{remark}
  \label{rem:uh-strong}
  Assume that the initial datum \(u_0\in W_0^{1,\infty}(\mathcal{O})\) is deterministic.
  Retracing the proof of \cref{thm:uh-strong} presented in Subsection~\ref{ssec:proof-uh-strong},
  we conclude that the error estimates \cref{eq:u-uh-LpLq,eq:u-uh-LpCLq} with \(\alpha = 0\)
  and \cref{eq:u-uh-LpCLinf} remain valid.
  Moreover, although the present article focuses on the \(P_2\) finite element method, 
  the strong error estimates \cref{eq:u-uh-LpLq,eq:u-uh-LpCLq} for \(\alpha=0\) and \cref{eq:u-uh-LpCLinf}
  extend directly to the \(P_1\) finite element method.
\end{remark}

\begin{remark}
  The additional factor $h^{-\varepsilon}$ in \cref{eq:u-uh-LpCLq} stems from the additive noise, which limits the spatial regularity of the stochastic convolution $G$.
  In view of \cref{eq:G-LpCHq}, for any $ p \in (2,\infty) $ and $ q \in [2,\infty) $, the inclusion $ u \in L_{\mathbb{F}}^p(\Omega; C([0,T]; H_{D}^{1,q})) $ does not hold; rather, one obtains $ u \in L_{\mathbb{F}}^p(\Omega; C([0,T]; H_{D}^{1-\varepsilon,q})) $ for arbitrarily small $\varepsilon>0$.
  This intrinsic regularity gap directly leads to the $\varepsilon$-loss in the strong error estimate.
\end{remark}

\begin{remark}
  To place our spatial error estimates in context, we draw a comparison with the spectral Galerkin method.
  Let $n \geqslant 1$ be an integer and define $\Lambda_n := \operatorname{span}\{\sin(k\pi x) : k = 1, \dots, n\}$.
  Denote by $\mathcal{P}_n$ the $L^2(\mathcal{O})$-orthogonal projection onto $\Lambda_n$, and let $u^{(n)}$ be the
  semi-discrete spectral Galerkin approximation of \cref{eq:burgers}, defined $\mathbb{P}$-a.s.\ by
  \[
    u^{(n)}(t) = \mathcal{P}_nS(t) u_0 -\frac{1}{2}\mathcal{P}_n\int_0^t S(t-s)\partial_x\bigl((u^{(n)}(s))^2\bigr)\,\mathrm{d}s
    + \mathcal{P}_n\int_0^t S(t-s) Q\,\mathrm{d}W(s)
    \quad\text{for all } t\in[0,T].
  \]
  For $u_0 \in H_{D}^{1,2}$, it follows from \cite[Corollary~3.11]{Hutzenthaler2020perturbation} that
  for any $ p \in (0,\infty) $ and $\varepsilon \in (0,1)$,
  \[
    \sup_{t\in[0,T]}\bigl\|u(t)-u^{(n)}(t)\bigr\|_{L^p(\Omega;L^2(\mathcal{O}))} \leqslant C n^{\varepsilon-1},
  \]
  where $C$ is a positive constant independent of $n$.
  By identifying the truncation level $n$ with the inverse mesh size $h^{-1}$—as both parametrize the spatial degrees
  of freedom—we observe that the rate $n^{\varepsilon-1} \sim h^{1-\varepsilon}$ precisely mirrors the $h^{1-\varepsilon}$
  factor in \cref{eq:u-uh-LpCLq} (taking $\alpha=0$ and $q=2$).
\end{remark}

\begin{theorem}[Weak error estimates]
  \label{thm:uh-weak}
  Under the hypotheses of \cref{thm:uh-strong},
  for any $2 \leqslant q \leqslant p < \infty$
  and $\varepsilon\in(0,1)$, we have
  \begin{align}
    \Bigl| \mathbb{E} \|u(T)\|_{L^q(\mathcal{O})}^p - \mathbb{E}\|u_h(T)\|_{L^q(\mathcal{O})}^p \Bigr|
    & \leqslant C_1 h^{2-\varepsilon}, \label{eq:weak-error-1} \\[4pt]
    \Bigl| \mathbb{E} \|u\|_{L^p(0,T;L^q(\mathcal{O}))}^p - \mathbb{E}\|u_h\|_{L^p(0,T;L^q(\mathcal{O}))}^p \Bigr|
    & \leqslant C_2 h^{2}, \label{eq:weak-error-2}
  \end{align}
  where the constants $C_1, C_2 > 0$ are independent of $h$ and depend on $\mathcal{O}, T, Q, u_0, p$, and $q$, with $C_1$ additionally depending on $\varepsilon$.
\end{theorem}

\begin{remark}
  The proof of the weak error estimates in \cref{thm:uh-weak} employs a standard duality argument 
  (cf.\ \cite{Geissert2009rate}, \cite[Chapter 5]{kruse2014book}, and \cite{AnderssonKruseLarsson2016Duality}
  for analogous arguments in Hilbert space settings), 
  using the strong error estimates in negative-order Sobolev spaces (specifically, \cref{eq:u-uh-LpLq,eq:u-uh-LpCLq})
  and the regularity properties of $u$ and $u_h$. 
  To the best of our knowledge, these are the first weak error estimates for spatial semi-discretizations of stochastic partial differential equations within the $L^q(\mathcal{O})$-framework. 
  For the stochastic Burgers equation, the only prior contribution appears to be \cite{BrehierCoxMillet2026}, which established a nearly second-order weak convergence rate for a semi-discrete spectral Galerkin approximation in the $L^2(\mathcal{O})$-framework. 
  Furthermore, we emphasize that the weak error estimate \cref{eq:weak-error-2} achieves the full second-order rate, which is exactly twice the corresponding strong error rate in the $L^p([0,T] \times \Omega;L^q(\mathcal{O}))$ norm (see \cref{eq:u-uh-LpLq} with $\alpha = 0$).
\end{remark}

The remainder of this section is devoted to proving \cref{prop:uh-regu,thm:uh-strong,thm:uh-weak}.
For notational convenience, we write \(a \lesssim b\) to denote \(a \leqslant C b\),
where the generic constant \(C > 0\) is independent of the mesh size \(h\).
This constant may depend on the domain \(\mathcal{O}\), the terminal time \(T\),
the operator $Q$, the initial datum \(u_0\), the parameters \(p, q, \varepsilon\),
and the regularity index \(\theta\) associated with
the spaces \(H_{D}^{\theta,q}\) and \(H_{D,h}^{\theta,q}\).
We also write \(a \sim b\) if both \(a \lesssim b\) and \(b \lesssim a\) hold.
The section is organized as follows. Subsection~\ref{ssec:prelim} presents preliminary results;
Subsection~\ref{ssec:proof-uh-well-posed} proves \cref{prop:uh-regu};
Subsection~\ref{ssec:proof-uh-strong} establishes \cref{thm:uh-strong};
and Subsection~\ref{ssec:proof-uh-weak} proves \cref{thm:uh-weak}.

\subsection{Preliminary results}
\label{ssec:prelim}
It is well known (see \cite{Bakaev2002}) that for any $\omega_0 \in (\pi/2, \pi)$ and $q \in (1, \infty)$,
there exists a constant $C = C(\omega_0, q, \mathcal{O}) > 0$ such that
\begin{equation}
  \label{eq:resolvent-bound}
  \|(re^{i\vartheta}-A_h)^{-1}\|_{\mathcal{L}(H_{D,h}^{0,q}, H_{D,h}^{0,q})}
  \leqslant \frac{C}{1+r}, \quad r \geqslant 0, \quad \vartheta \in [-\omega_0, \omega_0].
\end{equation}
For fixed parameters $\gamma_{0} \in (0,\pi/2)$ and $q \in (1,\infty)$,
there exists a constant $C>0$, independent of the mesh size $h$ and $r$, such that
\[
  \|e^{irA_{h}}\|_{\mathcal{L}(H_{D,h}^{0,q}, H_{D,h}^{0,q})} \leqslant C e^{\gamma_{0} |r|}
  \quad\text{for all } r \in \mathbb{R}.
\]
We refer to \cite[Theorem~3.1]{LiZhouLp2026} for the proof in the three-dimensional case.
These estimates imply that $-A_h$ satisfies the hypotheses of \cite[Theorem~16.5]{Yagi2010};
consequently, we have the norm equivalence
\begin{equation}
  \label{eq:dotHh-complex-interp}
  \|v_{h}\|_{H_{D,h}^{\theta,q}} \sim \|v_{h}\|_{[H_{D,h}^{0,q},H_{D,h}^{2,q}]_{\theta/2}},
  \quad \theta \in (0,2), \quad q \in (1,\infty), \quad v_{h} \in X_{h}.
\end{equation}
Combining this equivalence with the standard inverse estimate $\|v_{h}\|_{H_{D,h}^{2,q}} \lesssim h^{-2} \|v_{h}\|_{H_{D,h}^{0,q}}$, complex interpolation directly yields the generalized inverse estimate
\begin{equation}
  \label{eq:inverse}
  \|v_{h}\|_{H_{D,h}^{\theta_2,q}} \lesssim h^{\theta_1-\theta_2}\, \|v_{h}\|_{H_{D,h}^{\theta_1,q}},
  \quad \theta_1 \leqslant \theta_2 \leqslant \theta_1 + 2, \quad q \in (1,\infty), \quad v_h \in X_h.
\end{equation}

\subsubsection{Fundamental properties of \texorpdfstring{\(A_{h}\) and \(P_h\)}{Ah and Ph}}
\label{subsubsec:Ah-Ph}
We recall the key approximation properties of the discrete Laplacian \(A_{h}\) required for the subsequent analysis; for proofs of these classical results, see~\cite[Chapters~5 and~8]{Brenner2008}.
\begin{lemma}
  \label{lem:Ah}
  Let \(q \in (1,\infty)\). The following operator norm estimates hold:
  \begin{enumerate}
    \item[\textup{(i)}]
      For every \(\theta \in [0,1]\),
      \( \bigl\|A_{h}^{-1} P_{h} - A^{-1}\bigr\|_{\mathcal{L}(L^q(\mathcal{O}),\,H_{D}^{\theta,q})}
        \;\lesssim\; h^{2-\theta} \).
    \item[\textup{(ii)}]
      For every \(\theta_1 \in [1,3]\) and \(\theta_2 \in [-1,1]\),
      \( \bigl\|I - A_{h}^{-1} P_{h}\, A\bigr\|_{\mathcal{L}(H_{D}^{\theta_1,q},\,H_{D}^{\theta_2,q})}
        \;\lesssim\; h^{\theta_1 - \theta_2}  \).
  \end{enumerate}
\end{lemma}
By combining the stability property \cref{eq:Ph-stab0} with the approximation properties of the finite element space \(X_h\)
and complex interpolation theory, we deduce the following error estimate for Sobolev spaces of positive order:
\[
  \|I - P_h\|_{\mathcal{L}(H_{D}^{\theta_1,q},\, H_{D}^{\theta_2,q})}
  \;\lesssim\; h^{\theta_1-\theta_2},
  \qquad \theta_1 \in [0,3], \quad \theta_2 \in [0,\min\{1,\theta_1\}], \quad q \in (1,\infty).
\]
A standard duality argument extends this estimate to spaces involving negative order:
\begin{equation}
  \label{eq:Ph-conv}
  \|I - P_h\|_{\mathcal{L}(H_{D}^{\theta_1,q},\, H_{D}^{\theta_2,q})}
  \;\lesssim\; h^{\theta_1-\theta_2},
  \qquad \theta_1 \in [0,3], \quad \theta_2 \in [-3,\min\{1,\theta_1\}], \quad q \in (1,\infty).
\end{equation}
Furthermore, the following \(L^\infty(\mathcal{O})\)-error estimate holds:
\begin{equation}
  \label{eq:Ph-conv-Linfty}
  \|I - P_h\|_{\mathcal{L}(H_{D}^{\theta,q}, \, L^\infty(\mathcal{O}))}
  \lesssim h^{\theta - 1/q}, \qquad q \in (1,\infty), \quad \theta \in \bigl(\tfrac{1}{q},3\bigr].
\end{equation}
Combining \cref{eq:Ph-conv} and \cref{lem:Ah}~\textup{(ii)} via the triangle inequality yields
\begin{equation}
  \label{eq:Ph-AhinvPhA}
  \bigl\|P_h - A_h^{-1} P_h A\bigr\|_{\mathcal{L}(H_{D}^{\theta_1,q},\,H_{D}^{\theta_2,q})}
  \;\lesssim\; h^{\theta_1 - \theta_2},
  \quad \theta_1 \in [1,3], \quad \theta_2 \in [-1,1], \quad q \in (1,\infty).
\end{equation}

In addition, we have the following variant of the estimate in \cref{lem:Ah}(i).
\begin{lemma}
  Let \(\theta_1 \in [1,2]\), \(\theta_2 \in [0,1]\), and \(q \in (1,\infty)\).
  Then
  \begin{equation}
    \label{eq:I-PhAinvAh}
    \bigl\|I - P_hA^{-1}A_h\bigr\|_{\mathcal{L}(H_{D,h}^{\theta_1,q},\,H_{D}^{\theta_2,q})}
    \;\lesssim\; h^{\theta_1 - \theta_2}.
  \end{equation}
\end{lemma}
\begin{proof}
  We employ the algebraic identity valid on $H_{D,h}^{\theta_1,q}$:
  \[
    I - P_hA^{-1}A_h = (I - P_h)A^{-1}A_h + (A_h^{-1}P_h - A^{-1})A_h,
  \]
  where we have used that $I = A_h^{-1}P_h A_h$ on $H_{D,h}^{\theta_1,q}$.
  Regarding the first term, the submultiplicativity of the operator norm,
  the projection error estimate $\|I-P_h\|_{\mathcal{L}(H_{D}^{2,q},H_{D}^{\theta_2,q})} \lesssim h^{2-\theta_2}$
  (a special case of \eqref{eq:Ph-conv}),
  and the isometric isomorphism $A^{-1} \colon L^q(\mathcal{O}) \to H_{D}^{2,q}$ imply
  \begin{align*}
    \|(I - P_h)A^{-1}A_h\|_{\mathcal{L}(H_{D,h}^{\theta_1,q}, H_{D}^{\theta_2,q})}
    &\leqslant \|I - P_h\|_{\mathcal{L}(H_{D}^{2,q}, H_{D}^{\theta_2,q})}
    \|A^{-1}\|_{\mathcal{L}(L^q(\mathcal{O}), H_{D}^{2,q})}
    \|A_h\|_{\mathcal{L}(H_{D,h}^{\theta_1,q}, H_{D,h}^{0,q})} \\
    &\lesssim h^{2-\theta_2} \|A_h\|_{\mathcal{L}(H_{D,h}^{\theta_1,q}, H_{D,h}^{0,q})}.
  \end{align*}
  Observing that the discrete inverse estimate \eqref{eq:inverse} implies
  $\|A_h\|_{\mathcal{L}(H_{D,h}^{\theta_1,q}, H_{D,h}^{0,q})} \lesssim h^{\theta_1-2}$,
  we then obtain
  \begin{align*}
    \|(I - P_h)A^{-1}A_h\|_{\mathcal{L}(H_{D,h}^{\theta_1,q}, H_{D}^{\theta_2,q})}
    \lesssim h^{\theta_1-\theta_2}.
  \end{align*}
  For the second term, \cref{lem:Ah}(i) and the estimate for $\|A_h\|_{\mathcal{L}(H_{D,h}^{\theta_1,q},H_{D,h}^{0,q})}$
  derived above lead to
  \begin{align*}
    \|(A_h^{-1}P_h - A^{-1})A_h\|_{\mathcal{L}(H_{D,h}^{\theta_1,q}, H_{D}^{\theta_2,q})}
    &\leqslant \|A_h^{-1}P_h - A^{-1}\|_{\mathcal{L}(L^q(\mathcal{O}), H_{D}^{\theta_2,q})}
    \|A_h\|_{\mathcal{L}(H_{D,h}^{\theta_1,q}, H_{D,h}^{0,q})} \\
    &\lesssim h^{2-\theta_2} \cdot h^{\theta_1-2}
    = h^{\theta_1-\theta_2}.
  \end{align*}
  Combining these two estimates via the triangle inequality establishes \eqref{eq:I-PhAinvAh}.
\end{proof}

The stability estimate \cref{eq:Ph-stab0} admits the following extension, which is well known;
we include a proof for completeness.
\begin{lemma}
  \label{lem:Ph-stab}
  For all \(\theta \in [0,2]\) and \(q \in (1,\infty)\), the operator norm \(\|P_h\|_{\mathcal{L}(H_{D}^{\theta,q},\, H_{D,h}^{\theta,q})}\) is uniformly bounded with respect to \(h \in (0,1]\).
\end{lemma}

\begin{proof}
  We first verify the stability estimates at the endpoints $\theta = 0$ and $\theta = 2$. The case $\theta = 0$ follows directly from the $L^q$-stability of $P_h$ stated in \cref{eq:Ph-stab0}.
  For $\theta = 2$, we decompose $P_h$ as $(P_h - A_h^{-1}P_hA) + A_h^{-1}P_hA$. By the triangle inequality,
  \[
    \|P_h\|_{\mathcal{L}(H_{D}^{2,q},\,H_{D,h}^{2,q})}
    \leqslant \|P_h - A_h^{-1}P_hA\|_{\mathcal{L}(H_{D}^{2,q},\,H_{D,h}^{2,q})}
    + \|A_h^{-1}P_hA\|_{\mathcal{L}(H_{D}^{2,q},\,H_{D,h}^{2,q})}.
  \]
  To estimate the first term, we note that the operator $P_h - A_h^{-1}P_hA$ maps into the finite element space $X_h$.
  Thus, applying the inverse estimate \cref{eq:inverse} with $(\theta_1, \theta_2) = (0,2)$ in conjunction with the approximation property \cref{eq:Ph-AhinvPhA} for $(\theta_1, \theta_2) = (2,0)$ yields
  \[
    \|P_h - A_h^{-1}P_hA\|_{\mathcal{L}(H_{D}^{2,q},\,H_{D,h}^{2,q})}
    \lesssim h^{-2}\, \|P_h - A_h^{-1}P_hA\|_{\mathcal{L}(H_{D}^{2,q},\,L^q(\mathcal{O}))} \lesssim 1.
  \]
  For the second term, we observe that $A \colon H_{D}^{2,q} \to H_{D}^{0,q}$ and
  $A_h^{-1} \colon H_{D,h}^{0,q} \to H_{D,h}^{2,q}$ are isometric isomorphisms.
  Therefore, by the submultiplicativity of the operator norm and the stability for $\theta = 0$, we obtain
  \[
    \|A_h^{-1}P_hA\|_{\mathcal{L}(H_{D}^{2,q},\,H_{D,h}^{2,q})} \lesssim
    \|P_h\|_{\mathcal{L}(H_{D}^{0,q}, H_{D,h}^{0,q})} \lesssim 1.
  \]
  Combining these bounds establishes the uniform estimate for $\theta = 2$.
  Finally, the stability for intermediate values $\theta \in (0,2)$
  follows from complex interpolation between the endpoint estimates,
  utilizing the norm equivalence in \cref{eq:dotHh-complex-interp}.
\end{proof}

\subsubsection{Embedding properties of the discrete spaces}
\label{subsubsec:dotHh-embed}

The discrete Sobolev spaces $H_{D,h}^{\theta,q}$ satisfy some embedding relations analogous to their
continuous counterparts. Although these properties can be established using standard finite element techniques (see, e.g.,~\cite[Chapter 8]{Brenner2008}), we provide complete proofs below for the reader's convenience.

We begin with a key norm equivalence for the discrete space $ H_{D,h}^{\theta,q} $.
\begin{lemma}
  \label{lem:dotHh-equiv}
  Let $\theta \in [-1,1]$ and $q \in (1,\infty)$. Then, for all $v_h \in X_h$,
  $ \|v_h\|_{H_{D,h}^{\theta,q}} \sim \|v_h\|_{H_{D}^{\theta,q}} $.
\end{lemma}

\begin{proof}
  Let \(\langle \cdot, \cdot \rangle\) denote the \(L^q(\mathcal{O})\)--\(L^{q'}(\mathcal{O})\) duality pairing,
  where \(q'\) is the Hölder conjugate exponent of \(q\). 
  We split the proof into two parts.

\textbf{Part (a).}
We first establish the equivalence for \(\theta \in [0,1]\).
Since \(A_h v_h \in X_h\), we have \(v_h = A_h^{-1}P_h A_h v_h\). By the triangle inequality,
\[
  \|v_h\|_{H_{D}^{\theta,q}}
  \leqslant \|(A_h^{-1}P_h - A^{-1})A_h v_h\|_{H_{D}^{\theta,q}}
  + \|A^{-1}A_h v_h\|_{H_{D}^{\theta,q}}.
\]
The first term is bounded by \cref{lem:Ah}(i) and the inverse estimate~\cref{eq:inverse}:
\[
  \|(A_h^{-1}P_h - A^{-1})A_h v_h\|_{H_{D}^{\theta,q}}
  \lesssim h^{2-\theta} \|A_h v_h\|_{L^q(\mathcal{O})}
  = h^{2-\theta} \|v_h\|_{H_{D,h}^{2,q}}
  \lesssim \|v_h\|_{H_{D,h}^{\theta,q}}.
\]
For the second term, we use \(\|A^{-1}A_h v_h\|_{H_{D}^{\theta,q}} = \|(-A)^{\theta/2-1}A_h v_h\|_{L^q(\mathcal{O})}\) and estimate by duality.
For any \(\varphi \in L^{q'}(\mathcal{O})\) with \(\|\varphi\|_{L^{q'}(\mathcal{O})} \leqslant 1\), we have
\[
  \begin{aligned}
    | \bigl\langle (-A)^{\theta/2-1}A_h v_h, \, \varphi \bigr\rangle |
    &= | \bigl\langle A_h v_h, \, (-A)^{\theta/2-1}\varphi \bigr\rangle | \\
    &= | \bigl\langle A_h v_h, \, P_h(-A)^{\theta/2-1}\varphi \bigr\rangle | \\
    &= | -\bigl\langle (-A_h)^{\theta/2}v_h, \, (-A_h)^{1-\theta/2}P_h(-A)^{\theta/2-1}\varphi \bigr\rangle | \\
    &\leqslant \|v_h\|_{H_{D,h}^{\theta,q}}\,
    \|P_h(-A)^{\theta/2-1}\varphi\|_{H_{D,h}^{2-\theta,q'}}.
  \end{aligned}
\]
By \cref{lem:Ph-stab} and the fact that \((-A)^{\theta/2-1}\colon L^{q'}(\mathcal{O}) \to H_{D}^{2-\theta,q'}\)
is an isometry, we deduce
\[
  \|P_h(-A)^{\theta/2-1}\varphi\|_{H_{D,h}^{2-\theta,q'}} \lesssim \|(-A)^{\theta/2-1}\varphi\|_{H_{D}^{2-\theta,q'}} = \|\varphi\|_{L^{q'}(\mathcal{O})}.
\]
Taking the supremum over \(\varphi\) yields that the second term satisfies \(\|A^{-1}A_h v_h\|_{H_{D}^{\theta,q}} \lesssim \|v_h\|_{H_{D,h}^{\theta,q}}\).
Hence, combining the estimates for the two terms gives \(\|v_h\|_{H_{D}^{\theta,q}} \lesssim \|v_h\|_{H_{D,h}^{\theta,q}}\).
  The reverse inequality for \(\theta \in [0,1]\) is an immediate consequence of \(v_h = P_h v_h\) and \cref{lem:Ph-stab}:
  \[
    \|v_h\|_{H_{D,h}^{\theta,q}} = \|P_h v_h\|_{H_{D,h}^{\theta,q}} \lesssim \|v_h\|_{H_{D}^{\theta,q}}.
  \]

  \textbf{Part (b).}
  We now consider the case \(\theta \in [-1,0)\).
  By the dual characterization of the continuous norm,
  \[
    \|v_h\|_{H_{D}^{\theta,q}} = \sup_{0 \neq \psi \in H_{D}^{-\theta,q'}} \frac{|\langle v_h, \psi \rangle|}{\|\psi\|_{H_{D}^{-\theta,q'}}}.
  \]
  For any \(\psi \in H_{D}^{-\theta,q'}\), the inclusion \(v_h \in X_h\) and the definition of \(P_h\) give \(\langle v_h, \psi \rangle = \langle v_h, P_h \psi \rangle\).
  Using the discrete duality estimate together with \cref{lem:Ph-stab}, we obtain
  \[
    |\langle v_h, \psi \rangle| \leqslant
    \|v_h\|_{H_{D,h}^{\theta,q}} \|P_h \psi\|_{H_{D,h}^{-\theta,q'}}
    \lesssim \|v_h\|_{H_{D,h}^{\theta,q}} \|\psi\|_{H_{D}^{-\theta,q'}}.
  \]
  Taking the supremum over all such \(\psi\) yields \(\|v_h\|_{H_{D}^{\theta,q}} \lesssim \|v_h\|_{H_{D,h}^{\theta,q}}\).

  For the reverse inequality, by the dual characterization of the discrete norm and the continuous duality inequality,
  \[
    \|v_h\|_{H_{D,h}^{\theta,q}} = \sup_{w_h \in X_h \setminus \{0\}} \frac{|\langle v_h, w_h \rangle|}{\|w_h\|_{H_{D,h}^{-\theta,q'}}} \leqslant \sup_{w_h \in X_h \setminus \{0\}} \frac{\|v_h\|_{H_{D}^{\theta,q}} \|w_h\|_{H_{D}^{-\theta,q'}}}{\|w_h\|_{H_{D,h}^{-\theta,q'}}}.
  \]
  Since \(-\theta \in (0,1]\) and \(q' \in (1,\infty)\), the norm equivalence established in Part (a) for the exponent
  pair \((-\theta,q')\) yields \(\|w_h\|_{H_{D}^{-\theta,q'}} \sim \|w_h\|_{H_{D,h}^{-\theta,q'}}\)
  for every \(w_h \in X_h\). Inserting this equivalence into the right--hand side gives
  \(\|v_h\|_{H_{D,h}^{\theta,q}} \lesssim \|v_h\|_{H_{D}^{\theta,q}}\).
  This completes the proof.
\end{proof}

By the norm equivalence from \cref{lem:dotHh-equiv}, we now extend the stability estimate of \cref{lem:Ph-stab} to the range \(\theta \in [-1,0)\).

\begin{lemma}
  \label{lem:Ph-stab-negative}
  For all \(\theta \in [-1,0)\) and \(q \in (1,\infty)\), the operator norm
  \(\|P_h\|_{\mathcal{L}(H_{D}^{\theta,q},\, H_{D,h}^{\theta,q})}\)
  is uniformly bounded with respect to \(h \in (0,1]\).
\end{lemma}

\begin{proof}
  Fix \(\theta \in [-1,0)\), \(q \in (1,\infty)\), and \(f \in H_{D}^{\theta,q}\).  
  By \cref{eq:Ph-extend} and the norm equivalence in \cref{lem:dotHh-equiv}, we obtain
  \begin{align*}
    \|P_h f\|_{H_{D,h}^{\theta,q}}
    &=
    \sup_{v_h \in X_h\setminus\{0\}} \frac{|\langle f, v_h \rangle|}{\|v_h\|_{H_{D,h}^{-\theta,q'}}}
    \\
    &\leqslant
    \sup_{v_h \in X_h\setminus\{0\}} \frac{\|f\|_{H_{D}^{\theta,q}} \|v_h\|_{H_{D}^{-\theta,q'}}}{\|v_h\|_{H_{D,h}^{-\theta,q'}}}
    \lesssim \|f\|_{H_{D}^{\theta,q}},
  \end{align*}
  where \(\langle \cdot, \cdot \rangle\) denotes the duality pairing between \(L^q(\mathcal{O})\) and \(L^{q'}(\mathcal{O})\).
  This proves the claimed uniform boundedness.
\end{proof}

Finally, we present the following $h$-uniform embeddings.
\begin{lemma}
  \label{lem:dotHh}
  The following embeddings hold uniformly with respect to $h$:
  \begin{enumerate}
    \item[\textup{(i)}] $H_{D,h}^{\theta,2} \hookrightarrow W^{1,\infty}(\mathcal{O})$ for every $\theta \in (3/2, 2)$;
    \item[\textup{(ii)}] $H_{D,h}^{2,q} \hookrightarrow H_{D,h}^{\theta,q}$ for every
    $q \in (1, \infty)$ and $0 \leqslant \theta < 2$;
    \item[\textup{(iii)}] $(H_{D,h}^{0,q},\,H_{D,h}^{2,q})_{\theta,p} \hookrightarrow H_{D,h}^{\alpha,q}$ for every $p, q \in (1, \infty)$, $\theta \in (0, 1)$, and $0 < \alpha < 2\theta$.
  \end{enumerate}
\end{lemma}
\begin{proof}
\textup{(i)} Fix $\theta \in (3/2, 2)$. For an arbitrary $v_h \in X_h$,
we decompose $v_h = (I - P_hA^{-1}A_h)v_h + P_h A^{-1} A_h v_h$ and estimate each term separately.
The inverse estimate \cite[Theorem~4.5.11]{Brenner2008} implies
\[
  \|(I - P_hA^{-1}A_h)v_h\|_{W^{1,\infty}(\mathcal{O})}
  \lesssim h^{-1/2} \|(I - P_hA^{-1}A_h)v_h\|_{H_{D}^{1,2}}
  \lesssim h^{\theta-3/2} \|v_h\|_{H_{D,h}^{\theta,2}}
  \lesssim \|v_h\|_{H_{D,h}^{\theta,2}},
\]
where the second inequality uses \eqref{eq:I-PhAinvAh} with $(\theta_1, \theta_2,q) = (\theta, 1, 2)$
and the third inequality uses the fact $h \leqslant 1$.
Since $\|P_h\|_{\mathcal{L}(H_{D}^{\theta,2},W^{1,\infty}(\mathcal{O}))} \lesssim 1$ (see \cref{rem:Ph-stab} below),
we have
\begin{align*}
  \|P_h A^{-1} A_h v_h\|_{W^{1,\infty}(\mathcal{O})}
  &\lesssim \|A^{-1} A_h v_h\|_{H_{D}^{\theta,2}} 
  = \| A_h v_h\|_{H_{D}^{\theta-2,2}} 
  \sim \|A_h v_h\|_{H_{D,h}^{\theta-2,2}}
  = \|v_h\|_{H_{D,h}^{\theta,2}},
\end{align*}
where we have also used the isometry $A^{-1} \colon H_{D}^{\theta-2,2} \to H_{D}^{\theta,2}$
and the norm equivalence $\|\cdot\|_{H_{D,h}^{\theta-2,2}} \sim \|\cdot\|_{H_{D}^{\theta-2,2}}$
on $X_h$ (see \cref{lem:dotHh-equiv}).
The claim follows by combining these two estimates via the triangle inequality.

\textup{(ii)} Let $q \in (1, \infty)$. The $L^q(\mathcal{O})$-stability of $P_h$ \eqref{eq:Ph-stab0} and
the boundedness of $A^{-1}$ on $L^q(\mathcal{O})$ imply
$ \|P_hA^{-1}A_h\|_{\mathcal{L}(H_{D,h}^{2,q}, L^q(\mathcal{O}))} \lesssim 1 $.
Combining this with estimate \eqref{eq:I-PhAinvAh} for $(\theta_1,\theta_2) = (2,0)$ via the triangle inequality yields
\[
  \|I\|_{\mathcal{L}(H_{D,h}^{2,q},L^q(\mathcal{O}))}
  \leqslant \|P_hA^{-1}A_h\|_{\mathcal{L}(H_{D,h}^{2,q},L^q(\mathcal{O}))}
  + \|I - P_hA^{-1}A_h\|_{\mathcal{L}(H_{D,h}^{2,q},L^q(\mathcal{O}))}
  \lesssim 1 + h^2 \lesssim 1.
\]
Thus, the embedding $H_{D,h}^{2,q} \hookrightarrow H_{D,h}^{0,q}$ is $h$-uniformly bounded.
By \cite[Corollary~2.8]{Lunardi2018} and the norm equivalence \eqref{eq:dotHh-complex-interp},
we obtain the desired $h$-uniform embedding $H_{D,h}^{2,q} \hookrightarrow H_{D,h}^{\theta,q}$
for all $0 \leqslant \theta < 2$.

\textup{(iii)} Let $p, q \in (1, \infty)$ and $\theta \in (0, 1)$.
By part \textup{(ii)}, the embedding $H_{D,h}^{2,q} \hookrightarrow H_{D,h}^{0,q}$ is bounded uniformly in $h$.
Define $r := \min\{q, q'\}$, where $q'$ denotes the H\"older conjugate exponent of $q$.
For any $0 < \alpha < 2\theta$, \cite[Proposition~1.4]{Lunardi2018} yields the $h$-uniform embedding
\[
  (H_{D,h}^{0,q},H_{D,h}^{2,q})_{\theta,p} \hookrightarrow (H_{D,h}^{0,q},H_{D,h}^{2,q})_{\alpha/2,r}.
\]
Furthermore, as $L^q(\mathcal{O})$ is of Fourier type $r$ (\cite[Example~2.4.14]{HytonenWeis2016}),
the definition of Fourier type (\cite[Definition~2.4.12]{HytonenWeis2016}) implies that
$H_{D,h}^{0,q}$ and $H_{D,h}^{2,q}$ are also of Fourier type $r$ and related constants
are independent of $h$.
Consequently, \cite[Theorem~C.4.1]{HytonenWeis2016} yields the $h$-uniform embedding
\[
  (H_{D,h}^{0,q},H_{D,h}^{2,q})_{\alpha/2,r} \hookrightarrow [H_{D,h}^{0,q},H_{D,h}^{2,q}]_{\alpha/2}.
\]
Combining this with the norm equivalence
$\|\cdot\|_{H_{D,h}^{\alpha,q}} \sim \|\cdot\|_{[H_{D,h}^{0,q},H_{D,h}^{2,q}]_{\alpha/2}}$
on $X_h$ (see \eqref{eq:dotHh-complex-interp}),
we obtain the desired $h$-uniform embedding $(H_{D,h}^{0,q},H_{D,h}^{2,q})_{\theta,p} \hookrightarrow H_{D,h}^{\alpha,q}$ for all $0 < \alpha < 2\theta$.
\end{proof}

\begin{remark}
  \label{rem:Ph-stab}
  Fix \(\theta \in (3/2,2)\) and let \(\mathcal{I}^h\) denote the standard Lagrange interpolation operator onto the
  finite element space \(X_h\). We observe the chain of estimates, for any $v \in H_{D}^{\theta,2}$,
  \begin{align*}
    \|(\mathcal{I}^h - P_h)v\|_{W^{1,\infty}(\mathcal{O})}
    &\lesssim h^{-1/2} \|(\mathcal{I}^h - P_h)v\|_{H_{D}^{1,2}} \\
    & \lesssim h^{-1/2} \|v - \mathcal{I}^hv\|_{H_{D}^{1,2}} 
    + h^{-1/2} \|v - P_hv\|_{H_{D}^{1,2}} \\
    & \lesssim h^{\theta-3/2}\|v\|_{H_{D}^{\theta,2}},
  \end{align*}
  where the first inequality follows from \cite[Theorem~4.5.11]{Brenner2008},
  the second is a consequence of the triangle inequality, and the third relies on
  \cite[Theorem~4.4.20]{Brenner2008} (via complex interpolation) together with
  \cref{eq:Ph-conv} applied with \((\theta_1,\theta_2,q) = (\theta,1,2)\). 
  Since \(\theta > 3/2\) and the mesh size \(h\) satisfies \(h \leqslant 1\),
  it follows that
  \[
    \|\mathcal{I}^h - P_h\|_{\mathcal{L}(H_{D}^{\theta,2},W^{1,\infty}(\mathcal{O}))}
    \lesssim 1.
  \]
  Moreover, \cite[Theorem~4.4.20]{Brenner2008} implies that
  \(\|\mathcal{I}^h\|_{\mathcal{L}(W^{1,\infty}(\mathcal{O}),W^{1,\infty}(\mathcal{O}))} \lesssim 1\).
  Together with the embedding \(H_{D}^{\theta,2} \hookrightarrow W^{1,\infty}(\mathcal{O})\), this yields
  $\|\mathcal{I}^h\|_{\mathcal{L}(H_{D}^{\theta,2},W^{1,\infty}(\mathcal{O}))} \lesssim 1$.
  Applying the triangle inequality to \(P_h = \mathcal{I}^h - (\mathcal{I}^h - P_h)\) yields the \(h\)-uniform operator norm estimate 
  \(\|P_h\|_{\mathcal{L}(H_{D}^{\theta,2},W^{1,\infty}(\mathcal{O}))} \lesssim 1\).
\end{remark}

\subsubsection{Discrete analytic semigroup: smoothing, convolutions, and maximal \texorpdfstring{$L^p$}{Lp}-regularity}
\label{subsubsec:Sh}
Let \(S_h(t):=e^{tA_h}\), \(t\geqslant0\), be the analytic semigroup generated by the discrete Laplacian \(A_h\).
Combining the sectorial resolvent estimate \eqref{eq:resolvent-bound} with standard semigroup theory (see, e.g., \cite[Equation~(6.25), Chapter~2]{Pazy1983}), we obtain the smoothing estimate
\begin{equation}
  \label{eq:Sh-smoothing}
  \|S_h(t)\|_{\mathcal{L}(H_{D,h}^{\theta_1,q},H_{D,h}^{\theta_2,q})}
  \lesssim t^{-(\theta_2-\theta_1)/2}
\end{equation}
for all \(\theta_1,\theta_2\in\mathbb R\) with \(\theta_1\leqslant\theta_2\), \(q\in(1,\infty)\), and \(t>0\).

We define the spatially semi-discrete deterministic convolution by
\begin{equation}
  \label{eq:S0h-def}
  (S_h \ast g_h)(t) := \int_0^t S_h(t-s) g_h(s) \, \mathrm{d}s,
  \quad t \in [0,T], \quad g_h \in L^1(0,T;H_{D,h}^{0,2}).
\end{equation}

\begin{lemma} \label{lem:Sh-ast-bound}
  The semi-discrete convolution operator satisfies the following stability estimates:
  \begin{enumerate}
    \item[\textup{(i)}] Let $p \in (1,\infty]$, $q \in (1,\infty)$,
      $\theta_1 \in \mathbb{R}$, and $ \theta_2 \in [\theta_1, \theta_1 + 2-2/p)$.
      For any $g_h \in L^p(0,T;H_{D,h}^{\theta_1,q})$, the following estimate holds:
      \[
        \|S_h \ast g_h\|_{C([0,T];H_{D,h}^{\theta_2,q})}
        \lesssim \|g_h\|_{L^p(0,T;H_{D,h}^{\theta_1,q})}.
      \]
    \item[\textup{(ii)}] Let $p \in (1,\infty]$, $q \in (1,\infty)$,
      and $\theta \in (-2,0]$ such that $(\theta + 2 - 2/p) q > 1$.
      Then, for any $g_h \in L^p(0,T;H_{D,h}^{\theta,q})$,
      \[
        \|S_h \ast g_h\|_{C([0,T];L^\infty(\mathcal{O}))} \lesssim \|g_h\|_{L^p(0,T;H_{D,h}^{\theta,q})}.
      \]
    \item[\textup{(iii)}] Let $p \in (4,\infty]$, $g \in L^\infty(0,T;L^\infty(\mathcal{O}))$,
      and $z \in L^p(0,T;L^2(\mathcal{O}))$. It follows that
      \[
        \|S_h \ast (P_h \partial_x (gz))\|_{C([0,T];L^\infty(\mathcal{O}))}
        \lesssim \|g\|_{L^\infty(0,T;L^\infty(\mathcal{O}))}
        \|z\|_{L^p(0,T;L^2(\mathcal{O}))}.
      \]
  \end{enumerate}
\end{lemma}
\begin{proof}
  \textbf{Proof of (i).}
  Since $S_h(t)$, $t \geqslant 0$, is the analytic semigroup generated by $A_h$ and
  the estimate \cref{eq:Sh-smoothing} holds uniformly in $h$,
  standard analytic semigroup theory (see, e.g., \cite[Theorem~3.1, Chapter~4]{Pazy1983}
  and \cite[Proposition~4.2.1]{Lunardi1995analytic} for the continuous analogue) yields
  the desired estimate.


\textbf{Proof of (ii).}
Select $\alpha$ such that $1/q < \alpha < \min\{1,\, \theta+2-2/p\}$; such a choice is admissible because $(\theta + 2 - 2/p)q > 1$.
Applying (i) with $(\theta_1,\theta_2) = (\theta,\alpha)$
and employing the norm equivalence $\|\cdot\|_{H_{D,h}^{\alpha,q}} \sim \|\cdot\|_{H_{D}^{\alpha,q}}$
on $X_h$ (see \cref{lem:dotHh-equiv}) yields
\[
  \|S_h \ast g_h\|_{C([0,T];H_{D}^{\alpha,q})}
  \lesssim \|g_h\|_{L^p(0,T;H_{D,h}^{\theta,q})}.
\]
Since $\alpha > 1/q$, the Sobolev embedding $H_{D}^{\alpha,q} \hookrightarrow L^\infty(\mathcal{O})$ implies the desired bound.

\textbf{Proof of (iii).}
Since $p>4$, the choice $\theta=-1$ and $q=2$ yields
\[
  \Bigl(\theta+2-\frac{2}{p}\Bigr)q = 2\Bigl(1-\frac{2}{p}\Bigr) > 1,
\]
which verifies the hypotheses of part~(ii).
For every $v_h\in X_h$, the identity $\|v_h\|_{H^{1,2}_{D,h}} = \|\partial_x v_h\|_{L^2(\mathcal{O})}$ combined with a standard duality argument gives
\[
  \|P_h\partial_x(gz)\|_{H^{-1,2}_{D,h}} \leqslant \|gz\|_{L^2(\mathcal{O})}
  \leqslant \|g\|_{L^\infty(\mathcal{O})}\|z\|_{L^2(\mathcal{O})}.
\]
Consequently,
\[
  \|P_h\partial_x(gz)\|_{L^p(0,T;H^{-1,2}_{D,h})}
  \lesssim \|g\|_{L^\infty(0,T;L^\infty(\mathcal{O}))}\|z\|_{L^p(0,T;L^2(\mathcal{O}))}.
\]
The desired estimate now follows from part~(ii) applied with $(\theta,q)=(-1,2)$.
\end{proof}

We also have the following discrete maximal $L^p$-regularity estimate \cite[Theorem~3.2]{Geissert2006}.
\begin{lemma}[Discrete maximal $L^p$-regularity]
  \label{lem:S0h}
  Let $p,q \in (1,\infty)$ and $ \theta \in \mathbb{R} $. For $g_h \in L^p(0,T;H_{D,h}^{\theta,q})$,
  \[
    \|S_h \ast g_h\|_{L^p(0,T;H_{D,h}^{\theta+2,q})} \lesssim \|g_h\|_{L^p(0,T;H_{D,h}^{\theta,q})}.
  \]
\end{lemma}

For $g_h \in L_{\mathbb{F}}^2([0,T] \times \Omega; \gamma(\ell^2, H_{D,h}^{0,2}))$,
the spatially semi-discrete stochastic convolution is defined by
\begin{equation}
  \label{eq:S1h-def}
  (S_h \diamond g_h)(t) := \int_0^t S_h(t-s)\, g_h(s)\,\mathrm{d}W(s),\quad t\in[0,T].
\end{equation}
This process possesses an $\mathbb{F}$-adapted continuous modification;
henceforth, we always identify it with this continuous modification.

\begin{lemma}[Discrete stochastic maximal $L^p$-regularity]
  \label{lem:S1h}
  Let $p \in (2,\infty)$ and $q \in [2,\infty)$.
  For any $g_h \in L_\mathbb{F}^p([0,T] \times \Omega;\gamma(\ell^2,H_{D,h}^{0,q}))$,
  \begin{align*}
    & \|S_h \diamond g_h\|_{L^p(\Omega;C([0,T];(H_{D,h}^{0,q}, H_{D,h}^{2,q})_{1/2-1/p,p}))}
    + \|S_h \diamond g_h\|_{L^p([0,T] \times \Omega;H_{D,h}^{1,q})} \\
    \lesssim{} & \|g_h\|_{L^p([0,T] \times \Omega;\gamma(\ell^2,H_{D,h}^{0,q}))}.
  \end{align*}
\end{lemma}
\begin{remark}
  Recently, in three dimensions, Li and Zhou~\cite[Theorem~3.1]{LiZhouLp2026} proved that the boundedness
  constant of the $H^\infty$-calculus for the negative discrete Laplacian $-\Delta_h$ is independent of the spatial mesh size $h$. By the same arguments, this $h$-independence extends to the $H^\infty$-calculus for $-A_h$.
  Consequently, \cref{lem:S1h} follows directly from \cite[Theorem~3.5]{Neerven2012b}; see \cite[Theorem~3.2]{LiZhouLp2026} for further details.
\end{remark}

\subsubsection{The spatially semi-discrete stochastic convolution \texorpdfstring{$G_h$}{Gh}}
\label{subsubsec:G-Gh}

For the stochastic convolution $G$ defined in \cref{eq:G-def}, its spatially semi-discrete counterpart $G_h$ is defined by
\begin{equation}
  \label{eq:Gh-def}
  G_h := S_h \diamond (P_hQ).
\end{equation}
In what follows, $G_h$ will always denote its continuous modification.

\begin{lemma}
  Let $G_h$ be defined by \cref{eq:Gh-def}. Then, for any $p, q \in [2,\infty)$ and $\theta \in (0,1)$,
  \begin{align}
    \sup_{0 < h \leqslant 1} \|G_h\|_{L^p([0,T] \times \Omega; H_{D}^{1,q})} &< \infty,
    \label{eq:Gh-Lp-H1q} \\
    \sup_{0 < h \leqslant 1} \|G_h\|_{L^p(\Omega;C([0,T];H_{D}^{\theta,q}))} &< \infty.
    \label{eq:Gh-Lp-C-Hthetaq}
  \end{align}
\end{lemma}

\begin{proof}
  By the discrete stochastic maximal $L^p$-regularity estimate (\cref{lem:S1h}) and the bound \cref{eq:PhQ-bound}, we have
  \[
    \sup_{0 < h \leqslant 1} \Big( \|G_h\|_{L^p([0,T] \times \Omega; H_{D,h}^{1,q})}
    + \|G_h\|_{L^p(\Omega;C([0,T];(H_{D,h}^{0,q}, H_{D,h}^{2,q})_{1/2-1/p,p}))} \Big) < \infty
  \]
  for $p \in (2,\infty)$ and $q \in [2,\infty)$.
  Using the norm equivalence on $X_h$ (\cref{lem:dotHh-equiv}) and the discrete Sobolev embedding (\cref{lem:dotHh}\textup{(iii)}), we deduce that for $p \in (2,\infty)$, $q \in [2,\infty)$, and $\theta \in (0, 1-2/p)$,
  \[
    \sup_{0 < h \leqslant 1} \Big( \|G_h\|_{L^p([0,T] \times \Omega; H_{D}^{1,q})}
    + \|G_h\|_{L^p(\Omega;C([0,T];H_{D}^{\theta,q}))} \Big) < \infty.
  \]
  For arbitrary $p, q \in [2,\infty)$ and $\theta \in (0,1)$, we choose $p_* > p$ sufficiently large so that $1 - 2/p_* > \theta$.
  Applying the preceding estimate with $(p, q, \theta)$ replaced by $(p_*, q, \theta)$,
  and using the embeddings $L^{p_*}([0,T]\times\Omega; H_{D}^{1,q}) \hookrightarrow L^{p}([0,T]\times\Omega; H_{D}^{1,q})$
  and $L^{p_*}(\Omega;C([0,T];H^{\theta,q})) \hookrightarrow L^{p}(\Omega;C([0,T];H^{\theta,q}))$,
  yields \cref{eq:Gh-Lp-H1q,eq:Gh-Lp-C-Hthetaq}.
\end{proof}

\begin{lemma} \label{lem:G-Gh}
  Let \(G\) and \(G_h\) be given by \cref{eq:G-def} and \cref{eq:Gh-def}, respectively.
  Then, for any \(\alpha \in [-1,0]\), \(p,q \in [2,\infty)\), and \(\varepsilon \in (0,1)\), the following estimates hold:
  \begin{align}
    \|G-G_h\|_{L^p([0,T] \times \Omega;H_{D}^{\alpha,q})} &\lesssim h^{1-\alpha}, \label{eq:G-Gh-LpLq} \\
    \|G-G_h\|_{L^p(\Omega;C([0,T];H_{D}^{\alpha,q}))} &\lesssim h^{1-\alpha-\varepsilon}, \label{eq:G-Gh-LpCLq} \\
    \|G-G_h\|_{L^p(\Omega;C([0,T];L^\infty(\mathcal{O})))} &\lesssim h^{1-\varepsilon}. \label{eq:G-Gh-LpCLinf}
  \end{align}
\end{lemma}

\begin{proof}
  By standard arguments (see, e.g., \cite[Proposition~4.4]{Neerven2012b}),
  there exists a $\mathbb{P}$-null set $N \subset \Omega$ such that,
  for every $\omega \in \Omega \setminus N$ and all $t \in [0,T]$,
  the following integral identity holds in $H_{D}^{-1,2}$:
  \[
    G(t) = \int_0^t A G(s) \, \mathrm{d}s + \int_0^t Q \, \mathrm{d}W(s).
  \]
  Here $A$ denotes the operator $A_{2,-1}$, whose negative $-A$ is sectorial on $H_{D}^{-1,2}$
  with domain $H_{D}^{1,2}$ (cf.~Subsection~\ref{subsec:functional-analytic-framework}).
  Applying the projection $P_h$ to both sides, we obtain, for every $\omega \in \Omega \setminus N$ and all $t \in [0,T]$,
  \[
    P_h G(t) = \int_0^t P_h A G(s) \, \mathrm{d}s + \int_0^t P_h Q \, \mathrm{d}W(s).
  \]
  Meanwhile, the discrete counterpart $G_h$ satisfies, for every $\omega \in \Omega \setminus N$ and all $t \in [0,T]$,
  \[
    G_h(t) = \int_0^t A_h G_h(s) \, \mathrm{d}s + \int_0^t P_h Q \, \mathrm{d}W(s),
  \]
  where $N$ has been enlarged by another $\mathbb{P}$-null set, still denoted by the same symbol.
  Let $z_h := G_h - P_h G$. Subtracting the identity for $P_h G$ from that for $G_h$,
  we deduce that, for every $\omega \in \Omega \setminus N$ and all $t \in [0,T]$,
  \[
    z_h(t) = \int_0^t \bigl( A_h z_h(s) + A_h P_h G(s) - P_h A G(s) \bigr) \, \mathrm{d}s.
  \]
  Hence, for every $\omega \in \Omega \setminus N$, $z_h$ is the strong solution of
  the deterministic evolution equation
  \[
    \frac{\mathrm{d}}{\mathrm{d}t} z_h(t) = A_h z_h(t) + A_h P_h G(t) - P_h A G(t),
    \quad \text{for a.e. } t \in [0,T], \quad z_h(0) = 0.
  \]
  Consequently, by the variation-of-constants formula, $z_h$ admits the following mild representation $\mathbb{P}$-a.s.:
  \begin{equation}
    \label{eq:zh-mild}
    z_h = S_h \ast (A_h P_h G - P_h A G).
  \end{equation}
  Finally, decomposing the error according to $G - G_h = (I - P_h)G - z_h$,
  it remains to estimate the two terms on the right-hand side separately.

\medskip
\textbf{Proof of \eqref{eq:G-Gh-LpLq}.}
In light of \cref{eq:zh-mild}, applying \cref{lem:S0h} with $\theta = \alpha-2$,
we deduce that
\begin{align*}
  \|z_h\|_{L^p([0,T] \times \Omega;\,H_{D,h}^{\alpha,q})}
  &\lesssim \|(A_h P_h - P_h A) G\|_{L^p([0,T] \times \Omega;\,H_{D,h}^{\alpha-2,q})} \\
  &= \|(P_h - A_h^{-1}P_h A) G\|_{L^p([0,T] \times \Omega;\,H_{D,h}^{\alpha,q})}.
\end{align*}
The norm equivalence $\|\cdot\|_{H_{D,h}^{\alpha,q}} \sim \|\cdot\|_{H_{D}^{\alpha,q}}$
on $X_h$ (\cref{lem:dotHh-equiv}) implies 
\[
  \|z_h\|_{L^p([0,T] \times \Omega;\,H_{D}^{\alpha,q})}
  \lesssim \|(P_h - A_h^{-1}P_h A) G\|_{L^p([0,T] \times \Omega;\,H_{D}^{\alpha,q})}.
\]
Applying the operator approximation estimate \eqref{eq:Ph-AhinvPhA} with $(\theta_1, \theta_2) = (1, \alpha)$ then gives
\begin{equation}
  \|z_h\|_{L^p([0,T] \times \Omega;\,H_{D}^{\alpha,q})}
  \lesssim h^{1-\alpha}\,\|G\|_{L^p([0,T] \times \Omega;\,H_{D}^{1,q})}.
  \label{eq:zh-bound}
\end{equation}
For the projection error, \eqref{eq:Ph-conv} with $(\theta_1, \theta_2) = (1, \alpha)$ yields
\begin{equation}
  \|(I-P_h)G\|_{L^p([0,T] \times \Omega;\,H_{D}^{\alpha,q})}
  \lesssim h^{1-\alpha}\,\|G\|_{L^p([0,T] \times \Omega;\,H_{D}^{1,q})}.
  \label{eq:I-Ph-bound}
\end{equation}
Combining \eqref{eq:zh-bound} and \eqref{eq:I-Ph-bound} via the triangle inequality,
and invoking the regularity of $G$ from \cref{eq:G-LpH1q}, establishes \eqref{eq:G-Gh-LpLq}.

\medskip
\textbf{Proof of \eqref{eq:G-Gh-LpCLq}.}
Let $\alpha \in [-1,0]$, $p,q \in [2,\infty)$, and $\varepsilon \in (0,1)$ be fixed, and choose
$p' \in [p,\infty)$ such that $2/p' < \varepsilon$.
An application of \cref{lem:Sh-ast-bound}(i) with $p=p'$ and
$(\theta_1,\theta_2)=(\alpha-2+\varepsilon,\alpha)$ yields
\begin{align*}
\|z_h\|_{L^{p'}(\Omega;C([0,T];H_{D,h}^{\alpha,q}))}
&\lesssim \|(A_hP_h-P_hA)G\|_{L^{p'}([0,T]\times\Omega;H_{D,h}^{\alpha-2+\varepsilon,q})} \\
&= \|(P_h-A_h^{-1}P_hA)G\|_{L^{p'}([0,T]\times\Omega;H_{D,h}^{\alpha+\varepsilon,q})}.
\end{align*}
By the norm equivalences $\|\cdot\|_{H_{D,h}^{\alpha,q}} \sim \|\cdot\|_{H_{D}^{\alpha,q}}$
and $\|\cdot\|_{H_{D,h}^{\alpha+\varepsilon,q}} \sim \|\cdot\|_{H_{D}^{\alpha+\varepsilon,q}}$
on $X_h$ (cf.~\cref{lem:dotHh-equiv}), we infer that
\begin{align*}
\|z_h\|_{L^{p'}(\Omega;C([0,T];H_{D}^{\alpha,q}))}
&\lesssim \|(P_h-A_h^{-1}P_hA)G\|_{L^{p'}([0,T]\times\Omega;H_{D}^{\alpha+\varepsilon,q})} \\
&\lesssim h^{1-\alpha-\varepsilon}\,\|G\|_{L^{p'}([0,T]\times\Omega;H_{D}^{1,q})},
\end{align*}
where the last inequality follows from \cref{eq:Ph-AhinvPhA} applied with
$(\theta_1,\theta_2)=(1,\alpha+\varepsilon)$.
Moreover, applying \cref{eq:Ph-conv} with $(\theta_1,\theta_2) = (1-\varepsilon,\alpha)$ gives
\[
\|(I-P_h)G\|_{L^{p'}(\Omega;C([0,T];H_{D}^{\alpha,q}))}
\lesssim h^{1-\alpha-\varepsilon}\,\|G\|_{L^{p'}(\Omega;C([0,T];H_{D}^{1-\varepsilon,q}))}.
\]
Combining the preceding two estimates via the triangle inequality and invoking the
regularity of $G$ provided by \cref{eq:G-LpCHq} (applied with $p=p'$ and
$\theta=1-\varepsilon$) and by \cref{eq:G-LpH1q} (applied with $p=p'$), we arrive at
\[
\|G-G_h\|_{L^{p'}(\Omega;C([0,T];H_{D}^{\alpha,q}))} \lesssim h^{1-\alpha-\varepsilon}.
\]
Since $p' \geqslant p$,
we have the embedding $L^{p'}(\Omega) \hookrightarrow L^p(\Omega)$, and so
\eqref{eq:G-Gh-LpCLq} follows.

\medskip
\textbf{Proof of \eqref{eq:G-Gh-LpCLinf}.}
Let $p\in[2,\infty)$ and $\varepsilon\in(0,1)$ be given, and choose
$p'\in[p,\infty)$ such that $4/p'<\varepsilon$.
Applying \cref{lem:Sh-ast-bound}(ii) with $(\theta,p,q)=(-2+4/p',p',p')$,
we obtain
\[
\begin{aligned}
\|z_h\|_{L^{p'}(\Omega;C([0,T];L^\infty(\mathcal{O})))}
&\lesssim \|(A_hP_h - P_hA)G\|_{L^{p'}([0,T]\times\Omega;H_{D,h}^{-2+4/p',p'})} \\
&= \|(P_h - A_h^{-1}P_hA)G\|_{L^{p'}([0,T]\times\Omega;H_{D,h}^{4/p',p'})} \\
&\stackrel{(a)}{\sim} \|(P_h - A_h^{-1}P_hA)G\|_{L^{p'}([0,T]\times\Omega;H_{D}^{4/p',p'})} \\
&\stackrel{(b)}{\lesssim} h^{1-4/p'}\,\|G\|_{L^{p'}([0,T]\times\Omega;H_{D}^{1,p'})} \\
&\stackrel{(c)}{\lesssim} h^{1-\varepsilon}\,\|G\|_{L^{p'}([0,T]\times\Omega;H_{D}^{1,p'})} \\
&\stackrel{(d)}{\lesssim} h^{1-\varepsilon},
\end{aligned}
\]
where step~(a) uses the norm equivalence
$\|\cdot\|_{H_{D,h}^{4/p',p'}} \sim \|\cdot\|_{H_{D}^{4/p',p'}}$ on $X_h$ (\cref{lem:dotHh-equiv});
step~(b) is a consequence of \cref{eq:Ph-AhinvPhA} applied with
$(\theta_1,\theta_2)=(1,4/p')$; step~(c) exploits the elementary inequality
$h^{1-4/p'}\leqslant h^{1-\varepsilon}$, valid for $h\in(0,1]$ since
$4/p'<\varepsilon$; and step~(d) follows from the regularity bound for $G$
provided by \cref{eq:G-LpH1q} with $(p,q)=(p',p')$.
Since $p'\geqslant p$, the continuous embedding
$L^{p'}(\Omega)\hookrightarrow L^p(\Omega)$ yields
\[
\|z_h\|_{L^{p}(\Omega;C([0,T];L^\infty(\mathcal{O})))}
\lesssim h^{1-\varepsilon}.
\]
Furthermore, by \cref{eq:Ph-conv-Linfty} with
$(\theta,q)=(1-\varepsilon/2,2/\varepsilon)$
and the regularity estimate \cref{eq:G-LpCHq} with $(\theta,q)=(1-\varepsilon/2,2/\varepsilon)$,
we deduce that
\[
\|(I-P_h)G\|_{L^{p}(\Omega;C([0,T];L^\infty(\mathcal{O})))}
\lesssim h^{1-\varepsilon}\,
\|G\|_{L^{p}(\Omega;C([0,T];H_{D}^{1-\varepsilon/2,2/\varepsilon}))}
\lesssim h^{1-\varepsilon}.
\]
Combining the preceding two estimates via the triangle inequality establishes
\cref{eq:G-Gh-LpCLinf} and thus completes the proof.
\end{proof}



\begin{remark}
  To place our results in perspective, we recall that for the $P_1$ finite element method,
  Yan \cite[Theorems~1.1 and 1.2]{Yan2004semidiscrete} established the following convergence
  rates under the present setting:
  \begin{align*}
    \sup_{t \in [0,T]} \|G(t) - G_h(t)\|_{L^2(\Omega;L^2(\mathcal{O}))} \lesssim h,
    \quad \sup_{t \in [0,T]} \|G(t) - G_h(t)\|_{L^2(\Omega;H_{D}^{-1,2})} \lesssim \log(1+1/h) \, h^2.
  \end{align*}
  Although the present article employs the $P_2$ finite element method,
  the estimates \cref{eq:G-Gh-LpLq,eq:G-Gh-LpCLq} for $\alpha=0$ and \cref{eq:G-Gh-LpCLinf} remain valid
  for the $P_1$ finite element method.
\end{remark}

\subsubsection{Two stability estimates}

This subsection presents two stability estimates that will be used in the proofs of
\cref{thm:uh-strong} in this section and \cref{thm:uh-U} in \cref{sec:full-discr}.

The first estimate concerns the spatial derivative of a product and is instrumental
in controlling the nonlinear convection term in the subsequent error analysis.

\begin{lemma}
  \label{lem:Ph-partialx-vw-H2}
  Let $q \in (1,\infty)$ and let $q'$ denote its Hölder conjugate exponent.
  Then, for any $v \in H_{D}^{1,q'}$ and $w \in L^q(\mathcal{O})$, the following estimate holds:
  \[
    \bigl\|P_h \partial_x(vw)\bigr\|_{H_{D,h}^{-2,q}}
    \;\lesssim\; \|v\|_{H_{D}^{1,q'}}
    \bigl(\|w\|_{H_{D}^{-1,q}} + h\, \|w\|_{L^q(\mathcal{O})}\bigr).
  \]
\end{lemma}

\begin{proof}
  Let $\langle \cdot, \cdot \rangle$ denote the duality pairing between $L^q(\mathcal{O})$ and $L^{q'}(\mathcal{O})$.
  For an arbitrary $g \in L^{q'}(\mathcal{O})$, since both $A_h^{-1} P_h \partial_x(vw)$ and $A_h^{-1} P_h g$ belong to the finite element space $X_h$, we have
  \[
    \langle A_h^{-1}P_h\partial_x(vw), \, g \rangle =
    \langle A_h^{-1}P_h\partial_x(vw), \, P_hg \rangle =
    \langle P_h\partial_x(vw), \, A_h^{-1} P_hg \rangle =
    \langle \partial_x(vw), \, A_h^{-1} P_hg \rangle.
  \]
  Here, the first and third equalities follow from the definition (or orthogonality property) of the projection $P_h$ extended by duality, and the second equality follows from the self-adjointness of $A_h^{-1}$ on $X_h$.
  Integration by parts yields
  \[
   \langle A_h^{-1}P_h\partial_x(vw), \, g \rangle = 
  -\langle vw, \, \partial_x A_h^{-1} P_hg \rangle.
  \]
  Decomposing $A_h^{-1} = A^{-1} + (A_h^{-1} - A^{-1})$ and using the identity $ P_h^2 = P_h $,
  we decompose this expression into 
  \begin{align*}
    \bigl\langle A_h^{-1} P_h\, \partial_x(v w),\, g \bigr\rangle
    = -\bigl\langle w,\, v\,\partial_x A^{-1} P_h g \bigr\rangle
       -\bigl\langle v w,\, \partial_x(A_h^{-1} P_h - A^{-1}) P_h g \bigr\rangle.
  \end{align*}
  Consequently, by the duality between $H_{D}^{-1,q}$ and $H_{D}^{1,q'}$ along with H\"older's inequality, we obtain 
  \begin{equation}
    \label{eq:731}
    \bigl| \bigl\langle A_h^{-1} P_h\, \partial_x(v w),\, g \bigr\rangle \bigr|
    \leqslant \|w\|_{H_{D}^{-1,q}}\,
               \|v\,\partial_x A^{-1} P_h g\|_{H_{D}^{1,q'}}
             + \|v w\|_{L^q(\mathcal{O})}\,
               \|\partial_x(A_h^{-1}P_h - A^{-1}) P_h g\|_{L^{q'}(\mathcal{O})}.
  \end{equation}
  To bound the first term, we apply the Leibniz rule, followed by the Sobolev embeddings $H_{D}^{1,q'} \hookrightarrow L^{\infty}(\mathcal{O})$
  and $H_{D}^{2,q'} \hookrightarrow W^{1,\infty}(\mathcal{O})$, together with the
  isometric isomorphism $A^{-1} \colon L^{q'}(\mathcal{O}) \to H_{D}^{2,q'}$.
  This yields
  \begin{align*}
    \|v\,\partial_x A^{-1} P_h g\|_{H_{D}^{1,q'}}
    &\lesssim \|v\|_{L^{\infty}(\mathcal{O})}\,\|A^{-1} P_h g\|_{H_{D}^{2,q'}}
             + \|v\|_{H_{D}^{1,q'}}\,\|\partial_x A^{-1} P_h g\|_{L^{\infty}(\mathcal{O})} \\
    &\lesssim \|v\|_{H_{D}^{1,q'}}\,\|P_h g\|_{L^{q'}(\mathcal{O})}.
  \end{align*}
  For the second term, H\"older's inequality and the embedding $H_{D}^{1,q'} \hookrightarrow L^{\infty}(\mathcal{O})$
  imply $\|v w\|_{L^q(\mathcal{O})} \leqslant \|v\|_{L^{\infty}}\,\|w\|_{L^q(\mathcal{O})} \lesssim \|v\|_{H_{D}^{1,q'}}\,\|w\|_{L^q(\mathcal{O})}$.
  Applying \cref{lem:Ah}(i) with $(\theta,q) = (1,q')$ to control the remaining factor, we deduce that
  \begin{align*}
    \|v w\|_{L^q(\mathcal{O})}\,
    \|\partial_x(A_h^{-1}P_h - A^{-1}) P_h g\|_{L^{q'}(\mathcal{O})}
    \;\lesssim\; h\,\|v\|_{H_{D}^{1,q'}}\,\|w\|_{L^q(\mathcal{O})}\,\|P_h g\|_{L^{q'}(\mathcal{O})}.
  \end{align*}
  Finally, inserting these bounds into \eqref{eq:731} and invoking the $L^{q'}(\mathcal{O})$-stability of $P_h$ (see \eqref{eq:Ph-stab0}, which gives $\|P_h g\|_{L^{q'}(\mathcal{O})} \lesssim \|g\|_{L^{q'}(\mathcal{O})}$), we conclude that
  \[
    \bigl|\bigl\langle A_h^{-1} P_h\, \partial_x(v w),\, g \bigr\rangle\bigr|
    \;\lesssim\; \|v\|_{H_{D}^{1,q'}}
    \bigl(\|w\|_{H_{D}^{-1,q}} + h\,\|w\|_{L^q(\mathcal{O})}\bigr)\,\|g\|_{L^{q'}(\mathcal{O})}.
  \]
  Since $g \in L^{q'}(\mathcal{O})$ was arbitrary, it follows that 
  \[
   \|A_h^{-1}P_h\partial_x(vw)\|_{L^q(\mathcal{O})}
    \;\lesssim\; \|v\|_{H_{D}^{1,q'}}
    \bigl(\|w\|_{H_{D}^{-1,q}} + h\,\|w\|_{L^q(\mathcal{O})}\bigr).
  \]
 The identity $ \|A_h^{-1}P_h\partial_x(vw)\|_{L^q(\mathcal{O})} = \|P_h\partial_x(vw)\|_{H_{D,h}^{-2,q}} $
 completes the proof.
\end{proof}

The second estimate addresses the stability of solutions to a random integral equation.

\begin{lemma}
  \label{lem:core-stability}
  Let $p \in (1,\infty)$ and $p_1 \in (4,\infty)$. 
  Assume that $g \in L^{8p}\bigl(\Omega; L^\infty(0,T; L^\infty(\mathcal{O}))\bigr)$ 
  is such that $g \in L^2(0,T; H_D^{1,2})$ $\mathbb{P}$-a.s., 
  and that there exists a constant $\kappa_0 > 0$ for which
  \[
    \mathbb{E}\!\left[ \exp\!\left( \kappa_0 \| g \|_{L^2(0,T;H_D^{1,2})}^2 \right) \right] < \infty.
  \]
  Furthermore, let $\sigma_h \in L^{2p}(\Omega; L^{p_1}(0,T; H_{D,h}^{0,2})) \cap L^{8p}(\Omega;L^{2}(0,T; H_{D,h}^{0,2}))$. 
  Suppose that $z_h \colon \Omega \to C([0,T];X_h)$ is strongly $\mathcal{F}$-measurable
  and satisfies, $\mathbb{P}$-a.s., for all $t \in [0,T]$,
  \begin{equation}
    \label{eq:zh-equ}
    z_h(t) = \int_0^t \left( A_h z_h(s) - \frac{1}{2} P_h \partial_x \bigl( g(s)(z_h(s) + \sigma_h(s)) \bigr) \right)
    \, \mathrm{d}s.
  \end{equation}
  Then, the following estimate holds:
  \begin{equation}
    \label{eq:core-estimate}
    \begin{aligned}
      \|z_h\|_{L^{p}(\Omega;\, C([0,T];\, L^\infty(\mathcal{O})))}
      &\lesssim \|g\|_{L^{2p}(\Omega;\, L^{\infty}(0,T;\, L^\infty(\mathcal{O})))}
      \Bigl[ \|\sigma_h\|_{L^{2p}(\Omega;L^{p_1}(0,T;L^2(\mathcal{O})))} \\
      &\qquad + \exp\!\Bigl(\frac{p^2T}{128\kappa_0^2}\Bigr)
        \Bigl( \mathbb{E}\Bigl[ \exp\!\bigl(\kappa_0\|g\|_{L^2(0,T;H_{D}^{1,2})}^2\bigr) \Bigr] \Bigr)^{\!\frac{1}{4p}} \\
      &\qquad \times \|g\|_{L^{8p}(\Omega;L^\infty(0,T;L^\infty(\mathcal{O})))}
        \|\sigma_h\|_{L^{8p}(\Omega;L^2(0,T;L^2(\mathcal{O})))} \Bigr].
    \end{aligned}
  \end{equation}
\end{lemma}

\begin{proof}
Applying the variation of constants formula, we obtain the representation $\mathbb{P}$-a.s.
\[
z_h = -\frac12 S_h \ast \bigl(P_h \partial_x (g z_h)\bigr) -\frac12 S_h \ast \bigl(P_h \partial_x (g \sigma_h)\bigr).
\]
We first estimate the term involving $z_h$.  Lemma~\ref{lem:Sh-ast-bound}(iii) with $p=\infty$ and $z=z_h$
(noting the embedding $C([0,T];L^2(\mathcal{O})) \hookrightarrow L^\infty(0,T;L^2(\mathcal{O}))$) yields $\mathbb{P}$-a.s.
\[
\frac12 \bigl\| S_h \ast \bigl(P_h \partial_x (g z_h)\bigr) \bigr\|_{C([0,T];L^\infty(\mathcal{O}))}
\lesssim \|g\|_{L^{\infty}(0,T;L^\infty(\mathcal{O}))} \|z_h\|_{C([0,T];L^2(\mathcal{O}))}.
\]
Taking the $L^{p}(\Omega)$-norm and applying Hölder's inequality, we deduce
\[
\frac12 \bigl\| S_h \ast \bigl(P_h \partial_x (g z_h)\bigr) \bigr\|_{L^{p}(\Omega;C([0,T];L^\infty(\mathcal{O})))}
\lesssim \|g\|_{L^{2p}(\Omega;L^\infty(0,T;L^\infty(\mathcal{O})))}
\|z_h\|_{L^{2p}(\Omega;C([0,T];L^2(\mathcal{O})))}.
\]
For the term containing $\sigma_h$, we invoke Lemma~\ref{lem:Sh-ast-bound}(iii) with $p=p_1$ and $z=\sigma_h$.
This gives, $\mathbb{P}$-a.s.,
\[
\bigl\| S_h \ast (P_h\partial_x(g \sigma_h)) \bigr\|_{C([0,T];L^\infty(\mathcal{O}))}
\lesssim \|g\|_{L^\infty(0,T;L^\infty(\mathcal{O}))}
\|\sigma_h\|_{L^{p_1}(0,T;L^2(\mathcal{O}))}.
\]
Applying the $L^{p}(\Omega)$-norm and Hölder's inequality once more, we obtain
\[
\bigl\| S_h \ast (P_h\partial_x(g \sigma_h)) \bigr\|_{L^p(\Omega;C([0,T];L^\infty(\mathcal{O})))}
\lesssim \|g\|_{L^{2p}(\Omega;L^\infty(0,T;L^\infty(\mathcal{O})))}
\|\sigma_h\|_{L^{2p}(\Omega;L^{p_1}(0,T;L^2(\mathcal{O})))}.
\]
Combining the two estimates and using the triangle inequality in the expression for $z_h$, we arrive at
\begin{align*}
& \|z_h\|_{L^{p}(\Omega;C([0,T];L^\infty(\mathcal{O})))} \\
\lesssim{} & \|g\|_{L^{2p}(\Omega;L^{\infty}(0,T;L^\infty(\mathcal{O})))}
\bigl(
\|z_h\|_{L^{2p}(\Omega;C([0,T];L^2(\mathcal{O})))}
+ \|\sigma_h\|_{L^{2p}(\Omega;L^{p_1}(0,T;L^2(\mathcal{O})))}
\bigr).
\end{align*}
Consequently, the proof is reduced to establishing the estimate
\begin{equation}
\label{eq:zh-bound0}
\begin{aligned}
\|z_h\|_{L^{2p}(\Omega;C([0,T];L^2(\mathcal{O})))}
&\leqslant \frac{1}{2} \exp\!\Bigl(\frac{p^2T}{128\kappa_0^2}\Bigr) 
\Bigl( \mathbb{E}\Bigl[ \exp\!\bigl(\kappa_0\|g\|_{L^2(0,T;H_{D}^{1,2})}^2\bigr) \Bigr] \Bigr)^{\!\frac{1}{4p}} \\
& \qquad \times \|g\|_{L^{8p}(\Omega;L^\infty(0,T;L^\infty(\mathcal{O})))}
\|\sigma_h\|_{L^{8p}(\Omega;L^2(0,T;L^2(\mathcal{O})))}.
\end{aligned}
\end{equation}

To this end, we perform an energy estimate.
Equation \cref{eq:zh-equ} implies that there exists a $\mathbb{P}$-null set $N$ such that, 
for all $\omega \in \Omega \setminus N$, $z_h(0) = 0$ and
\[
  \frac{\mathrm{d}}{\mathrm{d}t} z_h(t) = A_h z_h(t) - \frac{1}{2} P_h \partial_x \bigl( g(t)(z_h(t) + \sigma_h(t)) \bigr),
  \quad \text{a.e.~} t \in [0,T].
\]
Testing this equation against $z_h$ in $L^2(\mathcal{O})$, and employing the identity
$\langle A_h v_h, v_h\rangle = -\|v_h\|_{H_{D}^{1,2}}^2$ together with the $L^2$-orthogonality
of $P_h$, we obtain for all $\omega \in \Omega \setminus N$ that, for a.e.~$t \in [0,T]$,
\begin{equation}
  \label{eq:zh-differential}
  \frac{1}{2}\frac{\mathrm{d}}{\mathrm{d}t} \|z_h(t)\|_{L^2(\mathcal{O})}^2
  = -\|z_h(t)\|_{H_{D}^{1,2}}^2
  - \frac{1}{2} \bigl\langle \partial_x(g(t)z_h(t)),\, z_h(t)\bigr\rangle
  - \frac{1}{2} \bigl\langle \partial_x(g(t)\sigma_h(t)),\, z_h(t)\bigr\rangle,
\end{equation}
where $\langle\cdot,\cdot\rangle$ denotes the $L^2(\mathcal{O})$-inner product.
To estimate the second term on the right-hand side, we invoke the identity
$\langle\partial_x(gz_h), z_h\rangle = \frac{1}{2}\langle\partial_xg, z_h^2\rangle$ in conjunction with Hölder's inequality and the one-dimensional Gagliardo-Nirenberg inequality
$\|v\|_{L^\infty(\mathcal{O})}\leqslant\sqrt{2}\|v\|_{L^2(\mathcal{O})}^{1/2}\|v\|_{H_{D}^{1,2}}^{1/2}$
for $v \in H_{D}^{1,2}$.
This yields
\begin{align*}
-\frac{1}{2}\langle\partial_x(gz_h), z_h\rangle
&= -\frac{1}{4}\langle\partial_x g, z_h^2\rangle \\
& \leqslant \frac{1}{4} \|\partial_xg\|_{L^2(\mathcal{O})}
\|z_h\|_{L^2(\mathcal{O})} \|z_h\|_{L^\infty(\mathcal{O})} \\
& \leqslant \frac{\sqrt{2}}{4} \|\partial_xg\|_{L^2(\mathcal{O})}
\|z_h\|_{L^2(\mathcal{O})}^{3/2} \|z_h\|_{H_{D}^{1,2}}^{1/2} \\
&= \Big( \sqrt{\frac{\kappa_0}{2p}} \|g\|_{H_{D}^{1,2}} \|z_h\|_{L^2(\mathcal{O})} \Big)
\times \Big( \frac{1}{ 2^{\frac{5}{4}} \sqrt{\frac{\kappa_0}{p}}}
\|z_h\|_{L^2(\mathcal{O})}^{1/2} \Big)
\times \Big( 2^{\frac{1}{4}} \|z_h\|_{H_{D}^{1,2}}^{1/2} \Big),
\end{align*}
where we have used the equality $\|g\|_{H_{D}^{1,2}} = \|\partial_xg\|_{L^2(\mathcal{O})}$.
Applying the generalized Young's inequality, $abc \leqslant a^2/2 + b^4/4 + c^4/4$, we arrive at
\[
  -\frac{1}{2}\bigl\langle \partial_x(gz_h), z_h\bigr\rangle
  \leqslant \Bigl( \frac{\kappa_0}{4p} \|g\|_{H_{D}^{1,2}}^2 + \frac{p^2}{128\kappa_0^2} \Bigr) \|z_h\|_{L^2(\mathcal{O})}^2
  + \frac{1}{2}\|z_h\|_{H_{D}^{1,2}}^2 .
\]
For the third term, integration by parts followed by Hölder's and Young's inequalities implies
\begin{align*}
  -\frac{1}{2}\bigl\langle \partial_x(g\sigma_h),\, z_h \bigr\rangle
  &=\frac{1}{2}\bigl\langle g\sigma_h,\, \partial_x z_h \bigr\rangle \\
  &\leqslant \frac{1}{2} \|g\|_{L^\infty(\mathcal{O})} \|\sigma_h\|_{L^2(\mathcal{O})} \|z_h\|_{H_{D}^{1,2}} \\
  &\leqslant \frac{1}{8} \|g\|_{L^\infty(\mathcal{O})}^2 \|\sigma_h\|_{L^2(\mathcal{O})}^2
  + \frac{1}{2} \|z_h\|_{H_{D}^{1,2}}^2 .
\end{align*}
Substituting these estimates into \cref{eq:zh-differential} and observing that the $H_{D}^{1,2}$-norm terms on
the right-hand side cancel out, we deduce for all $\omega \in \Omega \setminus N$ that,
for a.e.~$t \in [0,T]$,
\begin{align*}
  \frac{1}{2} \frac{\mathrm{d}}{\mathrm{d}t} \|z_h(t)\|_{L^2(\mathcal{O})}^2
  \leqslant \frac{1}{8} \|g(t)\|_{L^\infty(\mathcal{O})}^2 \|\sigma_h(t)\|_{L^2(\mathcal{O})}^2
  + \Bigl( \frac{\kappa_0}{4p} \|g(t)\|_{H_{D}^{1,2}}^2 + \frac{p^2}{128\kappa_0^2} \Bigr)
  \|z_h(t)\|_{L^2(\mathcal{O})}^2.
\end{align*}
Multiplying both sides by $2$, integrating over $[0,t]$, and using the initial condition $z_h(0)=0$,
we obtain for all $\omega \in \Omega \setminus N$ and $t \in [0,T]$ that
\[
  \|z_h(t)\|_{L^2(\mathcal{O})}^2
  \leqslant \int_0^t \biggl[ \frac{1}{4}\|g(s)\|_{L^\infty(\mathcal{O})}^2 \|\sigma_h(s)\|_{L^2(\mathcal{O})}^2
    + \Bigl( \frac{\kappa_0}{2p} \|g(s)\|_{H_{D}^{1,2}}^2 + \frac{p^2}{64\kappa_0^2} \Bigr)
  \|z_h(s)\|_{L^2(\mathcal{O})}^2 \biggr] \mathrm{d}s.
\]
An application of Gronwall's inequality implies $\mathbb{P}$-a.s.~that
\[
  \|z_h\|_{C([0,T];L^2(\mathcal{O}))}^2
  \leqslant \frac{1}{4} \exp\!\biggl(\frac{p^2T}{64\kappa_0^2} + \frac{\kappa_0}{2p} \|g\|_{L^2(0,T;H_{D}^{1,2})}^2\biggr)
  \|g\|_{L^\infty(0,T;L^\infty(\mathcal{O}))}^2 \|\sigma_h\|_{L^2(0,T;L^2(\mathcal{O}))}^2 .
\]
Taking the $p$-th power and the expectation, and applying Hölder's inequality, we find
\begin{align*}
  \|z_h\|_{L^{2p}(\Omega;C([0,T];L^2(\mathcal{O})))}^{2p}
  & \leqslant \frac{1}{4^p} \exp\!\biggl(\frac{p^3T}{64\kappa_0^2}\biggr)
  \left( \mathbb{E} \left[ \exp\!\left(\kappa_0 \|g\|_{L^2(0,T;H_{D}^{1,2})}^2\right) \right] \right)^{\!\frac{1}{2}} \\
  & \quad \times \|g\|_{L^{8p}(\Omega;L^\infty(0,T;L^\infty(\mathcal{O})))}^{2p}
  \|\sigma_h\|_{L^{8p}(\Omega;L^2(0,T;L^2(\mathcal{O})))}^{2p}.
\end{align*}
Finally, taking the $2p$-th root yields \cref{eq:zh-bound0}, which completes the proof of the lemma.
\end{proof}

\subsection{Proof of \texorpdfstring{\cref{prop:uh-regu}}{Proposition~\ref{prop:uh-regu}}}
\label{ssec:proof-uh-well-posed}
The spatial semi-discretization \eqref{eq:uh} is a finite-dimensional stochastic differential equation (SDE) on $X_h$ driven by additive noise. Since $X_h$ is finite-dimensional, the drift operator $v_h \mapsto A_h v_h - \frac12 P_h\partial_x(v_h^2)$ is a polynomial mapping with respect to the degrees of freedom, and hence locally Lipschitz continuous. Furthermore, the operator satisfies the following dissipativity identity:
\begin{equation}
  \label{eq:key-identity}
  \langle A_h v_h - \tfrac12 P_h\partial_x(v_h^2), v_h \rangle = -\|v_h\|_{H_{D,h}^{1,2}}^2
  \quad \text{for all } v_h \in X_h,
\end{equation}
where $\langle\cdot, \cdot \rangle$ denotes the $L^2(\mathcal{O})$ inner product. To establish \eqref{eq:key-identity}, observe that $\langle A_h v_h, v_h \rangle = -\|v_h\|_{H_{D,h}^{1,2}}^2$ by definition, whereas
\[
\langle P_h\partial_x(v_h^2), v_h \rangle = \langle \partial_x(v_h^2), v_h \rangle = \frac{2}{3} \int_{\mathcal{O}} \partial_x(v_h^3)\,\mathrm{d}x = 0,
\]
where the first equality uses that $P_h$ is the $L^2$-orthogonal projection onto $X_h$ and $v_h \in X_h$,
and the last follows from the fundamental theorem of calculus together with the homogeneous Dirichlet boundary conditions $v_h(0)=v_h(1)=0$.
In addition, the bound \eqref{eq:PhQ-bound} implies that $P_h Q \in \gamma(\ell^2,H_{D,h}^{0,2})$.
Consequently, \cite[Theorem~3.27]{Pardoux2014} ensures the existence of a unique solution $u_h$ to \eqref{eq:uh}. Finally, the exponential moment estimate \eqref{eq:uh-exp-moment} follows by adapting the arguments established in \cite[Lemma~3.1]{BrehierCoxMillet2026} for semi-discrete spectral Galerkin approximations.

It remains to establish \cref{eq:uh-LpH1q,eq:uh-LpCHq}. The proof parallels that of \cref{prop:u-regu}
and proceeds in three steps.

\textbf{Step 1: Preliminary estimates.}
Using the deterministic initial datum $u_0 \in W_0^{1,\infty}(\mathcal{O})$ 
and the identity \cref{eq:key-identity},
a routine application of It\^o's formula and the Burkholder--Davis--Gundy inequality yields
(cf.~\cite[Lemma~3.1]{BrehierCoxMillet2026} for the spectral Galerkin approximations)
\begin{equation} \label{eq:uh-LpCL2}
  \|u_h\|_{L^p(\Omega;C([0,T]; L^2(\mathcal{O})))} \lesssim 1
  \quad \text{for all } p \in [2,\infty).
\end{equation}
Furthermore, since $u_0 \in W_0^{1,\infty}(\mathcal{O}) \hookrightarrow H_{D}^{1,q}$ for all $q \in [2,\infty)$,
we deduce from the $h$-uniform boundedness of $P_h$ in $\mathcal{L}(H_{D}^{1,q}, H_{D,h}^{1,q})$ (\cref{lem:Ph-stab}),
the estimate \cref{eq:Sh-smoothing} with $\theta_1=\theta_2=1$,
and the analyticity of $S_h$ on $H_{D,h}^{1,q}$ that
\[
 \|S_h(\cdot)P_hu_0 \|_{C([0,T];H_{D,h}^{1,q})} \lesssim 1 \quad \text{for all } q \in [2,\infty).
\]
In conjunction with the norm equivalence $\|\cdot\|_{H_{D,h}^{1,q}} \sim \|\cdot\|_{H_{D}^{1,q}}$ on $X_h$ (\cref{lem:dotHh-equiv}), this yields
\begin{equation}
 \label{eq:ShPhu0}
 \|S_h(\cdot)P_hu_0 \|_{C([0,T];H_{D}^{1,q})} \lesssim 1 \quad \text{for all } q \in [2,\infty).
\end{equation}

\textbf{Step 2: Bootstrap argument.}
It holds that, $\mathbb{P}$-a.s.,
\begin{equation} \label{eq:uh-mild}
u_h(t) = S_h(t)P_hu_0 - \frac{1}{2} (S_h \ast (P_h\partial_x(u_h^2)))(t) + G_h(t)
\quad\text{for all } t \in [0,T],
\end{equation}
where $G_h$ is defined in \cref{eq:Gh-def}.
By \cref{lem:dotHh}(i), the embedding $H_{D,h}^{3/2+\delta,2} \hookrightarrow W^{1,\infty}(\mathcal{O})$ holds uniformly
in $h$; a duality argument therefore yields
\[
\|P_h \partial_x (v_h^2)\|_{H_{D,h}^{-3/2-\delta,2}} \lesssim \|v_h\|_{L^2(\mathcal{O})}^2
\quad\text{for all } v_h \in X_h \text{ and } \delta \in (0,1/2).
\]
Since the mapping $X_h \ni v_h \mapsto P_h\partial_x(v_h^2) \in X_h$ is locally Lipschitz continuous, this bound combined with \cref{eq:uh-LpCL2} implies
\begin{equation}
  \label{eq:uh2-bound0}
\|P_h\partial_x(u_h^2)\|_{L^p(\Omega;C([0,T];H_{D,h}^{-3/2-\delta,2}))} \lesssim 1
\quad \text{for all } p \in [2,\infty) \text{ and } \delta \in (0,1/2).
\end{equation}
Moreover, by \cref{lem:Sh-ast-bound}(i) (applied with $p=\infty$ and $(\theta_1,\theta_2)=(-3/2-\delta,1/2-\varepsilon)$)
and the norm equivalence $H_{D,h}^{1/2-\varepsilon,2} \sim H_{D}^{1/2-\varepsilon,2}$
on $X_h$ (\cref{lem:dotHh-equiv}), we have, $\mathbb{P}$-a.s.,
\[
\|S_h \ast P_h\partial_x(u_h^2)\|_{C([0,T];H_{D}^{1/2-\varepsilon,2})}
\lesssim 
\|P_h\partial_x(u_h^2)\|_{C([0,T];H_{D,h}^{-3/2-\delta,2})}
\quad\text{for all } 0 < \delta < \varepsilon < 1/2.
\]
Choosing $\delta \in (0,\varepsilon)$ and invoking \cref{eq:uh2-bound0}, we deduce that
\begin{equation}
  \label{eq:Sh-ast-uh2-0}
  \|S_h \ast P_h\partial_x(u_h^2)\|_{L^p(\Omega;C([0,T];H_{D}^{1/2-\varepsilon,2}))}
  \lesssim 1
  \quad \text{for all } p \in [2,\infty) \text{ and } \varepsilon \in (0,1/2).
\end{equation}
In view of the mild formulation \cref{eq:uh-mild}, this estimate together with \cref{eq:Gh-Lp-C-Hthetaq,eq:ShPhu0} yields
\[
\|u_h\|_{L^p(\Omega;C([0,T];H_{D}^{1/2-\varepsilon,2}))} \lesssim 1
\quad \text{for all } p \in [2,\infty) \text{ and } \varepsilon \in (0,1/2),
\]
and hence, by the Sobolev embedding theorem,
\[
\|u_h\|_{L^p(\Omega;C([0,T];L^q(\mathcal{O})))} \lesssim 1
\quad \text{for all } p,q \in [2,\infty).
\]
Arguing as in the derivation of \cref{eq:u2-bound}, we obtain
\[
\|\partial_x(u_h^2)\|_{L^p(\Omega;C([0,T];H_{D}^{-1,q}))} \lesssim 1
\quad \text{for all } p,q \in [2,\infty),
\]
and the $h$-uniform stability bound $\|P_h\|_{\mathcal{L}(H_{D}^{-1,q},H_{D,h}^{-1,q})} \lesssim 1$ from \cref{lem:Ph-stab-negative} then gives
\[
\|P_h\partial_x(u_h^2)\|_{L^p(\Omega;C([0,T];H_{D,h}^{-1,q}))} \lesssim 1
\quad \text{for all } p,q \in [2,\infty).
\]
Repeating the argument leading to \cref{eq:Sh-ast-uh2-0}, but now applying \cref{lem:Sh-ast-bound}(i) with $p=\infty$ and $(\theta_1,\theta_2) = (-1,4/5)$ and using the norm equivalence $\|\cdot\|_{H_{D,h}^{4/5,q}} \sim \|\cdot\|_{H_{D}^{4/5,q}}$ on $X_h$ from \cref{lem:dotHh-equiv}, we conclude that
\[
\|S_h \ast P_h\partial_x(u_h^2)\|_{L^p(\Omega;C([0,T];H_{D}^{4/5,q}))}
\lesssim 1 \quad \text{for all } p,q \in [2,\infty).
\]
Combining this estimate with \cref{eq:Gh-Lp-C-Hthetaq,eq:ShPhu0}
and invoking \cref{eq:uh-mild}, we obtain
\[
\|u_h\|_{L^p(\Omega;C([0,T];H_{D}^{4/5,q}))} \lesssim 1
\quad \text{for all } p,q \in [2,\infty).
\]
Arguing as in the derivation of \cref{eq:partialxu2}, we deduce that
\[
\|\partial_x(u_h^2)\|_{L^p(\Omega;C([0,T];H_{D}^{-2/5,q}))} \lesssim 1
\quad \text{for all } p,q \in [2,\infty),
\]
and the $h$-uniform stability bound
$\|P_h\|_{\mathcal{L}(H_{D}^{-2/5,q},H_{D,h}^{-2/5,q})} \lesssim 1$
from \cref{lem:Ph-stab-negative} then gives
\[
\|P_h\partial_x(u_h^2)\|_{L^p(\Omega;C([0,T];H_{D,h}^{-2/5,q}))} \lesssim 1
\quad \text{for all } p,q \in [2,\infty).
\]
Finally, repeating the argument leading to \cref{eq:Sh-ast-uh2-0}, but
applying \cref{lem:Sh-ast-bound}(i) with $p=\infty$ and
$(\theta_1,\theta_2) = (-2/5,1)$ and using the norm equivalence
$\|\cdot\|_{H_{D,h}^{1,q}} \sim \|\cdot\|_{H_{D}^{1,q}}$ on $X_h$ from
\cref{lem:dotHh-equiv}, we conclude that
\begin{equation}
\label{eq:Sh-ast-uh2}
\|S_h \ast (P_h\partial_x(u_h^2))\|_{L^p(\Omega;C([0,T];H_{D}^{1,q}))}
\lesssim 1 \quad \text{for all } p,q \in [2,\infty).
\end{equation}

\textbf{Step 3: Conclusion.}
In view of the mild formulation \cref{eq:uh-mild}, the estimate \cref{eq:uh-LpH1q}
for all $p,q \in [2,\infty)$ follows from the triangle inequality,
the bounds \cref{eq:Gh-Lp-H1q,eq:ShPhu0,eq:Sh-ast-uh2},
and the embedding $L^p(\Omega;C([0,T];H_{D}^{1,q})) \hookrightarrow L^p([0,T] \times \Omega; H_{D}^{1,q})$. 
Analogously, for all $p,q \in [2,\infty)$ and $\theta \in (0,1)$,
the estimate \cref{eq:uh-LpCHq} follows from \cref{eq:Gh-Lp-C-Hthetaq,eq:ShPhu0,eq:Sh-ast-uh2}
combined with the embedding $H_{D}^{1,q} \hookrightarrow H_{D}^{\theta,q}$.

\qed

\subsection{Proof of Theorem~\ref{thm:uh-strong}}
\label{ssec:proof-uh-strong}
The error $u - u_h$ is decomposed into singular and regular components as follows:
\[
  u - u_h = \underbrace{(G - G_h)}_{\text{singular}} + \underbrace{(I - P_h)\xi - (\xi_h - P_h\xi)}_{\text{regular}},
\]
where $G$ and $G_h$ are defined in \cref{eq:G-def} and \cref{eq:Gh-def}, respectively,
and we set $\xi := u - G$ as in \cref{prop:u-regu} and $\xi_h := u_h - G_h$.
Combining the regularity of $\xi$ from \cref{eq:xi-LpCH1q} with the approximation properties of $P_h$ in \eqref{eq:Ph-conv}
and \eqref{eq:Ph-conv-Linfty}, we derive the following bounds for any $\alpha \in [-1,0]$, $p, q \in [2,\infty)$,
and $\varepsilon \in (0,1)$:
\begin{align}
  \|(I - P_h)\xi\|_{L^p(\Omega; C([0,T]; H_{D}^{\alpha,q}))} & \lesssim h^{1-\alpha}, \label{eq:xi-Phxi-LpC} \\
  \|(I - P_h)\xi\|_{L^p(\Omega; C([0,T]; L^\infty(\mathcal{O})))} & \lesssim h^{1-\varepsilon}. \notag
\end{align}
In view of the estimates for $G - G_h$ in \cref{lem:G-Gh}
and the embedding $L^p(\Omega;C([0,T];H_{D}^{\alpha,q})) \hookrightarrow L^p([0,T] \times \Omega;H_{D}^{\alpha,q})$,
the triangle inequality reduces the proof of \cref{thm:uh-strong} to establishing the following bounds
for all $\alpha \in [-1,0]$, $ p,q \in [2,\infty) $, and $ \varepsilon \in (0,1) $:
\begin{align}
  \|\xi_h - P_h\xi\|_{L^p([0,T] \times \Omega;H_{D}^{\alpha,q})} & \lesssim h^{1-\alpha}, \label{eq:xih-Phxi-bound1} \\
  \|\xi_h - P_h\xi\|_{L^p(\Omega; C([0,T]; H_{D}^{\alpha,q}))} & \lesssim h^{1-\alpha-\varepsilon}, \label{eq:xih-Phxi-LpCHq} \\
  \|\xi_h - P_h\xi\|_{L^p(\Omega;C([0,T];L^\infty(\mathcal{O})))} & \lesssim h^{1-\varepsilon}. \label{eq:xih-Phxi-bound2}
\end{align}
The remainder of the proof is organized into four steps to establish these estimates.

\medskip
\noindent\textbf{Step 1: Decomposition of the error \(\xi_h - P_h\xi\).}
We define the error process $\eta_h := \xi_h - P_h\xi$.
In view of \cref{eq:u-LpCHq} and the Sobolev embedding theorem, we have
$u \in C([0,T];L^\infty(\mathcal{O}))$ $\mathbb{P}$-a.s.
Consequently, $u^2 \in C([0,T];L^2(\mathcal{O}))$ and, since
$\partial_x \in \mathcal{L}(L^2(\mathcal{O}), H_{D}^{-1,2})$,
$\partial_x(u^2) \in C([0,T];H_{D}^{-1,2})$ $\mathbb{P}$-a.s.
Furthermore, the initial datum satisfies
$u_0 \in W_0^{1,\infty}(\mathcal{O}) \hookrightarrow H_{D}^{1,2}$.
In addition, \cref{eq:xi-LpCH1q} implies that
$\xi \in C([0,T];H_{D}^{1,2})$ $\mathbb{P}$-a.s.; since
$A \in \mathcal{L}(H_{D}^{1,2}, H_{D}^{-1,2})$, it follows that
$A\xi \in C([0,T];H_{D}^{-1,2})$ $\mathbb{P}$-a.s.
Combining these regularity properties with \cref{eq:xi-def},
we deduce that there exists a $\mathbb{P}$-null set $N$ such that,
for every $\omega \in \Omega \setminus N$, the path $t \mapsto \xi(t)$ is
continuously differentiable on $[0,T]$ with values in $H_{D}^{-1,2}$ and
satisfies $\xi(0) = u_0$ together with
\[
\frac{\mathrm{d}}{\mathrm{d}t}\,\xi(t) = A\xi(t) - \tfrac{1}{2}\,\partial_x\!\bigl(u^2(t)\bigr)
\quad \text{in } H_{D}^{-1,2} \text{ for all } t \in [0,T].
\]
(Here and in what follows, the $\mathbb{P}$-null set $N$ is enlarged as necessary,
independent of $t$, and we retain the same notation.)
Applying the projection $P_h$ to the preceding identity yields, for every
$\omega \in \Omega \setminus N$,
\[
\frac{\mathrm{d}}{\mathrm{d}t}\,P_h\xi(t) = P_hA\xi(t) - \tfrac{1}{2}\,P_h\partial_x\!\bigl(u^2(t)\bigr)
\quad \text{for all } t \in [0,T],
\qquad P_h\xi(0) = P_hu_0,
\]
and hence, upon integrating over $(0,t)$,
\begin{equation}
    \label{eq:Phxi}
P_h\xi(t) = P_hu_0 + \int_0^t \Bigl( P_hA\xi(s) - \tfrac{1}{2}\,P_h\partial_x\!\bigl(u^2(s)\bigr) \Bigr)\,\mathrm{d}s
\quad \text{for all } t \in [0,T].
\end{equation}
Next, recalling the definition of $G_h$ (see \cref{eq:Gh-def}), we have, for
every $\omega \in \Omega \setminus N$ and all $t \in [0,T]$,
\[
G_h(t) = \int_0^t A_h G_h(s)\,\mathrm{d}s + \int_0^t P_hQ\,\mathrm{d}W(s).
\]
Subtracting this equality from \cref{eq:uh} and using the 
relation $\xi_h = u_h - G_h$, we obtain, for every
$\omega \in \Omega \setminus N$ and all $t \in [0,T]$,
\[
\xi_h(t) = P_hu_0 + \int_0^t \Bigl( A_h\xi_h(s) - \tfrac{1}{2}\,P_h\partial_x\!\bigl(u_h^2(s)\bigr) \Bigr)\,\mathrm{d}s.
\]
Subtracting \cref{eq:Phxi} from the preceding identity, we conclude that,
for every $\omega \in \Omega \setminus N$ and all $t \in [0,T]$, 
\begin{equation}
  \label{eq:etah}
  \eta_h(t) = \int_0^t \Bigl( A_h\eta_h(s) + \bigl(A_hP_h\xi(s) - P_hA\xi(s)\bigr)
  + \tfrac{1}{2}\,P_h\partial_x\!\bigl[(u(s)+u_h(s))(u(s)-u_h(s))\bigr] \Bigr)\,\mathrm{d}s.
\end{equation}

To isolate the distinct sources of error, we substitute the decomposition 
$u(t) - u_h(t) = (G(t) - G_h(t)) + (I - P_h)\xi(t) - \eta_h(t)$ 
into the nonlinear term of \cref{eq:etah}. This motivates the introduction of three auxiliary
processes, $\chi_h^{(1)}$, $\chi_h^{(2)}$, and $\chi_h^{(3)}$,
which satisfy the following integral equations for every $\omega \in \Omega \setminus N$
and all $t \in [0,T]$:
\begin{align}
  \chi_h^{(1)}(t) &= \int_0^t \Bigl( A_h\chi_h^{(1)}(s) + A_h P_h\xi(s) - P_hA\xi(s) \Bigr) \, \mathrm{d}s, 
  \label{eq:chih1-def} \\
  \chi_h^{(2)}(t) &= \int_0^t \Bigl( A_h\chi_h^{(2)}(s) + \frac12 P_h\partial_x\!\Bigl[(u(s)+u_h(s))(\xi(s) - P_h\xi(s))\Bigr] \Bigr) \, \mathrm{d}s,  
  \label{eq:chih2-def} \\
  \chi_h^{(3)}(t) &= \int_0^t \Bigl( A_h\chi_h^{(3)}(s) + \frac12 P_h\partial_x\!\Bigl[(u(s)+u_h(s))(G(s) - G_h(s))\Bigr] \Bigr) \, \mathrm{d}s. 
  \label{eq:chih3-def}
\end{align}
Defining $\sigma_h := \chi_h^{(1)} + \chi_h^{(2)} + \chi_h^{(3)}$ and $z_h := \eta_h - \sigma_h$,
we subtract the sum of the equations for $\chi_h^{(i)}$ from \cref{eq:etah}.
Thus, $z_h$ satisfies the following integral equation 
for every $\omega \in \Omega \setminus N$ and all $t \in [0,T]$:
\begin{equation}\label{eq:Xih}
    z_h(t) = \int_0^t \Bigl( A_hz_h(s) - \frac12 P_h\partial_x\!\Bigl[(u(s)+u_h(s))\bigl(z_h(s) + \sigma_h(s)\bigr)\Bigr] \Bigr) \, \mathrm{d}s.
\end{equation}

\medskip
\noindent\textbf{Step 2: Estimates for \(\sigma_h\).}
We establish the following bounds:
\begin{align}
  \|\sigma_h\|_{L^p(\Omega;L^{p^*}(0,T;L^2(\mathcal{O})))} &\lesssim h^{2}
  && \text{for all } p\in [2,\infty), \label{eq:Sigmah-LpLp*L2} \\
  \|\sigma_h\|_{L^p([0,T] \times \Omega;H_{D}^{\alpha,q})} &\lesssim h^{1-\alpha}
  && \text{for all } p,q\in [2,\infty) \text{ and } \alpha \in [-1,0], \label{eq:Sigmah-LpLq} \\
  \|\sigma_h\|_{L^p(\Omega;C([0,T];H_{D}^{\alpha,q}))} &\lesssim h^{1-\alpha-\varepsilon}
  && \text{for all } p,q\in [2,\infty),\, \alpha \in [-1,0], \text{ and } \varepsilon \in (0,1), \label{eq:Sigmah-LpCHq} \\
  \|\sigma_h\|_{L^p(\Omega;C([0,T];L^\infty(\mathcal{O})))} &\lesssim h^{1-\varepsilon}
  && \text{for all } p\in [2,\infty) \text{ and } \varepsilon \in (0,1). \label{eq:Sigmah-Lp-C-Linfty}
\end{align}

\smallskip
\noindent\textit{Estimate for \(\chi_h^{(1)}\).}
Since \cref{eq:chih1-def} implies \(\chi_h^{(1)} = S_h \ast (A_hP_h\xi - P_hA\xi)\) $\mathbb{P}$-a.s., Lemma~\ref{lem:S0h} and the norm equivalence \(\|\cdot\|_{H_{D}^{\alpha,q}} \sim \|\cdot\|_{H_{D,h}^{\alpha,q}}\) on \(X_h\) (see \cref{lem:dotHh-equiv}) imply that, for any \(p,q \in [2,\infty)\) and \(\alpha \in [-1,0]\),
\begin{align*}
  \|\chi_h^{(1)}\|_{L^p([0,T] \times \Omega;H_{D}^{\alpha,q})} 
  & \sim \|\chi_h^{(1)}\|_{L^p([0,T] \times \Omega;H_{D,h}^{\alpha,q})} \\
  & \lesssim \|(A_hP_h - P_hA)\xi\|_{L^p([0,T] \times \Omega;H_{D,h}^{\alpha-2,q})} \\
  & = \|(P_h - A_h^{-1}P_hA)\xi\|_{L^p([0,T] \times \Omega;H_{D,h}^{\alpha,q})} \\
  & \lesssim \|(P_h - A_h^{-1}P_hA)\xi\|_{L^p([0,T] \times \Omega;H_{D}^{\alpha,q})}.
\end{align*}
Combining \eqref{eq:Ph-AhinvPhA} with \((\theta_1,\theta_2) = (1,\alpha)\) and the regularity of \(\xi\) from \cref{eq:xi-LpCH1q}, we obtain
\begin{equation}
  \label{eq:chi-h1-LpLq}
  \|\chi_h^{(1)}\|_{L^p([0,T] \times \Omega;H_{D}^{\alpha,q})} \lesssim h^{1-\alpha}
   \qquad \text{for all } p,q \in [2,\infty) \text{ and } \alpha \in [-1,0].
\end{equation}
Similarly, invoking the regularity result \cref{eq:xi-LpLp*H2} yields
\begin{equation}
  \label{eq:chi-h1-LpLp*L2}
  \|\chi_h^{(1)}\|_{L^p(\Omega;L^{p^*}(0,T;L^2(\mathcal{O})))} \lesssim h^{2}
   \qquad \text{for all } p \in [2,\infty).
\end{equation}

For any \(p,q \in [2,\infty)\), \(\alpha \in [-1,0]\), and \(\varepsilon \in (0,1)\),
the norm equivalences in \cref{lem:dotHh-equiv} and the convolution estimate in
\cref{lem:Sh-ast-bound}(i) (with $p=\infty$ and $(\theta_1,\theta_2)=(\alpha-2+\varepsilon, \alpha)$) yield
\begin{align*}
\|\chi_h^{(1)}\|_{L^{p}(\Omega;C([0,T];H_{D}^{\alpha,q}))}
& \sim \|\chi_h^{(1)}\|_{L^{p}(\Omega;C([0,T];H_{D,h}^{\alpha,q}))} \\
& \lesssim \|(A_hP_h - P_hA)\xi\|_{L^{p}(\Omega;L^\infty(0,T;H_{D,h}^{\alpha-2+\varepsilon,q}))} \\
& = \|(P_h - A_h^{-1}P_hA)\xi\|_{L^{p}(\Omega;L^\infty(0,T;H_{D,h}^{\alpha+\varepsilon,q}))} \\
& \lesssim \|(P_h - A_h^{-1}P_hA)\xi\|_{L^{p}(\Omega;L^\infty(0,T;H_{D}^{\alpha+\varepsilon,q}))}.
\end{align*}
Applying \eqref{eq:Ph-AhinvPhA} with \((\theta_1,\theta_2)=(1,\alpha+\varepsilon)\) in conjunction with the regularity from \cref{eq:xi-LpCH1q}, we infer that
\begin{equation}
  \label{eq:chi-h1-LpCHq}
\|\chi_h^{(1)}\|_{L^{p}(\Omega;C([0,T];H_{D}^{\alpha,q}))}
\lesssim h^{1-\alpha-\varepsilon} \quad \text{for all } p,q \in [2,\infty),
\, \alpha \in [-1,0], \text{ and } \varepsilon \in (0,1).
\end{equation}

For any \(p \in [2,\infty)\) and \(\varepsilon \in (0,1)\), by choosing \(\max\{1/\varepsilon,2\} < q < \infty\)
and applying \cref{lem:Sh-ast-bound}(ii) with $p=\infty$, we obtain
\begin{align*}
\|\chi_h^{(1)}\|_{L^{p}(\Omega;C([0,T];L^\infty(\mathcal{O})))}
& \lesssim \|(A_hP_h - P_hA)\xi\|_{L^{p}(\Omega;L^\infty(0,T;H_{D,h}^{-2+\varepsilon,q}))} \\
& = \|(P_h - A_h^{-1}P_hA)\xi\|_{L^{p}(\Omega;L^\infty(0,T;H_{D,h}^{\varepsilon,q}))} \\
& \lesssim \|(P_h - A_h^{-1}P_hA)\xi\|_{L^{p}(\Omega;L^\infty(0,T;H_{D}^{\varepsilon,q}))},
\end{align*}
where the last inequality uses the norm equivalence
$\|\cdot\|_{H_{D,h}^{\varepsilon,q}} \sim \|\cdot\|_{H_{D}^{\varepsilon,q}}$ on $X_h$ (\cref{lem:dotHh-equiv}).
An application of \eqref{eq:Ph-AhinvPhA} with \((\theta_1,\theta_2)=(1,\varepsilon)\), together with the regularity from \cref{eq:xi-LpCH1q}, yields
\begin{equation}
  \label{eq:chi-h1-LpLinf}
  \|\chi_h^{(1)}\|_{L^{p}(\Omega;C([0,T];L^\infty(\mathcal{O})))} \lesssim h^{1-\varepsilon}
   \qquad \text{for all } p \in [2,\infty) \text{ and } \varepsilon \in (0,1).
\end{equation}

\smallskip
\noindent\textit{Estimate for \(\chi_h^{(2)}\).}\quad
For any \(p,q \in [2,\infty)\), Lemma~\ref{lem:Ph-partialx-vw-H2} and H\"older's inequality imply,
$\mathbb{P}$-a.s.,
\begin{align*}
& \|P_h\partial_x[(u+u_h)(\xi-P_h\xi)]\|_{L^p([0,T] \times \Omega;H_{D,h}^{-2,q})} \\
\lesssim{} & 
\|u+u_h\|_{L^{2p}([0,T] \times \Omega;H_{D}^{1,q'})} \Bigl( \|\xi-P_h\xi\|_{L^{2p}([0,T] \times \Omega;H_{D}^{-1,q})}
      + h\|\xi-P_h\xi\|_{L^{2p}([0,T] \times \Omega;L^{q}(\mathcal{O}))} \Bigr) \\
\lesssim{} & 
\|u+u_h\|_{L^{2p}([0,T] \times \Omega;H_{D}^{1,2})} \Bigl( \|\xi-P_h\xi\|_{L^{2p}([0,T] \times \Omega;H_{D}^{-1,q})}
      + h\|\xi-P_h\xi\|_{L^{2p}([0,T] \times \Omega;L^{q}(\mathcal{O}))} \Bigr),
\end{align*}
where \(q'\) denotes the Hölder conjugate exponent of \(q\) and the last inequality exploits the continuous
embedding \(H_{D}^{1,2}\hookrightarrow H_{D}^{1,q'}\) (which holds because \(q'\leqslant 2\) as \(q\geqslant 2\)).  
By the triangle inequality, the regularity of $u$ established in \cref{eq:u-LpH1q}
and the $h$-uniform bound \cref{eq:uh-LpH1q} for $u_h$ (both applied with exponents $2p$ and $2$)
imply that \(\|u+u_h\|_{L^{2p}([0,T] \times \Omega;H_{D}^{1,2})}\) is bounded uniformly in \(h\).
Consequently, $\mathbb{P}$-a.s.,
\begin{align*}
  & \|P_h\partial_x[(u+u_h)(\xi-P_h\xi)]\|_{L^p([0,T] \times \Omega;H_{D,h}^{-2,q})} \\
  \lesssim{} &
   \|\xi-P_h\xi\|_{L^{2p}([0,T] \times \Omega;H_{D}^{-1,q})}
  + h\|\xi-P_h\xi\|_{L^{2p}([0,T] \times \Omega;L^{q}(\mathcal{O}))}.
\end{align*}
Applying the error estimates for \(\xi-P_h\xi\) from \eqref{eq:xi-Phxi-LpC}
(with \(p\) replaced by \(2p\) and \(\alpha\in\{-1,0\}\)) and employing the embeddings
\(L^{2p}(\Omega;C([0,T];H_{D}^{\alpha,q})) \hookrightarrow L^{2p}([0,T] \times \Omega;H_{D}^{\alpha,q})\)
for \(\alpha\in\{-1,0\}\), we deduce
\begin{equation}
\label{eq:100}
\|P_h\partial_x[(u+u_h)(\xi-P_h\xi)]\|_{L^p([0,T] \times \Omega;H_{D,h}^{-2,q})} \lesssim h^2
\qquad\text{for all } p,q\in[2,\infty).
\end{equation}
By \cref{eq:chih2-def} we have \(\chi_h^{(2)} = \tfrac{1}{2} S_h \ast \bigl( P_h\partial_x[(u+u_h)(\xi-P_h\xi)] \bigr)\)
\(\mathbb{P}\)-a.s.; therefore, Lemma~\ref{lem:S0h} and \eqref{eq:100} yield
\begin{equation}
\label{eq:chi-h2-LpLq}
\|\chi_h^{(2)}\|_{L^p([0,T] \times \Omega;L^q(\mathcal{O}))} \lesssim h^2
 \quad\text{for all } p,q \in [2,\infty).
\end{equation}
To establish the supremum norm bound, fix \(p \in [2,\infty)\) and \(\varepsilon \in (0,1)\),
and select \(p' \in [p, \infty)\) sufficiently large such that \(4/p' < \varepsilon\).
Using \cref{lem:Sh-ast-bound}(ii) with $(p,q)=(p',p')$,
the inverse estimate \cref{eq:inverse} with $(\theta_1,\theta_2,q) =(-2,-2+4/p',p')$,
and \eqref{eq:100} with \((p,q) = (p',p')\), we obtain
\begin{align*}
\|\chi_h^{(2)}\|_{L^{p'}(\Omega;C([0,T];L^\infty))}
&\lesssim \|P_h\partial_x[(u+u_h)(\xi-P_h\xi)]\|_{L^{p'}([0,T] \times \Omega;H_{D,h}^{-2+4/p',p'})} \\
&\lesssim h^{-4/p'} \|P_h\partial_x[(u+u_h)(\xi-P_h\xi)]\|_{L^{p'}([0,T] \times \Omega;H_{D,h}^{-2,p'})} \\
&\lesssim h^{2-4/p'} \lesssim h^{2-\varepsilon}.
\end{align*}
The embedding \(L^{p'}(\Omega) \hookrightarrow L^p(\Omega)\) then implies
\begin{equation}
\label{eq:chi-h2-LpLinf}
\|\chi_h^{(2)}\|_{L^p(\Omega;C([0,T];L^\infty(\mathcal{O})))} \lesssim h^{2-\varepsilon}
 \quad \text{for all } p \in[2,\infty) \text{ and } \varepsilon \in (0,1).
\end{equation}

\smallskip
\noindent\textit{Estimate for \(\chi_h^{(3)}\).}\quad
By \cref{eq:chih3-def}, \(\chi_h^{(3)} = \tfrac{1}{2} S_h \ast \bigl( P_h\partial_x[(u+u_h)(G-G_h)] \bigr)\)
holds \(\mathbb{P}\)-a.s. The same argument as for \(\chi_h^{(2)}\), with the error bounds for \(\xi-P_h\xi\)
replaced by the corresponding estimates for \(G-G_h\) from Lemma~\ref{lem:G-Gh}
(specifically, \cref{eq:G-Gh-LpLq} with \(\alpha \in \{-1, 0\}\)), yields
\begin{align}
  \lVert \chi_h^{(3)} \rVert_{L^p([0,T] \times \Omega; L^q(\mathcal{O}))}
    &\lesssim h^{2}
    && \text{for all } p,q \in [2,\infty),
    \label{eq:chi-h3-LpLq} \\
  \lVert \chi_h^{(3)} \rVert_{L^{p}(\Omega; C([0,T]; L^\infty(\mathcal{O})))}
    &\lesssim h^{2-\varepsilon}
    && \text{for all } p \in [2,\infty) \text{ and } \varepsilon \in (0,1).
    \label{eq:chi-h3-LpLinf}
\end{align}

\smallskip
\noindent\textit{Conclusion.}
Recall the decomposition $\sigma_h = \chi_h^{(1)} + \chi_h^{(2)} + \chi_h^{(3)}$.
To prove \cref{eq:Sigmah-LpLp*L2}, we first note that for $p \in [p^*,\infty)$ and $q \in [2,\infty)$, the embedding 
$L^p([0,T] \times \Omega;L^q(\mathcal{O})) \hookrightarrow L^p(\Omega;L^{p^*}(0,T;L^2(\mathcal{O})))$ holds. 
Combining this with \cref{eq:chi-h1-LpLp*L2,eq:chi-h2-LpLq,eq:chi-h3-LpLq} yields 
$\|\sigma_h\|_{L^p(\Omega;L^{p^*}(0,T;L^2(\mathcal{O})))} \lesssim h^2$ for $p \geqslant p^*$; the monotonicity of the $L^p(\Omega)$-norms then extends this estimate to all $p \in [2,\infty)$.
Next, \cref{eq:Sigmah-LpLq} follows from the triangle inequality applied to 
\cref{eq:chi-h1-LpLq,eq:chi-h2-LpLq,eq:chi-h3-LpLq} together with the embedding 
$L^q(\mathcal{O}) \hookrightarrow H_{D}^{\alpha,q}$ for $\alpha \in [-1,0]$.
Similarly, applying the triangle inequality to 
\cref{eq:chi-h1-LpCHq,eq:chi-h2-LpLinf,eq:chi-h3-LpLinf} and utilizing 
$L^\infty(\mathcal{O}) \hookrightarrow H_{D}^{\alpha,q}$ for $\alpha \in [-1,0]$ and $q \in [2,\infty)$ yields \cref{eq:Sigmah-LpCHq}.
Finally, \cref{eq:Sigmah-Lp-C-Linfty} is a direct consequence of 
\cref{eq:chi-h1-LpLinf,eq:chi-h2-LpLinf,eq:chi-h3-LpLinf}.

\medskip 
\textbf{Step 3: Estimate for $z_h$ in $L^p(\Omega;C([0,T];L^\infty(\mathcal{O})))$.}
We apply \cref{lem:core-stability} to obtain the desired estimate for $z_h$.
Combining the exponential moment estimates from Propositions~\ref{prop:u-regu} and~\ref{prop:uh-regu} with the triangle inequality and the elementary bound $(a+b)^2 \leqslant 2(a^2+b^2)$, we infer the existence of a constant $\kappa_0 > 0$ such that
\[
  \sup_{0<h\leqslant 1} \mathbb{E}\Bigl[ \exp\Bigl( \kappa_0 
  \|u+u_h\|_{L^2(0,T;H_{D}^{1,2})}^2 \Bigr) \Bigr] < \infty .
\]
Furthermore, invoking the Sobolev embedding $H_{D}^{\theta,q} \hookrightarrow L^\infty(\mathcal{O})$ (valid for $\theta q > 1$) together with the regularity estimates \cref{eq:u-LpCHq} and \cref{eq:uh-LpCHq}, we obtain the uniform bound
\[
  \sup_{0<h\leqslant 1}\|u+u_h\|_{L^{p}(\Omega;C([0,T];L^\infty(\mathcal{O})))} < \infty
   \quad \text{for all } p \in [2,\infty).
\]
We now apply \cref{lem:core-stability} to equation \eqref{eq:Xih} with $p_1 = p^*$ and $g = u+u_h$. For any $p \in [2,\infty)$, this yields
\[
  \|z_h\|_{L^p(\Omega;C([0,T];L^\infty(\mathcal{O})))} \lesssim
  \|\sigma_h\|_{L^{2p}(\Omega; L^{p^*}(0,T;L^2(\mathcal{O})))}
  + \|\sigma_h\|_{L^{8p}(\Omega; L^2(0,T;L^2(\mathcal{O})))}.
\]
Noting that the space $L^{8p}(\Omega;L^{p^*}(0,T;L^2(\mathcal{O})))$ continuously embeds into the intersection
$L^{2p}(\Omega;L^{p^*}(0,T;L^2(\mathcal{O}))) \cap L^{8p}(\Omega;L^2(0,T;L^2(\mathcal{O})))$, the above inequality simplifies to
\[
  \|z_h\|_{L^p(\Omega;C([0,T];L^\infty(\mathcal{O})))} \lesssim \|\sigma_h\|_{L^{8p}(\Omega;L^{p^*}(0,T;L^2(\mathcal{O})))}
   \quad \text{for all } p \in [2,\infty).
 \]
Finally, substituting the estimate for $\sigma_h$ from \cref{eq:Sigmah-LpLp*L2} (with $p$ replaced by $8p$), we conclude that
\begin{equation}
  \label{eq:zh-LpCLinf}
  \|z_h\|_{L^p(\Omega;C([0,T];L^\infty(\mathcal{O})))} \lesssim h^2 \quad \text{for all } p \in [2,\infty).
\end{equation}

\medskip
\noindent\textbf{Step 4: Estimate of $\xi_h - P_h\xi$.}
We employ the decomposition $\xi_h - P_h\xi = \sigma_h + z_h$. 
For $p,q \in [2,\infty)$ and $\alpha \in[-1,0]$, the triangle inequality and the embedding
$L^p(\Omega;C([0,T];L^\infty)) \hookrightarrow L^p([0,T] \times \Omega;H_{D}^{\alpha,q})$ yield
\begin{align*}
  \|\xi_h - P_h\xi\|_{L^p([0,T] \times \Omega;H_{D}^{\alpha,q})}
  & \lesssim \|\sigma_h\|_{L^p([0,T] \times \Omega;H_{D}^{\alpha,q})}
  + \|z_h\|_{L^p(\Omega;C([0,T];L^\infty(\mathcal{O})))} \\
  & \lesssim h^{1-\alpha} + h^2 \lesssim h^{1-\alpha},
\end{align*}
where we used \cref{eq:zh-LpCLinf,eq:Sigmah-LpLq}. This establishes \eqref{eq:xih-Phxi-bound1}.
Similarly, for $p,q \in [2,\infty)$, $\alpha \in[-1,0]$, and $\varepsilon \in (0,1)$,
\begin{align*}
  \|\xi_h - P_h\xi\|_{L^p(\Omega;C([0,T];H_{D}^{\alpha,q}))}
  & \lesssim \|\sigma_h\|_{L^p(\Omega;C([0,T];H_{D}^{\alpha,q}))}
  + \|z_h\|_{L^p(\Omega;C([0,T];L^\infty(\mathcal{O})))} \\
  & \lesssim  h^{1-\alpha-\varepsilon} + h^2 \lesssim h^{1-\alpha-\varepsilon},
\end{align*}
where we invoked \cref{eq:zh-LpCLinf,eq:Sigmah-LpCHq}, yielding \eqref{eq:xih-Phxi-LpCHq}.
Finally, for any $p \in [2,\infty)$ and $\varepsilon \in (0,1)$, \cref{eq:zh-LpCLinf,eq:Sigmah-Lp-C-Linfty} give
\begin{align*}
  \|\xi_h - P_h\xi\|_{L^p(\Omega;C([0,T];L^\infty))}
  & \lesssim \|\sigma_h\|_{L^p(\Omega;C([0,T];L^\infty))}
  + \|z_h\|_{L^p(\Omega;C([0,T];L^\infty))} \\
  & \lesssim h^{1-\varepsilon} + h^2 \lesssim h^{1-\varepsilon}.
\end{align*}
This verifies \eqref{eq:xih-Phxi-bound2} and concludes the proof of Theorem~\ref{thm:uh-strong}.

\subsection{Proof of \texorpdfstring{\cref{thm:uh-weak}}{}}
\label{ssec:proof-uh-weak}
We begin by establishing several boundedness results for the Nemytskii operator $v \mapsto |v|^{q-2}v$ for $q \in [2,\infty)$.
\begin{lemma}
  \label{lem:composition-embeddings}
  Let $q \in [2, \infty)$ and let $q' := q/(q-1)$ be its H\"older conjugate exponent. The following estimates hold:
  \begin{enumerate}
    \item[\textup{(i)}] Let $1/2 < \theta_1 < \theta_2 < 1$. For any $v \in H_{D}^{\theta_2,q}$, we have
    $\bigl\||v|^{q-2}v\bigr\|_{H_{D}^{\theta_1,q'}} \lesssim \|v\|_{H_{D}^{\theta_2,q}}^{q-1}$.
    \item[\textup{(ii)}] For any $v \in H_{D}^{1,q}$,
    $\bigl\||v|^{q-2}v\bigr\|_{H_{D}^{1,q'}} \lesssim \|v\|_{H_{D}^{1,q}}^{q-1} $.
  \end{enumerate}
\end{lemma}
\begin{proof}
These estimates are standard (see, e.g., \cite[Chapter~5]{Runst1996sobolev}); we include a brief proof for completeness.

\textit{Proof of (i).}
Recall that $H_{D}^{\theta_j,q} = \big[L^q(\mathcal{O}), H_{D}^{1,q}\big]_{\theta_j}$ for $j \in \{1,2\}$,
with equivalent norms (cf.~\cite{Seeley1971norms} and \cite[Theorem~4.17]{Lunardi2018}).
By \cite[Theorem~C.4.1]{HytonenWeis2016} and \cite[Proposition~1.4]{Lunardi2018}, we have the embeddings
\[
H_{D}^{\theta_2,q} \hookrightarrow \big(L^q(\mathcal{O}), H_{D}^{1,q}\big)_{\theta_2,q}
\hookrightarrow \big(L^q(\mathcal{O}), H_{D}^{1,q}\big)_{\theta_1,q'} \hookrightarrow H_{D}^{\theta_1,q},
\]
where we used the fact that both $L^q(\mathcal{O})$ and $H_{D}^{1,q}$
have Fourier type $q'$ (see \cite[Example~2.4.14]{HytonenWeis2016}).
Since $q \geqslant 2$ implies $q \geqslant q'$,
it follows that $H_{D}^{\theta_1,q} \hookrightarrow H_{D}^{\theta_1,q'}$.
Consequently, we obtain the chain
\begin{equation}
  \label{eq:continuous-embeddings}
  H_{D}^{\theta_2,q}
  \hookrightarrow \big(L^q(\mathcal{O}), H_{D}^{1,q}\big)_{\theta_1,q'}
  \hookrightarrow H_{D}^{\theta_1,q'}.
\end{equation}
The real interpolation space $\big(L^q(\mathcal{O}), H_{D}^{1,q}\big)_{\theta_1,q'}$ admits the equivalent norm
(see \cite[Definitions~1.2 and 1.12 and Proposition~1.29]{Guidetti1991interpolation})
\[
\|v\|_{L^q(\mathcal{O})} + \left( \int_0^1 r^{-\theta_1q'} \left( \int_{\mathbb{R}} |v(r+s) - v(s)|^q \, \mathrm{d}s \right)^{q'/q}
\, \frac{\mathrm{d}r}{r} \right)^{1/q'}, \quad v \in \big(L^q(\mathcal{O}), H_{D}^{1,q}\big)_{\theta_1,q'},
\]
where $v$ is extended by zero outside $\mathcal{O}$.
A direct calculation, using the pointwise estimate
$ \bigl| |a|^{q-2}a - |b|^{q-2}b \bigr| \lesssim \bigl(|a|^{q-2} + |b|^{q-2}\bigr)|a-b| $, yields
\[
\bigl\||v|^{q-2}v\bigr\|_{(L^q,H_{D}^{1,q})_{\theta_1,q'}} \lesssim \|v\|_{L^\infty(\mathcal{O})}^{q-2}
\|v\|_{\big(L^q(\mathcal{O}), H_{D}^{1,q}\big)_{\theta_1,q'}}
\]
for all $v \in L^\infty(\mathcal{O}) \cap \big(L^q(\mathcal{O}), H_{D}^{1,q}\big)_{\theta_1,q'}$.
Finally, combining the embeddings in \eqref{eq:continuous-embeddings} with the Sobolev embedding $H_{D}^{\theta_2,q} \hookrightarrow L^\infty(\mathcal{O})$ yields the desired estimate. 

\textit{Proof of (ii).} The case $q=2$ is immediate. For $q>2$, using the embedding $H_{D}^{1,q} \hookrightarrow H_{D}^{1,q'}$,
the chain rule $\partial_x(|v|^{q-2}v)=(q-1)|v|^{q-2}\partial_x v$,
H\"older's inequality, and the embedding $H_{D}^{1,q} \hookrightarrow L^\infty(\mathcal{O})$, we obtain
\[
\||v|^{q-2}v\|_{H_{D}^{1,q'}} \lesssim
\||v|^{q-2}v\|_{H_{D}^{1,q}} \lesssim \||v|^{q-2}\partial_x v\|_{L^{q}(\mathcal{O})}
 \leqslant \|v\|_{L^\infty(\mathcal{O})}^{q-2} \|\partial_x v\|_{L^{q}(\mathcal{O})} \lesssim \|v\|_{H_{D}^{1,q}}^{q-1}.
\]
\end{proof}

Now, let us prove \cref{thm:uh-weak}.

\textbf{Proof of \cref{eq:weak-error-1}.}
Fix arbitrary exponents \(2 \leqslant q \leqslant p < \infty\) and set \(p' := p/(p-1)\), \(q':= q/(q-1)\).
Applying \cref{eq:u-LpCHq,eq:uh-LpCHq} with \(\theta\) replaced by \(\vartheta\),
we infer that \(u(T)\) and \(u_h(T)\) belong to \(L^{p}(\Omega;H_{D}^{\vartheta,q})\) with norms
bounded uniformly in \(h\), for each $\vartheta \in (1/2,1)$.
Define \(\Phi(s) := u_h(T)+s(u(T)-u_h(T))\) for \(s\in[0,1]\). By the triangle inequality,
\[
\sup_{0<h\leqslant 1}\sup_{s\in[0,1]} \|\Phi(s)\|_{L^{p}(\Omega;H_{D}^{\vartheta,q})} < \infty
\quad\text{for all } \vartheta \in (1/2,1).
\]
Applying \cref{lem:composition-embeddings}(i) together with the embedding
\(H_{D}^{\vartheta,q}\hookrightarrow L^q(\mathcal{O})\), we obtain the uniform estimate,
for any $1/2 < \theta < \vartheta < 1$,
\[
\begin{aligned}
&\sup_{0<h\leqslant 1}\sup_{s\in[0,1]} \,
\Bigl\| \|\Phi(s)\|_{L^q(\mathcal{O})}^{p-q} \, \bigl\| |\Phi(s)|^{q-2} \Phi(s) \bigr\|_{H_{D}^{\theta,q'}} \Bigr\|_{L^{p'}(\Omega)} \\
\lesssim{} & \sup_{0<h\leqslant 1}\sup_{s\in[0,1]} \,
\Bigl\|  \Phi(s) \Bigr\|_{L^{p}(\Omega; H_{D}^{\vartheta,q})}^{p-1} < \infty.
\end{aligned}
\] 
Using the mean value theorem and the duality pairing between \(H_{D}^{\theta,q'}\) and \(H_{D}^{-\theta,q}\),
we derive for any $\theta \in (1/2,1)$,
\[
\begin{aligned}
& \Bigl|\mathbb{E}\|u(T)\|_{L^q(\mathcal{O})}^{p} - \mathbb{E}\|u_h(T)\|_{L^q(\mathcal{O})}^{p}\Bigr| \\
={} & \biggl|\mathbb{E} \Bigl[ \int_0^1  p \|\Phi(s)\|_{L^q(\mathcal{O})}^{p-q} \int_{\mathcal{O}} |\Phi(s)|^{q-2}\Phi(s) (u(T) - u_h(T))\, \mathrm{d}x \, \mathrm{d}s \Bigr]  \biggr| \\
\leqslant{} & p \int_{0}^{1} \mathbb{E}\Bigl[ \|\Phi(s)\|_{L^q(\mathcal{O})}^{p-q} \bigl\| |\Phi(s)|^{q-2}\Phi(s) \bigr\|_{H_{D}^{\theta,q'}} \bigl\| u(T)-u_h(T) \bigr\|_{H_{D}^{-\theta,q}} \Bigr] \,\mathrm{d}s \\
\lesssim{} & \|u(T)-u_h(T)\|_{L^{p}(\Omega;H_{D}^{-\theta,q})},
\end{aligned}
\]
where the final inequality follows from H\"older's inequality in \(\Omega\) and the preceding uniform estimate.
Finally, let \(\varepsilon\in(0,1)\) be given and choose \(\theta = 1-\varepsilon/2\), so that \(\theta\in(1/2,1)\). Applying \cref{eq:u-uh-LpCLq} with \(\alpha = -\theta\) and small parameter \(\varepsilon/2\), we obtain
\[
\|u(T)-u_h(T)\|_{L^{p}(\Omega;H_{D}^{-\theta,q})}
\lesssim h^{1+\theta-\varepsilon/2} = h^{2-\varepsilon}.
\]
Inserting this bound into the previous inequality completes the proof of \cref{eq:weak-error-1}.

\textbf{Proof of \cref{eq:weak-error-2}.}
We apply an analogous duality argument over $[0,T] \times \Omega$.
Fix arbitrary exponents \(2 \leqslant q \leqslant p < \infty\) and set \(p' := p/(p-1)\), \(q':= q/(q-1)\).
With $\varTheta(s) := u_h + s(u-u_h)$, \cref{eq:u-LpH1q,eq:uh-LpH1q} imply
\[
\sup_{0<h\leqslant 1} \sup_{s \in [0,1]}
\, \bigl\| \varTheta(s) \bigr\|_{L^{p}([0,T] \times \Omega; H_{D}^{1,q})} < \infty.
\]
By \cref{lem:composition-embeddings}(ii) and
the embedding $H_{D}^{1,q} \hookrightarrow L^q(\mathcal{O})$,
we obtain the uniform bound
\begin{align*}
& \sup_{0<h\leqslant 1} \sup_{s \in [0,1]} \, \Bigl\| \|\varTheta(s)\|_{L^q(\mathcal{O})}^{p-q} \bigl\| |\varTheta(s)|^{q-2}\varTheta(s) \bigr\|_{H_{D}^{1,q'}} \Bigr\|_{L^{p'}([0,T] \times \Omega)} \\
\lesssim{} &
\sup_{0<h\leqslant 1} \sup_{s \in [0,1]}
\, \Bigl\| \varTheta(s) \Bigr\|_{L^{p}([0,T] \times \Omega;H_{D}^{1,q})}^{p-1} < \infty.
\end{align*}
The mean value theorem and the $H_{D}^{1,q'}$--$H_{D}^{-1,q}$ duality pairing yield
\[
\begin{aligned}
  & \Bigl|\mathbb{E}\|u\|_{L^p(0,T;L^{q}(\mathcal{O}))}^{p} - \mathbb{E}\|u_h\|_{L^p(0,T;L^{q}(\mathcal{O}))}^{p}\Bigr| \\
  \leqslant{} & p \int_0^1 \mathbb{E} \Bigl[ \int_0^T \|\varTheta(s)\|_{L^q(\mathcal{O})}^{p-q} \bigl\| |\varTheta(s)|^{q-2}\varTheta(s) \bigr\|_{H_{D}^{1,q'}} \bigl\| u-u_h \bigr\|_{H_{D}^{-1,q}} \, \mathrm{d}t \Bigr] \, \mathrm{d}s \\
  \lesssim{} & \|u-u_h\|_{L^{p}([0,T] \times \Omega;H_{D}^{-1,q})},
\end{aligned}
\]
where the last step follows from H\"older's inequality on $[0,T] \times \Omega$ and the uniform bound above.
Finally, applying \cref{eq:u-uh-LpLq} with $\alpha = -1$ concludes the proof.

\qed

\section{Fully discrete scheme}
\label{sec:full-discr}
In this section, we analyze a fully discrete approximation of \cref{eq:burgers}. 
Let $J \in \mathbb{N}$ denote the total number of time steps, and set $\tau := T/J$ as the uniform time step size.
For $j = 0, 1, \dots, J$, define $t_j := j\tau$. This section considers the following fully discrete scheme:
\begin{equation}
  \label{eq:U}
  \begin{cases}
    U_{j+1} - U_j = \tau A_h U_{j+1} - \frac{\tau}{2} P_h\partial_x(U_{j+1}^2)
    + \int_{t_j}^{t_{j+1}} P_hQ \, \mathrm{d}W(t), \quad j = 0, 1, \dots, J-1, \\[4pt]
    U_0 = P_h u_0 .
  \end{cases}
\end{equation}

\begin{definition}[Fully discrete solution]
  \label{def:U}
  Let $u_0 \colon \Omega \to L^2(\mathcal{O})$ be strongly $\mathcal{F}_0$-measurable.
  An $(\mathcal{F}_{t_j})_{j=0}^J$-adapted sequence of $X_h$-valued random variables $(U_j)_{j=0}^J$ is
  called a \emph{solution} to \cref{eq:U} if \cref{eq:U} holds in $X_h$ $\mathbb{P}$-a.s.~for all $j = 0, 1, \dots, J-1$.
  Two solutions $(U_j)_{j=0}^J$ and $(V_j)_{j=0}^J$ are said to be identical if
  $U_j = V_j$ $\mathbb{P}$-a.s.~for all $j = 0, 1, \dots, J$.
\end{definition}

\begin{proposition}
  \label{prop:U}
  Assume that the initial datum $u_0$ is deterministic and belongs to $W_0^{1,\infty}(\mathcal{O})$.
  Then, there exists at least one solution to the fully discrete scheme \cref{eq:U}.
  Furthermore, every solution of \cref{eq:U} satisfies the following properties:
  \begin{enumerate}
    \item[\textup{(i)}]
      For any $p,q \in [2,\infty)$ and $ \theta \in (0,1) $, the solution $(U_j)_{j=0}^J$ satisfies the following uniform bounds:
      \begin{align}
        & \sup_{0<h, \tau\leqslant 1} \mathbb{E} \Bigl[ \sum_{j=0}^{J} \tau \|U_j\|_{H_{D}^{1,q}}^p \Bigr] < \infty, \label{eq:U-LpH1q} \\
        & \sup_{0<h,\tau \leqslant 1} \mathbb{E} \Bigl[ \max_{0 \leqslant j \leqslant J}
        \|U_j\|_{H_{D}^{\theta,q}}^p \Bigr] < \infty. \label{eq:U-LpCHq}
      \end{align}
    \item[\textup{(ii)}]
      There exist a constant $\kappa > 0 $ and $ \tau_0 \in (0,1) $, independent of $ h$ and $ \tau $, such that
      \begin{equation}
        \label{eq:U-exp-moment}
        \sup_{0<h\leqslant 1,\, 0 <\tau \leqslant \tau_0}\mathbb{E} \Bigl[
          \exp\Bigl( \kappa\tau \sum_{j=1}^J \|U_j\|_{H_{D}^{1,2}}^2\Bigr) \Bigr] < \infty.
      \end{equation}
  \end{enumerate}
\end{proposition}

We now state the main result of this section.

\begin{theorem}
  \label{thm:uh-U}
  Let $u_0 \in W_0^{1,\infty}(\mathcal{O})$ be a deterministic initial datum.
  Let $u_h$ denote the solution to the spatial semi-discretization \cref{eq:uh},
  and let $(U_j)_{j=0}^J$ be a solution to the fully discrete scheme \cref{eq:U}.
  Assume that $\tau \leqslant h^2 < 1$.
  Then, for all $p \in [2, \infty)$ and $\varepsilon \in (0, 1/2)$,
  \begin{equation} \label{eq:uh-U-Linfty}
    \left( \mathbb{E} \left[ \max_{1 \leqslant j \leqslant J} \| u_h(t_j) - U_j \|_{L^\infty(\mathcal{O})}^p \right] \right)^{1/p}
    \leqslant C \tau^{1/2-\varepsilon},
  \end{equation}
  where the constant $C>0$ is independent of the discretization parameters $h$ and $\tau$.
\end{theorem}

\begin{remark}
The constraint $\tau \leqslant h^2$ is primarily a technical artifact of the current proof
rather than an essential limitation of the method. 
It is imposed to balance the temporal discretization error with the spatial one:
under this scaling, $\tau^{1/2-\varepsilon} \leqslant h^{1-2\varepsilon}$,
so the temporal rate in \eqref{eq:uh-U-Linfty} is compatible with
the spatial rate $O(h^{1-\varepsilon})$ from \eqref{eq:u-uh-LpCLinf}.
The proof relies on this condition solely for certain inverse estimates in
Subsection~\cref{ssec:proof-thm-uh-U}.
We expect that this nonessential restriction can be relaxed in future investigations.
\end{remark}

In the remainder of this section, we prove \cref{prop:U} and \cref{thm:uh-U}.
Following the notation from \cref{sec:spatial-semi-discretization}, we write \(a \lesssim b\) to denote \(a \leqslant C b\) for a generic constant \(C > 0\) independent of the spatial mesh size \(h\), the time step \(\tau\), and the time index \(j\).
The constant \(C\) may depend on the domain \(\mathcal{O}\), the terminal time \(T\), 
the initial datum \(u_{0}\), the parameters \(p, q, \varepsilon\), and the orders \(\theta\) defining the spaces \(H_{D}^{\theta,q}\) and \(H_{D,h}^{\theta,q}\).
The rest of this section is organized as follows: \cref{ssec:prelim2} provides the necessary preliminaries, \cref{ssec:proof-U} contains the proof of \cref{prop:U}, and \cref{ssec:proof-thm-uh-U} presents the proof of \cref{thm:uh-U}.

\subsection{Preliminary results}
\label{ssec:prelim2}

We begin by recalling standard estimates for the implicit Euler approximation operator $(I - \tau A_h)^{-m}$,
$m \geqslant 1$, and the semigroup $S_h(t)$ generated by $A_h$, as introduced in \cref{subsubsec:Sh}.

\begin{lemma}
  \label{lem:resolvent-bound}
  Let $ q \in (1,\infty) $ and $ -\infty < \theta_1 \leqslant \theta_2 \leqslant \theta_1 + 2 < \infty$. Then, the following estimates hold:
  \begin{enumerate}
    \item[\textup{(i)}] $ \|(I - \tau A_h)^{-m}\|_{\mathcal{L}(H_{D,h}^{\theta_1,q},H_{D,h}^{\theta_2,q})}
    \lesssim (m\tau)^{(\theta_1-\theta_2)/2} $ for all $ m \in \mathbb{N}_{>0} $.
    \item[\textup{(ii)}] $ \|I - S_h(t)\|_{\mathcal{L}(H_{D,h}^{\theta_2,q},H_{D,h}^{\theta_1,q})}
    \lesssim t^{(\theta_2-\theta_1)/2} $ for all $ 0 < t \leqslant T $.
  \end{enumerate}
\end{lemma}

\begin{proof}
  Since these results are classical, we provide only a brief outline of the arguments.
  The property \cref{eq:resolvent-bound} allows for the application of standard analytic semigroup theory.
  In particular, estimate \textup{(i)} for the endpoint cases $ \theta_2 = \theta_1 $ and $ \theta_2 = \theta_1 + 2 $ follows
  directly from \cite[Theorem~6.3, Chapter~1]{Pazy1983} and \cite[Theorem~5.5, Chapter~2]{Pazy1983}, respectively.
  The general case is then obtained by complex interpolation.
  Assertion \textup{(ii)} is a direct consequence of \cite[Theorem~6.13, Chapter~2]{Pazy1983}.
\end{proof}

We now establish a discrete maximal regularity estimate for the implicit Euler scheme,
which follows from a direct computation based on \cref{lem:resolvent-bound}(i) and H\"older's inequality.

\begin{lemma}
\label{lem:discrete-maximal}
Let $q \in (1,\infty)$ and $-\infty < \theta_1 < \theta_2 < \theta_1 + 2 < \infty$.
For any $g_h \in L^\infty(0,T;H_{D,h}^{\theta_1,q})$, the sequence $(Z_j)_{j=1}^J$, defined by
\[
Z_j := \sum_{k=0}^{j-1} \int_{t_k}^{t_{k+1}} (I-\tau A_h)^{-(j-k)} g_h(t)\,\mathrm{d}t,
\qquad 1 \leqslant j \leqslant J,
\]
satisfies
\[
\max_{1 \leqslant j \leqslant J} \|Z_j\|_{H_{D,h}^{\theta_2,q}} \lesssim \|g_h\|_{L^\infty(0,T;H_{D,h}^{\theta_1,q})}.
\]
\end{lemma}

Finally, we present the discrete stochastic maximal $L^p$-regularity estimate and a discrete
maximal inequality for the Euler--Maruyama scheme.
\begin{lemma}
  \label{lem:discrete_stoch_maximal}
  Let \( p \in (2,\infty) \), \( q \in [2,\infty) \), and 
  \( g_h \in L_{\mathbb{F}}^p([0,T] \times \Omega; \gamma(\ell^2,H_{D,h}^{0,q})) \).
  Suppose that the sequence \( (Z_j)_{j=0}^J \) satisfies \( Z_0 = 0 \) and, \( \mathbb{P} \)-a.s.,
  \[
    Z_{j+1} - Z_j - \tau A_h Z_{j+1} = \int_{t_j}^{t_{j+1}} g_h(t) \, \mathrm{d}W(t), \quad 0 \leqslant j < J.
  \]
  Then,
  \begin{align}
    \bigg(\mathbb{E} \bigg[ \sum_{j=1}^J \tau \|Z_j\|_{H_{D,h}^{1,q}}^p \bigg] \bigg)^{1/p}
    &\lesssim \|g_h\|_{L^p([0,T] \times \Omega;\gamma(\ell^2,H_{D,h}^{0,q}))}, 
    \label{eq:discrete_stoch_maximal} \\
    \bigg(\mathbb{E} \bigg[ \max_{1 \leqslant j \leqslant J} \|Z_j\|_{(H_{D,h}^{0,q},H_{D,h}^{2,q})_{1/2-1/p,p}}^p \bigg] \bigg)^{1/p}
    &\lesssim \|g_h\|_{L^p([0,T] \times \Omega;\gamma(\ell^2,H_{D,h}^{0,q}))}.
    \label{eq:discrete_stoch_maximal2}
  \end{align}
\end{lemma}

\begin{proof}
  First, observe that $A_h$ satisfies the resolvent bound \cref{eq:resolvent-bound}.
  Furthermore, by following the argument in \cite[Theorem~3.1]{LiZhouLp2026},
  one can show that the boundedness constant of the $ H^\infty $-calculus for $ -A_h $ is independent of $ h $.
  Consequently, by adapting the proof of \cite[Theorem~3.2]{li2025stability}—specifically, by replacing $ L^q(\mathcal{O}) $ with $ H_{D,h}^{0,q}$—we establish \eqref{eq:discrete_stoch_maximal}.
  It remains to prove \eqref{eq:discrete_stoch_maximal2}.
  While this inequality can be inferred by following the argument of
  \cite[Remark~4.2]{li2025stability} (or by arguments similar to those in
  \cite[Theorem~6.1]{Evangelopoulos-Ntemiris2026discrete}),
  we provide a complete proof below.

  We divide the proof into two steps, where we set
  \[
    \mathscr{G} := \|g_h\|_{L^p([0,T] \times \Omega;\gamma(\ell^2,H_{D,h}^{0,q}))}
    \text{ and } B_{p,q} := (H_{D,h}^{0,q},H_{D,h}^{2,q})_{1/2-1/p,p}.
  \]

  \textbf{Step 1.}
  Let $z_h := S_h \diamond g_h$ denote the spatially semi-discrete stochastic convolution defined in
  \cref{eq:S1h-def}. From \cref{lem:S1h}, we immediately obtain the estimates
  \begin{align}
    \|z_h\|_{L^p([0,T] \times \Omega;H_{D,h}^{1,q})}
    &\lesssim \mathscr{G}, \label{eq:zh-LpH1q} \\
    \|z_h\|_{L^p(\Omega;C([0,T];B_{p,q}))}
    &\lesssim \mathscr{G}. \label{eq:zh-LpCHq}
  \end{align}
  Combining \cref{eq:discrete_stoch_maximal} with estimate \eqref{eq:zh-LpH1q} via the triangle inequality yields
  \[
    \biggl( \mathbb{E} \bigg[\sum_{j=0}^{J-1} \int_{t_j}^{t_{j+1}} \|z_h(t) - Z_j\|_{H_{D,h}^{1,q}}^p \, \mathrm{d}t\bigg]\biggr)^{1/p}
    \lesssim \mathscr{G}.
  \]
  Furthermore, adapting the proof of \cite[Theorem~4.1]{li2025stability} to the
  spatially discrete setting, we find that
  \[
    \biggl( \mathbb{E} \bigg[\sum_{j=0}^{J-1} \int_{t_j}^{t_{j+1}} \|z_h(t) - Z_j\|_{H_{D,h}^{0,q}}^p \, \mathrm{d}t\bigg]\biggr)^{1/p}
    \lesssim \tau^{1/2} \mathscr{G}.
  \]
  Using these estimates and the real interpolation inequality (\cite[Corollary~1.7]{Lunardi2018})
  \[
    \|v_h\|_{(H_{D,h}^{0,q},H_{D,h}^{1,q})_{1-2/p,p}}
    \lesssim \|v_h\|_{H_{D,h}^{0,q}}^{2/p} \|v_h\|_{H_{D,h}^{1,q}}^{1-2/p},
    \quad v_h \in X_h,
  \]
  a simple calculation using H\"older's inequality yields
  \begin{equation} \label{eq:interpolation-error}
    \biggl(
      \mathbb{E} \bigg[
        \sum_{j=0}^{J-1} \int_{t_j}^{t_{j+1}} \|z_h(t) - Z_j\|_{(H_{D,h}^{0,q},H_{D,h}^{1,q})_{1-2/p,p}}^p \, \mathrm{d}t
      \bigg]
    \biggr)^{1/p} \lesssim \tau^{1/p} \mathscr{G}.
  \end{equation}
  Next, observe that the inequality
  \begin{align*}
    \max_{1\leqslant j < J} \|Z_j\|_{(H_{D,h}^{0,q},H_{D,h}^{1,q})_{1-2/p,p}}
    &\leqslant \|z_h\|_{C([0,T];(H_{D,h}^{0,q},H_{D,h}^{1,q})_{1-2/p,p})} \\
    & \quad {} + \max_{1 \leqslant j < J} \frac{1}{\tau} \int_{t_j}^{t_{j+1}}
    \|z_h(t) - Z_j\|_{(H_{D,h}^{0,q},H_{D,h}^{1,q})_{1-2/p,p}} \, \mathrm{d}t
  \end{align*}
  holds $\mathbb{P}$-a.s. Taking the $L^p(\Omega)$-norm on both sides and applying the triangle inequality, we obtain
  \begin{align*}
    &\bigg(\mathbb{E} \bigg[ \max_{1\leqslant j < J} \|Z_j\|_{(H_{D,h}^{0,q},H_{D,h}^{1,q})_{1-2/p,p}}^p \bigg]\bigg)^{1/p}\\
    \leqslant{}&
    \|z_h\|_{L^p(\Omega;C([0,T];(H_{D,h}^{0,q},H_{D,h}^{1,q})_{1-2/p,p}))} \\
    &\quad +
    \biggl(
      \mathbb{E}\biggl[
        \biggl(\max_{1 \leqslant j < J} \frac{1}{\tau} \int_{t_j}^{t_{j+1}} \|z_h(t) - Z_j\|_{(H_{D,h}^{0,q},H_{D,h}^{1,q})_{1-2/p,p}} \, \mathrm{d}t\biggr)^p
      \biggr]
    \biggr)^{1/p}.
  \end{align*}
  By the $h$-uniform embedding
  $B_{p,q} \hookrightarrow (H_{D,h}^{0,q},H_{D,h}^{1,q})_{1-2/p,p}$
  (see \cref{rem:dotHh-reiteration}),
  estimate \eqref{eq:zh-LpCHq} implies that the first term is bounded by $\mathscr{G}$.
  To estimate the second term, we employ Hölder's inequality, which yields
  \begin{align*}
    &\biggl(
      \mathbb{E}\biggl[
        \biggl(\max_{1 \leqslant j < J} \frac{1}{\tau} \int_{t_j}^{t_{j+1}} \|z_h(t) - Z_j\|_{(H_{D,h}^{0,q},H_{D,h}^{1,q})_{1-2/p,p}} \, \mathrm{d}t\biggr)^p
      \biggr]
    \biggr)^{1/p} \\
    \leqslant{}& \tau^{-1/p}
    \biggl(
      \mathbb{E}\biggl[
        \max_{1 \leqslant j < J} \int_{t_j}^{t_{j+1}} \|z_h(t) - Z_j\|_{(H_{D,h}^{0,q},H_{D,h}^{1,q})_{1-2/p,p}}^p \, \mathrm{d}t
      \biggr]
    \biggr)^{1/p} \\
    \leqslant{}& \tau^{-1/p}
    \biggl(
      \mathbb{E}\biggl[
        \sum_{j=0}^{J-1} \int_{t_j}^{t_{j+1}} \|z_h(t) - Z_j\|_{(H_{D,h}^{0,q},H_{D,h}^{1,q})_{1-2/p,p}}^p \, \mathrm{d}t
      \biggr]
    \biggr)^{1/p}.
  \end{align*}
  Invoking \eqref{eq:interpolation-error} shows that this term is also controlled by $\mathscr{G}$. Consequently, we conclude that
  \begin{align*}
    \bigg(\mathbb{E} \bigg[ \max_{1\leqslant j < J} \|Z_j\|_{(H_{D,h}^{0,q},H_{D,h}^{1,q})_{1-2/p,p}}^p \bigg]\bigg)^{1/p}
    \lesssim \mathscr{G}.
  \end{align*}
  Finally, by the $h$-uniform embedding
  $(H_{D,h}^{0,q},H_{D,h}^{1,q})_{1-2/p,p} \hookrightarrow B_{p,q}$
  (see \cref{rem:dotHh-reiteration}), we obtain
  \begin{equation}
    \label{eq:lxy1}
    \bigg(\mathbb{E} \bigg[ \max_{1\leqslant j < J} \|Z_j\|_{B_{p,q}}^p \bigg]\bigg)^{1/p}
    \lesssim \mathscr{G}.
  \end{equation}

  \textbf{Step 2.} By definition, we have
  $Z_J = (I - \tau A_h)^{-1} \big( Z_{J-1} + \int_{t_{J-1}}^{t_J} g_h(t) \, \mathrm{d}W(t) \big)$.
  Lemma \ref{lem:resolvent-bound}(i) and real interpolation yield the bound
  $\|(I - \tau A_h)^{-1}\|_{\mathcal{L}(B_{p,q},B_{p,q})} \lesssim 1$. Consequently,
  \[
    \|(I - \tau A_h)^{-1} Z_{J-1}\|_{L^p(\Omega;B_{p,q})}
    \lesssim \|Z_{J-1}\|_{L^p(\Omega;B_{p,q})}.
  \]
  Similarly, we obtain the operator norm estimate
  $ \|(I - \tau A_h)^{-1}\|_{\mathcal{L}(H_{D,h}^{0,q}, B_{p,q})} \lesssim \tau^{-(1/2-1/p)}$.
  Applying Proposition \ref{prop:stoch-int}, the ideal property \eqref{eq:ideal}, and Hölder's inequality, we deduce that
  \begin{align*}
    & \bigg( \mathbb{E}\bigg[
        \Big\| (I-\tau A_h)^{-1} \int_{t_{J-1}}^{t_J} g_h(t) \, \mathrm{d}W(t) \Big\|_{B_{p,q}}^p
    \bigg] \bigg)^{1/p} \\
    \lesssim{} & \bigg( \mathbb{E} \bigg[
        \Big(
          \int_{t_{J-1}}^{t_J} \tau^{-(1-2/p)} \|g_h(t)\|_{\gamma(\ell^2,H_{D,h}^{0,q})}^2 \, \mathrm{d}t
        \Big)^{p/2}
    \bigg] \bigg)^{1/p} \\
    \lesssim{} & \bigg( \mathbb{E} \bigg[ \int_{t_{J-1}}^{t_J} \|g_h(t)\|_{\gamma(\ell^2,H_{D,h}^{0,q})}^p \, \mathrm{d}t \bigg] \bigg)^{1/p}
    \leqslant \mathscr{G}.
  \end{align*}
  Combining the preceding estimates via the triangle inequality gives
  $ \|Z_J\|_{L^p(\Omega;B_{p,q})}
  \lesssim \|Z_{J-1}\|_{L^p(\Omega;B_{p,q})}
  + \mathscr{G}$.
  Since \eqref{eq:lxy1} implies $ \|Z_{J-1}\|_{L^p(\Omega;B_{p,q})} \lesssim \mathscr{G}$, we conclude that
  $\|Z_J\|_{L^p(\Omega;B_{p,q})} \lesssim \mathscr{G}$.
  Combining this estimate with \cref{eq:lxy1} and using the triangle inequality yields the desired estimate \eqref{eq:discrete_stoch_maximal2}, since
  $\max_{1 \leqslant j \leqslant J} \|Z_j\|_{B_{p,q}}
  \leqslant \max_{1 \leqslant j < J} \|Z_j\|_{B_{p,q}}
  + \|Z_J\|_{B_{p,q}}$ holds $\mathbb{P}$-a.s.
  This completes the proof.
\end{proof}

\begin{remark}
  \label{rem:dotHh-reiteration}
  Let $p \in (2,\infty)$ and $q \in [2,\infty)$ be fixed. 
  First, observe that \cref{lem:dotHh}(ii) ensures the 
  $h$-uniform embedding $H_{D,h}^{2,q} \hookrightarrow H_{D,h}^{0,q}$, and
  \cref{eq:dotHh-complex-interp} yields the norm equivalence 
  $\|\cdot\|_{H_{D,h}^{1,q}} \sim \|\cdot\|_{[H_{D,h}^{0,q},H_{D,h}^{2,q}]_{1/2}}$ on $X_h$. 
  Let $K(t, \cdot)$ denote the $K$-functional associated with the real interpolation couple $(H_{D,h}^{0,q}, H_{D,h}^{2,q})$, as defined in \cite[Definition~1.1]{Lunardi2018}. 
  By definition, the basic estimates
  \begin{align*}
    K(t,v_h) \leqslant \|v_h\|_{H_{D,h}^{0,q}}, \qquad K(t,v_h) \leqslant t\|v_h\|_{H_{D,h}^{2,q}}
  \end{align*}
  hold for all $t \in (0,\infty)$ and $v_h \in X_h$.
  By complex interpolation \cite[Theorem~2.7]{Lunardi2018},
  we obtain the estimate $K(t,v_h) \leqslant C t^{1/2} \|v_h\|_{H_{D,h}^{1,q}}$
  for all $t \in (0,\infty)$ and $v_h \in X_h$,
  where the constant $C$ is independent of $t$, $h$, and $v_h$. 
  Furthermore, \cite[Corollary~2.8]{Lunardi2018} yields the interpolation inequality
  \[
  \|v_h\|_{H_{D,h}^{1,q}} \leqslant C \|v_h\|_{H_{D,h}^{0,q}}^{1/2} \|v_h\|_{H_{D,h}^{2,q}}^{1/2},
  \quad v_h \in X_h,
  \]
  where $C>0$ is independent of $h$ and $v_h$.
  Thus, applying \cite[Theorem~1.23]{Lunardi2018} (with $(\theta_0,\theta_1,\theta,E_0,E_1,X,Y) = (0,1/2,1-2/p,H_{D,h}^{0,q},H_{D,h}^{1,q},H_{D,h}^{0,q},H_{D,h}^{2,q})$)
  yields the following $h$-uniform embeddings:
  \[
    (H_{D,h}^{0,q},H_{D,h}^{1,q})_{1-2/p,p} \hookrightarrow (H_{D,h}^{0,q},H_{D,h}^{2,q})_{1/2-1/p,p} \hookrightarrow (H_{D,h}^{0,q},H_{D,h}^{1,q})_{1-2/p,p}.
  \]
\end{remark}

\subsection{Proof of \texorpdfstring{\cref{prop:U}}{}}
\label{ssec:proof-U}
\textbf{Part (a): Existence of solutions.}
The existence of solutions to \cref{eq:U} can be established via the Leray--Schauder fixed point theorem
and the Kuratowski--Ryll-Nardzewski measurable selection theorem; see \cite[Lemma~4.1]{Banas2014convergent}
for analogous arguments. We include the proof here for completeness. 
We first note that $X_h$ is a finite-dimensional space, and let $\mathcal{B}(X_h)$ denote its Borel $\sigma$-algebra.
Fix $w_h \in X_h$ and define a continuous nonlinear operator $\mathcal{T}_h^{w_h} \colon X_h \to X_h$ by  
\[
\mathcal{T}_h^{w_h} v_h := \tau A_h v_h - \frac{\tau}{2} P_h\partial_x(v_h^2) + w_h, \quad v_h \in X_h.
\]
Since $X_h$ is finite-dimensional, $\mathcal{T}_h^{w_h}$ is compact.
Moreover, for $\sigma \in [0,1]$ and $v_h \in X_h$ satisfying $v_h = \sigma \mathcal{T}_h^{w_h} v_h$,
taking the $L^2(\mathcal{O})$-inner product of this equation with $v_h$ and using the identity 
\cref{eq:key-identity}, together with the Cauchy--Schwarz inequality, we obtain  
\[
\|v_h\|_{L^2(\mathcal{O})}^2 \leqslant -\sigma\tau \|v_h\|_{H_{D,h}^{1,2}}^2 + \sigma \|v_h\|_{L^2(\mathcal{O})} \|w_h\|_{L^2(\mathcal{O})}.
\]
Consequently, $\|v_h\|_{L^2(\mathcal{O})} \leqslant \sigma\|w_h\|_{L^2(\mathcal{O})}$,
and $v_h$ is therefore bounded uniformly in $\sigma \in [0,1]$.
By the Leray--Schauder fixed-point theorem (see, e.g., \cite[Theorem~11.3]{Gilbarg2001}),
there exists at least one $v_h \in X_h$ satisfying $v_h = \mathcal{T}_h^{w_h} v_h$. 
Equivalently, for any $w_h \in X_h$, the equation $F(v_h, w_h) = 0$ admits at least one solution $v_h \in X_h$,
where the continuous nonlinear operator $F \colon X_h \times X_h \to X_h$ is defined by
\[
  F(v_h,w_h) := (I - \tau A_h)v_h + \frac{\tau}{2} P_h\partial_x (v_h^2) - w_h.
\]
Consequently, the set-valued mapping $\Gamma \colon \Omega \to 2^{X_h}$, defined for $\mathbb{P}$-a.s.~$\omega$ by
\[
  \Gamma(\omega) = \left\{ v_h \in X_h : F\biggl(v_h, \, U_0 + \int_{t_0}^{t_1} P_hQ \, \mathrm{d}W(s)\biggr) = 0 \right\},
\]
is such that $\Gamma(\omega)$ is a nonempty, closed subset of $X_h$ $\mathbb{P}$-a.s.,
where the closedness follows from the continuity of $F$.

Furthermore, from the strongly $\mathcal{F}_{t_1}$-measurability of $U_0 + \int_{t_0}^{t_1} P_hQ \,\mathrm{d}W(s)$, along with the continuity of $F$, we infer that the zero set
\[
  \mathcal{Z} := \biggl\{ (\omega, v_h) \in \Omega \times X_h : F\biggl(v_h, U_0 + \int_{t_0}^{t_1} P_hQ\,\mathrm{d}W(s)\biggr) = 0 \biggr\}
\]
belongs to $\mathcal{F}_{t_1} \otimes \mathcal{B}(X_h)$. 
For any open set $\mathcal{U} \subset X_h$, the event $\{\omega \in \Omega : \Gamma(\omega) \cap \mathcal{U} \neq \emptyset\}$
coincides with the projection of $\mathcal{Z} \cap (\Omega \times \mathcal{U})$ onto $\Omega$. 
Since $(\Omega, \mathcal{F}_{t_1}, \mathbb{P})$ is a complete probability space and $X_h$ is a Polish space,
the measurable projection theorem \cite[Theorem 18.25]{aliprantis2006infinite} guarantees that this event is
$\mathcal{F}_{t_1}$-measurable, thereby establishing that $\Gamma$ is weakly measurable (see \cite[Definition~18.1]{aliprantis2006infinite}). 
By the Kuratowski--Ryll-Nardzewski measurable selection theorem \cite[Theorem~18.13]{aliprantis2006infinite},
there exists an $\mathcal{F}_{t_1}$-measurable random variable $ U_1 \colon \Omega \to X_h $
such that $U_1(\omega) \in \Gamma(\omega)$ $\mathbb{P}$-a.s. By the definition of $\Gamma$,
\[
  U_1 - U_0 = \tau A_h U_1 - \frac{\tau}{2} P_h\partial_x(U_1^2)
  + \int_{t_0}^{t_1} P_hQ \, \mathrm{d}W(s),
  \quad \mathbb{P}\text{-a.s.}
\]

Proceeding inductively, we construct a sequence $(U_j)_{j=0}^J$ such that, for each $j$, $U_j \colon \Omega \to X_h$ is an $\mathcal{F}_{t_j}$-measurable random variable that satisfies \cref{eq:U} $\mathbb{P}$-a.s. This yields a solution to \cref{eq:U}.

\textbf{Part (b): Proof of \cref{eq:U-LpH1q} and \cref{eq:U-LpCHq}.}
The proof closely follows the arguments used to establish \cref{eq:uh-LpH1q,eq:uh-LpCHq}.
By exploiting the identity \(\int_{\mathcal{O}} P_h\partial_x(v_h^2) \, v_h \, \mathrm{d}x = 0\) for all \(v_h \in X_h\) and the fact that the initial datum \(u_0 \in W_0^{1,\infty}(\mathcal{O})\) is deterministic, standard arguments (cf.\ \cite[Theorem~2.6]{gyongy2007rate}) yield
\begin{equation}
  \label{eq:U-LpCL2}
  \mathbb{E} \bigg[ \max_{0\leqslant j \leqslant J} \|U_j\|_{L^2(\mathcal{O})}^p \bigg]
  \lesssim 1 \quad \text{for all } p \in [2,\infty).
\end{equation}
Define the sequence \((G_{h,j})_{j=0}^J\) \(\mathbb{P}\)-a.s.~by
\begin{equation}
  \label{eq:Ghj-def}
  \begin{cases}
    G_{h,0} = 0, \\
    G_{h,j+1} - G_{h,j} - \tau A_h G_{h,j+1} = \displaystyle\int_{t_j}^{t_{j+1}} P_hQ \, \mathrm{d}W(t), \qquad 0 \leqslant j < J.
  \end{cases}
\end{equation}
Applying \cref{lem:discrete_stoch_maximal} with \(g_h \equiv P_h Q\), together with \eqref{eq:PhQ-bound} and \cref{lem:dotHh-equiv}, yields \(\mathbb{E} \bigl[ \sum_{j=1}^J \tau \|G_{h,j}\|_{H_{D}^{1,q}}^p \bigr] \lesssim 1\) for \(p \in (2,\infty)\) and \(q \in [2,\infty)\). Extending this to \(p=2\) via the embedding \(L^p(\Omega) \hookrightarrow L^2(\Omega)\), we conclude that
\begin{equation}
    \label{eq:Ghj-LpH1q} 
   \mathbb{E} \biggl[ \sum_{j=1}^J \tau \|G_{h,j}\|_{H_{D}^{1,q}}^p \biggr] \lesssim 1 
   \quad\text{for all } p,q \in [2,\infty).
\end{equation}
Moreover, by the same argument leading to \eqref{eq:G-LpCHq}—combining the discrete maximal inequality \eqref{eq:discrete_stoch_maximal2} with \cref{lem:dotHh}\textup{(iii)} and \cref{lem:dotHh-equiv}—we obtain its discrete analogue:
\begin{equation}
    \label{eq:Ghj-LpCHq}
   \mathbb{E} \biggl[ \max_{1\leqslant j\leqslant J} \|G_{h,j}\|_{H_{D}^{\theta,q}}^p \biggr] \lesssim 1
   \quad\text{for all } p,q \in [2,\infty) \text{ and } \theta \in (0,1).
\end{equation}
Observing that \(u_0 \in W_0^{1,\infty}(\mathcal{O}) \hookrightarrow H_{D}^{1,q}\) for all \(q \in [2,\infty)\),
we infer from the \(h\)-uniform stability of \(P_h\) in \(\mathcal{L}(H_{D}^{1,q},H_{D,h}^{1,q})\) (\cref{lem:Ph-stab}),
\cref{lem:resolvent-bound}(i) with \(\theta_1 = \theta_2 = 1\),
and the norm equivalence \(\|\cdot\|_{H_{D}^{1,q}} \sim \|\cdot\|_{H_{D,h}^{1,q}}\)
on \(X_h\) (\cref{lem:dotHh-equiv}) that
\begin{equation}
  \label{eq:discrete-ShPhu0}
   \mathbb{E} \biggl[ \max_{0 \leqslant j \leqslant J} \|(I - \tau A_h)^{-j}P_hu_0\|_{H_{D}^{1,q}}^p \biggr]
   \lesssim 1 \quad\text{for all } p,q \in [2,\infty).
\end{equation}
Combining \cref{eq:U} with the definition of $G_{h,j}$ in \cref{eq:Ghj-def}, we arrive at,
$\mathbb{P}$-a.s.,
\begin{equation} \label{eq:U-discrete-mild}
  U_j = (I - \tau A_h)^{-j}P_hu_0 - \frac{\tau}{2} \sum_{k=0}^{j-1} (I - \tau A_h)^{-(j-k)}P_h\partial_x (U_{k+1}^2) + G_{h,j},
  \quad 1 \leqslant j \leqslant J.
\end{equation}
In analogy with \cref{eq:uh2-bound0}, the estimate \cref{eq:U-LpCL2} implies that,
for all $p \in [2,\infty)$ and $\delta \in (0,1/2)$,
\[
\mathbb{E} \bigg[ \max_{0\leqslant j \leqslant J} \|P_h\partial_x(U_j^2)\|_{H_{D,h}^{-3/2-\delta,2}}^p \bigg] \lesssim 1.
\]
Applying \cref{lem:discrete-maximal} and the norm equivalence from \cref{lem:dotHh-equiv} yields,
for all $p \in [2,\infty)$ and $\varepsilon \in (0,1/2)$,
\[
  \mathbb{E}\bigg[ \max_{1 \leqslant j \leqslant J}
    \Big\| \tau \sum_{k=0}^{j-1} (I - \tau A_h)^{-(j-k)}P_h\partial_x (U_{k+1}^2) \Big\|_{H_{D}^{1/2-\varepsilon,2}}^p
  \bigg] \lesssim 1.
\]
In conjunction with \cref{eq:discrete-ShPhu0,eq:Ghj-LpCHq},
this bound and the representation \cref{eq:U-discrete-mild} allow us to deduce that,
for all $p \in [2,\infty)$ and $\varepsilon \in (0,1/2)$,
\[
\mathbb{E} \bigg[ \max_{0\leqslant j \leqslant J} \|U_j\|_{H_{D}^{1/2-\varepsilon,2}}^p \bigg] \lesssim 1
\]
 A similar bootstrap argument to that used for \cref{eq:Sh-ast-uh2}—relying on \cref{lem:discrete-maximal}
 instead of \cref{lem:Sh-ast-bound}(i)—leads to the higher spatial regularity estimate,
 for all $p,q \in [2,\infty)$,
\begin{equation}
  \label{eq:lbj}
  \mathbb{E}\bigg[ \max_{1 \leqslant j \leqslant J}
    \Big\| \tau \sum_{k=0}^{j-1} (I - \tau A_h)^{-(j-k)}P_h\partial_x (U_{k+1}^2) \Big\|_{H_{D}^{1,q}}^p
  \bigg] \lesssim 1.
\end{equation}
 Finally, \cref{eq:U-LpH1q} follows directly from \cref{eq:U-discrete-mild} combined with the estimates
 \cref{eq:lbj,eq:Ghj-LpH1q,eq:discrete-ShPhu0}. Likewise, the estimate \cref{eq:U-LpCHq} is obtained by
 combining \cref{eq:U-discrete-mild} with \cref{eq:lbj,eq:Ghj-LpCHq,eq:discrete-ShPhu0}.

\textbf{Part (c): Proof of \cref{eq:U-exp-moment}.}
The exponential moment bound \cref{eq:U-exp-moment}
is the discrete counterpart to the exponential moment bound \cref{eq:u-exp-moment}
for the continuous setting. The proof is similar to that of
\cite[Theorem~15]{Bessaih2022spacetime} for a numerical scheme for the 2D stochastic Navier--Stokes equations
with periodic boundary conditions.
We provide the proof in \cref{sec:proof_U_exp_moment} for completeness.

\qed

\subsection{Proof of Theorem~\ref{thm:uh-U}}
\label{ssec:proof-thm-uh-U}
Let $\bar{U}$ denote the piecewise constant-in-time process given by $\bar{U}(0) = U_0$ and
\begin{equation}
  \label{eq:barU-def}
  \bar{U}(t) := U_{j+1} \quad \text{for all } t \in (t_j,t_{j+1}] \text{ with } 0 \leqslant j < J.
\end{equation}
The stability estimate \cref{eq:U-LpH1q} implies
\begin{equation} \label{eq:barU-LpLq}
  \sup_{0 < h,\tau \leqslant 1} \|\bar{U}\|_{L^p([0,T] \times \Omega;H_{D}^{1,q})} < \infty
  \quad \text{for all } p,q \in [2,\infty).
\end{equation}
Next, the stability estimate \cref{eq:U-LpCHq} implies 
\[
\sup_{0 < h,\tau \leqslant 1} \|\bar{U}\|_{L^p(\Omega;L^\infty(0,T;H_{D}^{\theta,q}))} < \infty
\quad \text{for all } p,q \in [2,\infty) \text{ and } \theta \in (0,1),
\]
which, together with \cref{eq:uh-LpCHq},
the triangle inequality, and the embedding $H_{D}^{\theta,q} \hookrightarrow L^\infty(\mathcal{O})$
(valid for $\theta q > 1$), implies that
\begin{equation}
  \label{eq:uh+barU-LpLinf}
  \sup_{0 < h,\tau \leqslant 1} \|u_h+\bar{U}\|_{L^p(\Omega;L^\infty(0,T;L^\infty(\mathcal{O})))} < \infty
  \quad \text{for all } p \in [2,\infty).
\end{equation}
Furthermore, since $\|\bar{U}\|_{L^2(0,T;H_{D}^{1,2})}^2 = \tau\sum_{j=1}^{J}\|U_j\|_{H_{D}^{1,2}}^2$
$\mathbb{P}$-a.s.,
the exponential moment bounds in \cref{prop:uh-regu} and Proposition~\ref{prop:U}(ii),
together with the elementary inequality $(a+b)^2 \leqslant 2a^2 + 2b^2$ and H\"{o}lder's inequality,
imply the existence of constants $\kappa_0 > 0$ and $\tau_0 \in (0,1)$ such that
\begin{equation}
  \label{eq:uh+barU-exp-moment}
  \sup_{\substack{0 < h \leqslant 1\\ 0 < \tau \leqslant \tau_0}}
  \mathbb{E}\Bigl[\exp\!\bigl(\kappa_0\|u_h+\bar{U}\|_{L^2(0,T;H_{D}^{1,2})}^2\bigr)\Bigr] < \infty .
\end{equation}
We may assume without loss of generality that $\tau \leqslant \tau_0$. Indeed,
for $\tau > \tau_0$, the estimates \cref{eq:uh-LpCHq,eq:U-LpCHq}, together with the triangle inequality and
the Sobolev embedding theorem, imply
\[
\sup_{0<h,\tau<1} \mathbb{E}\Bigl[ \|u_h(t_j) - U_j\|_{L^\infty(\mathcal{O})}^p \Bigr] < \infty
\quad \text{for all } p \in [2,\infty),
\]
so the desired error estimate \cref{eq:uh-U-Linfty} holds 
immediately for a sufficiently large constant $C$.

The remainder of the proof is divided into two parts:
Part~(a) proves the error estimate \cref{eq:uh-U-Linfty},
while Part~(b) contains the proofs of the necessary auxiliary estimates.

\textbf{Part (a).}
Let $\widetilde{U}: [0,T] \times \Omega \to X_h$ be a process with 
$\mathbb{P}$-a.s.~continuous paths satisfying $\mathbb{P}$-a.s. for all $t \in [0,T]$,
\begin{equation}
  \label{eq:widetildeU-def}
  \widetilde{U}(t) = P_h u_0 + \int_0^t \Big( A_h\bar{U}(s) - \frac{1}{2}P_h\partial_x\big(\bar{U}^2(s)\big) \Big) \, \mathrm{d}s + \int_0^t P_h Q \, \mathrm{d}W(s).
\end{equation}
By the definitions of $\bar{U}$ and the fully discrete scheme \cref{eq:U}, induction shows that $\widetilde{U}$ interpolates the numerical solution at the temporal nodes; that is,
\begin{equation}
  \label{eq:widetildeU-U}
  \widetilde{U}(t_{j}) = U_{j} \quad \text{$\mathbb{P}$-a.s.\ for all } 0 \leqslant j \leqslant J.
\end{equation}
Define the error process $e_h := u_h - \widetilde{U}$.
Subtracting \cref{eq:widetildeU-def} from \cref{eq:uh-def} gives, $\mathbb{P}$-a.s.,
\begin{equation}
  \label{eq:full-eh}
  e_h(t) = \int_0^t \left( A_h e_h(s) + A_h\big(\widetilde{U}(s) - \bar{U}(s)\big)
  - \dfrac{1}{2} P_h \partial_x \big(u_h^2(s) - \bar{U}^2(s)\big) \right) \, \mathrm{d}s,
  \quad t \in [0,T],
\end{equation}
Using the identity $u_h^2 - \bar{U}^2 = (u_h + \bar{U})e_h + (u_h+\bar{U})(\widetilde{U} - \bar{U})$,
we rewrite \cref{eq:full-eh} as, $\mathbb{P}$-a.s.,
\begin{equation}
  \label{eq:eh-ODE}
  e_h(t) = \int_0^t \Big( A_h e_h(s) - \dfrac{1}{2} P_h \partial_x \big[(u_h(s) + \bar{U}(s)) e_h(s) \big] 
  + w_h(s) \Big) \, \mathrm{d}s, \quad t \in [0,T],
\end{equation}
where 
\begin{equation}
  \label{eq:R-def}
w_h := A_h(\widetilde{U} - \bar{U}) - \frac{1}{2} P_h\partial_x[(u_h + \bar{U})(\widetilde{U} - \bar{U})].
\end{equation}
We decompose the error by introducing an auxiliary continuous process $\sigma_h: [0,T] \times \Omega \to X_h$
satisfying, $\mathbb{P}$-a.s.,
\begin{equation}
  \label{eq:full-Sigmah}
  \sigma_h(t) = \int_0^t \big( A_h\sigma_h(s) + w_h(s) \big) \, \mathrm{d}s,
   \quad t \in [0,T].
\end{equation}
Setting $z_h := e_h - \sigma_h$ and subtracting \cref{eq:full-Sigmah} from \cref{eq:eh-ODE} yields
that, $\mathbb{P}$-a.s., $z_h$ satisfies 
\begin{equation} \label{eq:Xih-ODE}
  z_h(t) = \int_0^t A_hz_h(s) - \frac{1}{2} P_h\partial_x\Big[\big(u_h(s) + \bar{U}(s)\big)(z_h(s) + \sigma_h(s))\Big] \, \mathrm{d}s,
  \quad t \in [0,T].
\end{equation}
In light of the bounds \cref{eq:uh+barU-LpLinf,eq:uh+barU-exp-moment},
we apply \cref{lem:core-stability} with $p_1 = 8p$ and $g = u_h+\bar{U}$ to obtain, for any $p \in [2,\infty)$,
\[
  \|z_h\|_{L^p(\Omega;C([0,T];L^\infty(\mathcal{O})))} \lesssim
  \|\sigma_h\|_{L^{2p}(\Omega; L^{8p}(0,T;L^2(\mathcal{O})))}
  + \|\sigma_h\|_{L^{8p}(\Omega; L^2(0,T;L^2(\mathcal{O})))}.
\]
By the embedding 
\[ 
L^{8p}([0,T] \times \Omega;L^2(\mathcal{O})) \hookrightarrow L^{2p}(\Omega;L^{8p}(0,T;L^{2}(\mathcal{O})))
\cap L^{8p}(\Omega;L^2(0,T;L^2(\mathcal{O}))),
\]
it follows that
\[
\|z_h\|_{L^p(\Omega;C([0,T];L^\infty(\mathcal{O})))}
\lesssim \|\sigma_h\|_{L^{8p}([0,T] \times \Omega;L^2(\mathcal{O}))}.
\]
Now applying \cref{eq:Sigmah-bound1} with exponents \((8p,2)\) from Part (b) yields
\[
\|z_h\|_{L^p(\Omega;C([0,T];L^\infty(\mathcal{O})))} \lesssim \tau^{1/2}.
\]
Since \(e_h=z_h+\sigma_h\), combining this estimate with 
\cref{eq:Sigmah-bound2} from Part (b) via the triangle inequality gives,
for any $p \in [2,\infty)$ and $\varepsilon \in (0,1/2)$,
\[
\|e_h\|_{L^p(\Omega;C([0,T];L^\infty(\mathcal{O})))} \lesssim \tau^{1/2} + \tau^{1/2-\varepsilon}
\lesssim \tau^{1/2-\varepsilon}.
\]
Finally, noting that \(e_h(t_j)=u_h(t_j)-U_j\) \(\mathbb{P}\)-a.s.\ by \cref{eq:widetildeU-U},
we obtain the desired error estimate \cref{eq:uh-U-Linfty}.

\textbf{Part (b).} We establish the following bounds. For any $p, q \in [2, \infty)$, we have
\begin{align}
  & \|\bar{U} - \widetilde{U}\|_{L^p([0,T] \times \Omega; L^q(\mathcal{O}))} \lesssim \tau^{1/2}, \label{eq:key-bound0} \\
  & \|P_h\partial_x[(u_h+\bar{U})(\widetilde{U} - \bar{U})]\|_{L^p([0,T] \times \Omega; H_{D,h}^{-2,q})} \lesssim \tau^{1/2}, \label{eq:key-bound1} \\
  & \|\sigma_h\|_{L^p([0,T] \times \Omega; L^q(\mathcal{O}))} \lesssim \tau^{1/2}, \label{eq:Sigmah-bound1}
\end{align}
and for any $p \in [2, \infty)$ and $\varepsilon \in (0, 1/2)$,
\begin{equation} \label{eq:Sigmah-bound2}
  \|\sigma_h\|_{L^p(\Omega; C([0,T]; L^\infty(\mathcal{O})))} \lesssim \tau^{1/2-\varepsilon}.
\end{equation}

\textit{Proof of \cref{eq:key-bound0}.}
Let $p,q \in [2,\infty)$. For any $t \in (t_j, t_{j+1}]$ with $0 \leqslant j < J $,
\cref{eq:barU-def,eq:widetildeU-U} give $\bar{U}(t) = \widetilde{U}(t_{j+1})$,
which, together with \cref{eq:widetildeU-def}, yields
\[
  \bar{U}(t) - \widetilde{U}(t) =
  \int_t^{t_{j+1}} \Bigl( A_h \bar{U}(s) - \tfrac12 P_h\partial_x\bigl(\bar{U}^2(s)\bigr) \Bigr) \, \mathrm{d}s
  + \int_t^{t_{j+1}} P_h Q \, \mathrm{d}W(s).
\]
By the triangle inequality,
\[
\begin{aligned}
 \|\bar{U} - \widetilde{U}\|_{L^p([0,T] \times \Omega;L^q(\mathcal{O}))}
&\lesssim \bigg( \mathbb{E} \sum_{j=0}^{J-1} \int_{t_j}^{t_{j+1}} \bigg\| \int_t^{t_{j+1}} \big( A_h\bar{U} - \tfrac12 P_h\partial_x(\bar{U}^2) \big)(s) \, \mathrm{d}s \bigg\|_{L^q(\mathcal{O})}^p \, \mathrm{d}t \bigg)^{\!1/p} \\
& \quad + \bigg( \mathbb{E} \sum_{j=0}^{J-1} \int_{t_j}^{t_{j+1}} \bigg\| \int_t^{t_{j+1}} P_h Q \, \mathrm{d}W(s) \bigg\|_{L^q(\mathcal{O})}^p \, \mathrm{d}t \bigg)^{\!1/p}.
\end{aligned}
\]
Since $\bar{U}$ is piecewise constant in time, the first term is bounded by
\[
  \tau \Big( \|A_h\bar{U}\|_{L^p([0,T] \times \Omega;L^q(\mathcal{O}))}
  + \|P_h\partial_x(\bar{U}^2)\|_{L^p([0,T] \times \Omega;L^q(\mathcal{O}))} \Big).
\]
For the linear term, the inverse estimate \cref{eq:inverse} with $(\theta_1,\theta_2) = (1,2)$ and \cref{lem:dotHh-equiv} imply
\[
  \|A_h\bar{U}\|_{L^p([0,T] \times \Omega;L^q(\mathcal{O}))} \lesssim h^{-1} \|\bar{U}\|_{L^p([0,T] \times \Omega;H_{D}^{1,q})}.
\]
For the nonlinear term, $L^q$-stability of $P_h$ (\cref{eq:Ph-stab0}), H\"older's inequality, and $H_{D}^{1,2q} \hookrightarrow L^{2q}(\mathcal{O})$ give
\begin{align*}
  \|P_h\partial_x(\bar{U}^2)\|_{L^p([0,T] \times \Omega;L^q(\mathcal{O}))}
  &\lesssim \|\bar{U}\partial_x\bar{U}\|_{L^p([0,T] \times \Omega;L^q(\mathcal{O}))} \\
  &\leqslant \|\bar{U}\|_{L^{2p}([0,T] \times \Omega;L^{2q}(\mathcal{O}))} \|\partial_x\bar{U}\|_{L^{2p}([0,T] \times \Omega;L^{2q}(\mathcal{O}))} \\
  &\lesssim \|\bar{U}\|_{L^{2p}([0,T] \times \Omega;H_{D}^{1,2q})}^{2}.
\end{align*}
Combining these with the uniform bound \cref{eq:barU-LpLq} for $(p,q)$ and $(2p,2q)$ yields a bound $\lesssim \tau/h$ for the first term.
For the stochastic integral, \cref{prop:stoch-int} and \cref{eq:PhQ-bound} yield a bound $\lesssim \tau^{1/2}$.
Finally, using the mesh condition $\tau \leqslant h^2$ (so $\tau/h \leqslant \tau^{1/2}$),
we obtain the desired estimate \cref{eq:key-bound0}.

\textit{Proof of \cref{eq:key-bound1}.}
Let $p, q \in [2, \infty)$ and $q' = q/(q-1)$. By \cref{lem:Ph-partialx-vw-H2}, the inverse estimate \cref{eq:inverse}, and Hölder's inequality,
\begin{align*}
&\| P_h\partial_x[ (u_h + \bar{U})(\widetilde{U} - \bar{U}) ] \|_{L^p([0,T] \times \Omega; H_{D,h}^{-2,q})} \\
\lesssim{} &
 \|u_h+\bar{U}\|_{L^{2p}([0,T] \times \Omega;H_{D}^{1,q'})}
\|\widetilde{U} - \bar{U}\|_{L^{2p}([0,T] \times \Omega;H_{D,h}^{-1,q})} \\
\lesssim{} &
 \|u_h+\bar{U}\|_{L^{2p}([0,T] \times \Omega;H_{D}^{1,2})}
\|\widetilde{U} - \bar{U}\|_{L^{2p}([0,T] \times \Omega;H_{D,h}^{-1,q})},
\end{align*}
where the second inequality uses the embedding
 $H_{D}^{1,2} \hookrightarrow H_{D}^{1,q'}$ since $q' \leqslant 2$.
 The estimates \cref{eq:uh-LpH1q} (with $p$ replaced by $2p$ and $q$ replaced by $2$)
 and \eqref{eq:barU-LpLq} (with $p$ replaced by $2p$ and $q$ replaced by $2$)
 ensure that the first factor is uniformly bounded
with respect to $h$ and $\tau$. For the second factor,
the $h$-uniform embedding $H_{D,h}^{0,q} \hookrightarrow H_{D,h}^{-1,q}$
(a consequence of \cref{lem:dotHh}(ii)) allows application of \eqref{eq:key-bound0} with exponents $(2p, q)$, yielding a bound of order $\tau^{1/2}$.

\textit{Proof of \cref{eq:Sigmah-bound1}.}
Applying the triangle inequality to the definition of $w_h$ in \cref{eq:R-def} and using the bounds \cref{eq:key-bound0,eq:key-bound1}, we obtain
\begin{equation}
  \|w_h\|_{L^p([0,T] \times \Omega;H_{D,h}^{-2,q})} \lesssim \tau^{1/2}. \label{eq:R-bound} 
\end{equation}
Since $\sigma_h = S_h \ast w_h$ by \cref{eq:full-Sigmah}, the discrete maximal $L^p$-regularity estimate in \cref{lem:S0h}
yields
\[
\|\sigma_h\|_{L^p([0,T] \times \Omega;L^q(\mathcal{O}))} \lesssim \|w_h\|_{L^p([0,T] \times \Omega;H_{D,h}^{-2,q})}
\lesssim \tau^{1/2} \quad \text{for all } p,q \in [2,\infty),
\]
which completes the proof.

\textit{Proof of \cref{eq:Sigmah-bound2}.}
We fix $\varepsilon \in (0,1/2)$ and initially assume $p \in [2/\varepsilon,\infty)$. 
Applying \cref{lem:Sh-ast-bound}(ii) with $(\theta,q) = (4/p-2,p)$, in conjunction with the inverse estimate \cref{eq:inverse} for $(\theta_1,\theta_2) = (-2,4/p-2)$, yields
\[
  \|\sigma_h\|_{L^p(\Omega;C([0,T];L^\infty(\mathcal{O})))}
  \lesssim h^{-4/p} \|w_h\|_{L^p([0,T] \times \Omega;H_{D,h}^{-2,p})}
  \lesssim h^{-4/p} \tau^{1/2},
\]
where in the last step we invoked \cref{eq:R-bound} with $q=p$. Recalling the mesh condition $\tau \leqslant h^2$,
we have $h^{-4/p} = (h^2)^{-2/p} \leqslant \tau^{-2/p}$. Moreover,
since $p \geqslant 2/\varepsilon$ implies $2/p \leqslant \varepsilon$, and noting that $0 < \tau \leqslant \tau_0 < 1$, we deduce
\[
  h^{-4/p} \tau^{1/2} \leqslant \tau^{1/2-2/p} \leqslant \tau^{1/2-\varepsilon}.
\]
This establishes \cref{eq:Sigmah-bound2} for all $\varepsilon \in (0,1/2)$ and $p \in [2/\varepsilon,\infty)$. 
Finally, the inclusion $L^{p_2}(\Omega) \subset L^{p_1}(\Omega)$ for $p_1 \leqslant p_2$
allows us to extend this estimate to all $p \in [2, 2/\varepsilon]$,
thereby completing the proof of \cref{eq:Sigmah-bound2} for all $\varepsilon \in (0,1/2)$ and $p \in [2,\infty)$.
\qed

\section{Numerical results}
\label{sec:numerical}
This section presents numerical experiments to illustrate our theoretical findings. 
For the model problem \cref{eq:burgers}, we set the initial condition $u_0(x) = \sin(\pi x)$ for $x \in \mathcal{O}$. 
The coefficients defining the operator $Q$ in \cref{eq:Q-def} are chosen as
$\lambda_n = n^{-1}$ for $1 \leqslant n \leqslant 600$ and $\lambda_n = 0$ for $n > 600$.

The implementation of the fully discrete scheme \cref{eq:U} requires, at each time step and for every sample path, solving a nonlinear algebraic system of the form
\[
v_h = \tau A_h v_h - \frac{\tau}{2}P_h \partial_x(v_h^2) + w_h,
\]  
where $v_h \in X_h$ denotes the unknown and $w_h \in X_h$ is a prescribed right-hand side. The construction of efficient solvers for these nonlinear systems lies beyond the scope of the present work; instead, we employ a standard fixed-point iteration scheme to approximate $v_h$:
\[
v_h^{(k+1)} = (I - \tau A_h)^{-1} \Bigl( -\frac{\tau}{2}  P_h\partial_x \bigl((v_h^{(k)})^2\bigr) +  w_h \Bigr), \quad k = 0, 1, \dots, 39.
\]
Here, the initial iterate $v_h^{(0)}$ is chosen as the numerical solution from the preceding time step,
and the iteration is performed for a fixed total of $40$ steps.

We first verify the spatial error estimates established in \cref{sec:spatial-semi-discretization}. 
To this end, we fix $T = 0.05$ and $\tau = 1.25 \times 10^{-5}$, set $J=T/\tau$, and define the following error metrics:
\begin{align*}
\mathrm{Err}_{1,h} &:= \left(\mathbb{E}\Bigl[ \max_{1\leqslant j \leqslant J} \|U_j - U_j^{*}\|_{L^2(\mathcal{O})}^2 \Bigr]\right)^{1/2}, \\
\mathrm{Err}_{2,h} &:= \left(\mathbb{E}\Bigl[ \max_{1\leqslant j \leqslant J} \|U_j - U_j^{*}\|_{L^\infty(\mathcal{O})}^2 \Bigr]\right)^{1/2}, \\
\mathrm{Err}_{3,h} &:= \Bigl|\mathbb{E}\Bigl[ \|U_J\|_{L^2(\mathcal{O})}^2 - \|U_J^{*}\|_{L^2(\mathcal{O})}^2 \Bigr] \Bigr|.
\end{align*}
Here, $(U_j)_{j=0}^J$ denotes the numerical solution computed with spatial mesh size $h$ and time step $\tau$, while $(U_j^{*})_{j=0}^J$ represents the reference solution, generated on a finer spatial mesh of size $h_{\mathrm{ref}} = 2^{-9}$ using the same time step $\tau$. 
These error metrics are estimated via Monte Carlo simulation using $500$ independent sample paths.
\Cref{fig:spatial_strong_sub} displays the strong errors $\mathrm{Err}_{1,h}$ and $\mathrm{Err}_{2,h}$ versus $h$ on a log-log scale. 
Both metrics exhibit first-order convergence, in agreement with the theoretical predictions of \cref{thm:uh-strong}. 
Furthermore, \Cref{fig:spatial_weak_sub} plots the weak error $\mathrm{Err}_{3,h}$ versus $h$ on a log-log scale. 
The observed convergence rate is nearly quadratic, which validates the weak error estimate \cref{eq:weak-error-1} (with $p=q=2$) from \cref{thm:uh-weak}.

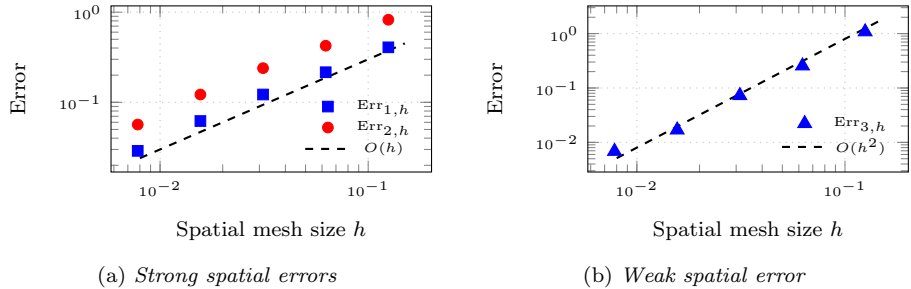
\begin{figure}[htbp]
    \centering
    \begin{subfigure}[b]{0.48\textwidth}
        \centering
        \begin{tikzpicture}
            \begin{axis}[
                xmode=log,
                ymode=log,
                xlabel={\footnotesize Spatial mesh size $h$},
                ylabel={\footnotesize Error},
                yticklabel style={font=\tiny},
                xticklabel style={font=\tiny},
                legend pos=south east,
                legend style={font=\tiny, cells={align=left}, draw=none, fill=none, inner sep=2pt, row sep=-1pt},
                width=\linewidth,
                height=0.65\linewidth,
                grid=major,
                grid style={dotted, gray!50}
              ]

              \addplot[only marks, mark=square*, color=blue, mark size=2pt] coordinates {
                (0.125,     0.407320862168835)
                (0.0625,    0.215556858708228)
                (0.03125,   0.122047652272250)
                (0.015625,  0.061962756885662)
                (0.0078125, 0.028914820729291)
              };
              \addlegendentry{$\mathrm{Err}_{1,h}$}

              \addplot[only marks, mark=*, color=red, mark size=2pt] coordinates {
                (0.125,     0.824083845798806)
                (0.0625,    0.424440686574544)
                (0.03125,   0.239231372061147)
                (0.015625,  0.121910033370202)
                (0.0078125, 0.056636271791737)
              };
              \addlegendentry{$\mathrm{Err}_{2,h}$}

              \addplot[domain=0.008:0.15, thick, dashed, color=black] {3 * x};
              \addlegendentry{$O(h)$}
            \end{axis}
        \end{tikzpicture}
        \caption{Strong spatial errors}
        \label{fig:spatial_strong_sub}
    \end{subfigure}
    \hfill
    \begin{subfigure}[b]{0.48\textwidth}
        \centering
        \begin{tikzpicture}
            \begin{axis}[
                xmode=log,
                ymode=log,
                xlabel={\footnotesize Spatial mesh size $h$},
                ylabel={\footnotesize Error},
                yticklabel style={font=\tiny},
                xticklabel style={font=\tiny},
                legend pos=south east,
                legend style={font=\tiny, cells={align=left}, draw=none, fill=none, inner sep=2pt, row sep=-1pt},
                width=\linewidth,
                height=0.65\linewidth,
                grid=major,
                grid style={dotted, gray!50}
              ]
              
              \addplot[only marks, mark=triangle*, color=blue, mark size=3pt] coordinates {
                (0.125,     1.079709210601505)
                (0.0625,    0.254843346256177)
                (0.03125,   0.072512657380643)
                (0.015625,  0.016901013778658)
                (0.0078125, 0.006838641685798)
              };
              \addlegendentry{$\mathrm{Err}_{3,h}$}

              \addplot[domain=0.008:0.15, thick, dashed, color=black] {80 * x^2};
              \addlegendentry{$O(h^2)$}
            \end{axis}
        \end{tikzpicture}
        \caption{Weak spatial error}
        \label{fig:spatial_weak_sub}
    \end{subfigure}
    \caption{Spatial errors versus mesh size $h$ on a log-log scale: (a) strong errors $\mathrm{Err}_{1,h}$ and $\mathrm{Err}_{2,h}$; (b) weak error $\mathrm{Err}_{3,h}$.}
    \label{fig:spatial_errors_combined}
\end{figure}

We next examine the temporal convergence rate. To this end, we fix $T = 0.0625$ and define the error metric:
\[
  \mathrm{Err}_{\tau} := \left(\mathbb{E}\Bigl[
      \max_{1\leqslant j \leqslant J} \|U_{j}^{\tau} - U_{\lfloor j\tau/\tau^* \rfloor}^{\tau^*}\|_{L^\infty(\mathcal{O})}^2
  \Bigr]\right)^{1/2},
\]
where $J = T/\tau$. Here, $(U_j^{\tau})_{j=0}^J$ denotes the numerical solution computed with time step $\tau$ and spatial mesh size $h = \sqrt{\tau}$,
a choice that balances the spatial and temporal errors. The reference solution $(U_j^{\tau^*})_{j=0}^{J^*}$ is generated using a finer time step $\tau^* = T/4^8$ and the corresponding spatial mesh size $h^* = \sqrt{\tau^*}$. 
The error metric $\mathrm{Err}_{\tau}$ is estimated via Monte Carlo simulation using $500$ independent sample paths.
\Cref{fig:time} displays the strong error $\mathrm{Err}_{\tau}$ versus $\tau$ on a log-log scale. 
The results indicate that the error decays at a rate of approximately $O(\tau^{1/2})$, which aligns with the theoretical predictions.

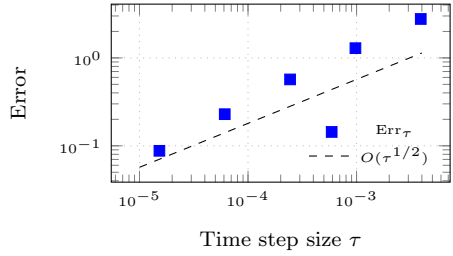
\begin{figure}[htbp]
  \centering
  \begin{tikzpicture}
      \begin{axis}[
          xmode=log,
          ymode=log,
          xlabel={\footnotesize Time step size $\tau$},
          ylabel={\footnotesize Error},
          yticklabel style={font=\tiny},
          xticklabel style={font=\tiny},
          legend pos=south east,
          legend style={font=\tiny, cells={align=left}, draw=none, fill=none, inner sep=2pt, row sep=-1pt},
          width=0.5\linewidth,
          height=0.65*0.5\linewidth,
          grid=major,
          grid style={dotted, gray!50}
        ]

      \addplot[only marks, mark=square*, color=blue, mark size=2pt] coordinates {
        (0.00390625,     2.7734)
        (0.0009765625,   1.2906)
        (0.000244140625, 0.5693)
        (0.00006103515625, 0.2289)
        (0.0000152587890625, 0.0877)
      };
      \addlegendentry{$\mathrm{Err}_{\tau}$}

      \addplot[domain=1e-5:.004, dashed, color=black] {18*x^(.5)};
      \addlegendentry{$O(\tau^{1/2})$}
    \end{axis}
  \end{tikzpicture}
  \caption{Strong error $\mathrm{Err}_{\tau}$ versus time step size $\tau$ with $h=\sqrt{\tau}$ on a log-log scale.}
  \label{fig:time}
\end{figure}

\section{Conclusions}
\label{sec:concluding}

In this paper, we have conducted a rigorous error analysis for the finite element approximation of
the one-dimensional stochastic Burgers equation driven by additive trace-class noise. 
For the spatial semi-discretization based on $P_2$ finite elements, we established regularity-optimal strong convergence rates in $L^p([0,T] \times \Omega; H_{D}^{\alpha,q})$ and almost regularity-optimal rates in $L^p(\Omega; C([0,T]; H_{D}^{\alpha,q}))$. 
Notably, we also obtained an almost regularity-optimal convergence rate in $L^p(\Omega; C([0,T]; L^\infty(\mathcal{O})))$. 
Some weak error estimates in the $L^q$-setting were derived, 
 with weak convergence rates (nearly) twice the corresponding strong ones.
For the fully discrete scheme, obtained by coupling the $P_2$ finite element method with
a drift-implicit Euler-Maruyama scheme in time,
we proved a temporal strong convergence rate of order $\tau^{1/2-\varepsilon}$ in the
discrete space-time maximum norm, subject to the condition $\tau \leqslant h^2$.
The numerical experiments corroborate the theoretically predicted convergence rates.

The present work combines the classical $L^2$-based techniques for
the numerical analysis of stochastic partial differential equations with the discrete stochastic maximal $L^p$-regularity
theory, thereby providing a numerical analysis of the stochastic Burgers equation in
general spatial $L^q$-spaces.
The proposed framework is expected to extend to other nonlinear
stochastic partial differential equations.
In particular, for the fully discrete finite element approximation of the
one-dimensional stochastic Allen--Cahn equation driven by space-time white noise,
it would be of considerable interest to establish strong error estimates in the
$L^p(\Omega; C([0,T]; L^\infty(\mathcal{O})))$-norm, as well as weak error estimates for the
terminal-time functionals $\|\cdot\|_{L^q(\mathcal{O})}^p$; both problems
will be addressed in forthcoming work.

\section*{Acknowledgements}
This work was partially supported by National Natural Science Foundation of China (12301525,12571434)
and by   Natural Science Foundation of Sichuan Province (2026YFTX0024).

\appendix

\section{Proof of Proposition~\ref{prop:U}(ii)}
\label{sec:proof_U_exp_moment}

When $Q \equiv 0$, the bound \cref{eq:u-exp-moment} holds trivially by
\cref{eq:U-LpH1q} (applied with $p=q=2$) given that $(U_j)_{j=0}^{J}$
is deterministic. We thus assume $Q \neq 0$ in the sequel.

For any fixed $v \in L^2(\mathcal{O})$, the mapping $w \mapsto \langle w, v\rangle$ defines a bounded linear functional
on $L^2(\mathcal{O})$, where $\langle\cdot,\cdot\rangle$ denotes the $L^2(\mathcal{O})$-inner product.
Consequently, in view of the ideal property \cref{eq:ideal}
and the property $Q \in \gamma(\ell^2, L^2(\mathcal{O}))$ (see \cref{eq:Q-property}),
the map $v \mapsto \langle Q, v\rangle$ is a bounded linear operator from $L^2(\mathcal{O})$ into $\gamma(\ell^2,\mathbb{R})$,
satisfying the norm estimate
\[
\|\langle Q, v\rangle\|_{\gamma(\ell^2,\mathbb{R})} \leqslant \|Q\|_{\gamma(\ell^2,L^2(\mathcal{O}))} \|v\|_{L^2(\mathcal{O})}, \quad v \in L^2(\mathcal{O}).
\]
Let $U$ denote the piecewise constant process defined by $U(t) = U_j$ for $t \in [t_j, t_{j+1})$, where $0 \leqslant j < J$, and $U(T) = U_J$. Since \cref{eq:U-LpH1q} holds and each $U_j$ is $\mathcal{F}_{t_j}$-measurable, it follows that $U \in L_{\mathbb{F}}^2([0,T] \times \Omega;L^2(\mathcal{O}))$. Hence, $\langle Q, U \rangle \in L_{\mathbb{F}}^2([0,T] \times \Omega;\gamma(\ell^2,\mathbb{R}))$, and the inequality
\begin{equation}
  \label{eq:QU}
\|\langle Q, U \rangle\|_{\gamma(\ell^2,\mathbb{R})} \leqslant \|Q\|_{\gamma(\ell^2,L^2(\mathcal{O}))} \|U\|_{L^2(\mathcal{O})}
\end{equation}
holds $(\mathbb{P}\otimes \mathrm{d}t)$-a.e.~on $[0,T] \times \Omega$.

Let \(M\) be the continuous \(\mathbb{R}\)-valued martingale defined by
$ M(t)=\int_0^t \langle Q,U(s)\rangle\,\mathrm{d}W(s)$, $t\in[0,T]$.
By the identity \cref{eq:key-identity} and the norm equivalence $\|\cdot\|_{H_{D,h}^{1,2}} = \|\cdot\|_{H_{D}^{1,2}}$
on $X_h$, standard arguments (see \cite[Equation~(2.22)]{gyongy2007rate})
show that, for each \(0\leqslant j<J\), \(\mathbb{P}\)-a.s.,
\begin{align*}
& \|U_{j+1}\|_{L^2(\mathcal{O})}^2 + 2\tau \|U_{j+1}\|_{H_{D}^{1,2}}^2 \\
\leqslant{} & \|U_j\|_{L^2(\mathcal{O})}^2
 + 2\bigl(M(t_{j+1})-M(t_j)\bigr) 
 + \Bigl\|\int_{t_j}^{t_{j+1}} P_h Q\,\mathrm{d}W(s)\Bigr\|_{L^2(\mathcal{O})}^2.
\end{align*}
Summing over \(j=0,\dots,J-1\) and dropping the nonnegative term
\(\|U_J\|_{L^2(\mathcal{O})}^2\) from the left hand side, we obtain
\begin{equation}
\label{eq:78}
2\tau \sum_{j=1}^J \|U_j\|_{H_{D}^{1,2}}^2
\leqslant \|U_0\|_{L^2(\mathcal{O})}^2
 + 2M(T) + \sum_{j=0}^{J-1} \Bigl\|\int_{t_j}^{t_{j+1}} P_h Q\,\mathrm{d}W(s)\Bigr\|_{L^2(\mathcal{O})}^2.
\end{equation}
The quadratic variation process of $M$ is given by
\[
\langle M \rangle_t = \int_0^t \| \langle Q, U(s) \rangle \|_{\gamma(\ell^2,\mathbb{R})}^2 \, \mathrm{d}s, \quad t \in [0,T].
\]
Using \cref{eq:QU} and the Poincaré inequality $\|v\|_{L^2(\mathcal{O})} \leqslant \pi^{-1} \|v\|_{H_{D}^{1,2}}$
($v \in H_{D}^{1,2}$), we estimate
\begin{align*}
\langle M \rangle_T 
& \leqslant \|Q\|_{\gamma(\ell^2,L^2(\mathcal{O}))}^2 \int_0^{T} \|U(s)\|_{L^2(\mathcal{O})}^2 \, \mathrm{d}s \\
& = \|Q\|_{\gamma(\ell^2,L^2(\mathcal{O}))}^2 
\Big( \tau\|U_0\|_{L^2(\mathcal{O})}^2 + \tau\sum_{j=1}^{J-1} \|U_j\|_{L^2(\mathcal{O})}^2 \Big) \\
& \leqslant \|Q\|_{\gamma(\ell^2,L^2(\mathcal{O}))}^2
\Big(
  \tau\|U_0\|_{L^2(\mathcal{O})}^2 + \frac{\tau}{\pi^2}\sum_{j=1}^J \|U_j\|_{H_{D}^{1,2}}^2
  \Big).
\end{align*}
For arbitrary $\kappa, \delta > 0$, the decomposition $\kappa M(T) = (\kappa M(T) - \frac{\delta}{2} \kappa^2\langle  M \rangle_T) + \frac{\delta}{2} \kappa^2\langle M \rangle_T$ implies
\[
\kappa M(T) \leqslant \Big(\kappa M(T) - \frac{\delta}{2}\kappa^2 \langle M \rangle_T \Big)
+ \frac{\delta}{2} \kappa^2 \|Q\|_{\gamma(\ell^2,L^2(\mathcal{O}))}^2 
\Big(
  \tau\|U_0\|_{L^2(\mathcal{O})}^2 + \frac{\tau}{\pi^2} \sum_{j=1}^J \|U_j\|_{H_{D}^{1,2}}^2
  \Big).
\]
Multiplying \cref{eq:78} by $\kappa$ and substituting the bound above yields, $\mathbb{P}$-a.s.,
\begin{equation}
  \label{eq:zb1}
  \begin{aligned}
    2\kappa \tau \sum_{j=1}^J \|U_j\|_{H_{D}^{1,2}}^2
    &\leqslant \kappa(1 + \kappa\tau \delta\|Q\|_{\gamma(\ell^2,L^2(\mathcal{O}))}^2) \|U_0\|_{L^2(\mathcal{O})}^2
    + 2 \biggl(\kappa M(T) - \frac{\delta}{2} \kappa^2 \langle M \rangle_T \biggr) \\
    & \quad + \frac{\delta}{\pi^2} \kappa^2 \|Q\|_{\gamma(\ell^2,L^2(\mathcal{O}))}^2
    \Big(\tau\sum_{j=1}^{J} \|U_j\|_{H_{D}^{1,2}}^2\Big) 
     + \sum_{j=0}^{J-1} \kappa \biggl\| \int_{t_j}^{t_{j+1}} P_h Q \, \mathrm{d}W(s) \biggr\|_{L^2(\mathcal{O})}^2.
  \end{aligned}
\end{equation}

Set $\delta := \pi^2 \kappa^{-1} \|Q\|_{\gamma(\ell^2,L^2(\mathcal{O}))}^{-2}$, so that $\frac{\delta}{\pi^2} \kappa^2 \|Q\|_{\gamma(\ell^2,L^2(\mathcal{O}))}^2 = \kappa$.
Substituting this identity into \cref{eq:zb1} leads to the estimate
\begin{align*}
\kappa\tau \sum_{j=1}^J \|U_j\|_{H_{D}^{1,2}}^2
&\leqslant \kappa(1 + \tau\pi^2) \|U_0\|_{L^2(\mathcal{O})}^2
+ 2 \Bigl(\kappa M(T) - \frac{\delta}{2}\kappa^2 \langle M \rangle_T \Bigr) \\
& \quad {} + \sum_{j=0}^{J-1} \kappa \biggl\| \int_{t_j}^{t_{j+1}} P_h Q \, \mathrm{d}W(s) \biggr\|_{L^2(\mathcal{O})}^2.
\end{align*}
Applying the elementary inequality $e^{x+y} \leqslant \frac{1}{2}e^{2x} + \frac{1}{2}e^{2y}$ to the exponential of both sides gives
\begin{align*}
\exp\biggl( \kappa\tau \sum_{j=1}^J \|U_j\|_{H_{D}^{1,2}}^2 \biggr)
&\leqslant \frac{1}{2}\exp\bigl(2\kappa(1+\tau\pi^2)\|U_0\|_{L^2(\mathcal{O})}^2\bigr)
\exp\biggl( 4\kappa M(T) - 2\delta\kappa^2 \langle M \rangle_T \biggr) \\
 & \qquad {} + \frac{1}{2}\exp\biggl( 2\kappa\sum_{j=0}^{J-1} \biggl\| \int_{t_j}^{t_{j+1}} P_h Q \, \mathrm{d}W(s) \biggr\|_{L^2(\mathcal{O})}^2 \biggr).
\end{align*}
Taking expectations and recalling that $U_0$ is deterministic
and $\|U_0\|_{L^2(\mathcal{O})} = \|P_hu_0\|_{L^2(\mathcal{O})} \leqslant \|u_0\|_{L^2(\mathcal{O})}$, we conclude
\begin{equation}
\label{eq:exp_bound}
\begin{aligned}
& \mathbb{E}\biggl[\exp\biggl( \kappa\tau\sum_{j=1}^J \|U_j\|_{H_{D}^{1,2}}^2 \biggr)\biggr] \\
\leqslant{} & \frac{1}{2}\exp\bigl(2\kappa(1+\tau\pi^2)\|u_0\|_{L^2(\mathcal{O})}^2\bigr)
 \, \mathbb{E}\biggl[\exp\biggl( 4\kappa M(T) - 2\delta\kappa^2 \langle M \rangle_T \biggr) \biggr]  \\
& \quad {} + \frac{1}{2}\mathbb{E}\biggl[\exp\biggl( 2\kappa\sum_{j=0}^{J-1} \biggl\| \int_{t_j}^{t_{j+1}} P_h Q \, \mathrm{d}W(s) \biggr\|_{L^2(\mathcal{O})}^2 \biggr)\biggr].
\end{aligned}
\end{equation}

Let $\kappa := \frac{\pi^2}{4\|Q\|_{\gamma(\ell^2,L^2(\mathcal{O}))}^2}$, which implies $\delta = 4$.
With this choice, the process $\exp\bigl( 4\kappa M - 2\delta\kappa^2 \langle M \rangle \bigr)$ is precisely
the Doléans-Dade exponential of $4\kappa M$.
As a positive local martingale, it is a supermartingale,
and $M(0) = 0$ $\mathbb{P}$-a.s.; consequently,
\begin{equation}
  \label{eq:exp_bound_1}
\mathbb{E}\biggl[\exp\biggl( 4\kappa M(T) - 2\delta\kappa^2 \langle M \rangle_T \biggr) \biggr]
\leqslant 1.
\end{equation}
Turning to the stochastic integral term, the random variables $\bigl\| \int_{t_j}^{t_{j+1}} P_hQ \,\mathrm{d}W(s) \bigr\|_{L^2(\mathcal{O})}^2$, $0 \leqslant j < J$, are mutually independent.
By \cite[Proposition~2.17]{Pardoux2014} and the estimate
$\|P_hQ\|_{\gamma(\ell^2,L^2(\mathcal{O}))} \leqslant \|Q\|_{\gamma(\ell^2,L^2(\mathcal{O}))}$—which follows
from the ideal property \eqref{eq:ideal} and $\|P_h\|_{\mathcal{L}(L^2(\mathcal{O}),L^2(\mathcal{O}))}=1$—we have,
for $0 < \tau < \frac{1}{\pi^2}$,
\begin{align}
  & \mathbb{E} \biggl[ \exp\biggl( 2\kappa\sum_{j=0}^{J-1}
  \biggl\| \int_{t_j}^{t_{j+1}} P_h Q \, \mathrm{d}W(s) \biggr\|_{L^2(\mathcal{O})}^2 \biggr) \biggr] \notag \\
  ={} & \prod_{j=0}^{J-1} \mathbb{E} \biggl[ \exp\biggl( 2\kappa \biggl\| \int_{t_j}^{t_{j+1}} P_h Q \, \mathrm{d}W(s) \biggr\|_{L^2(\mathcal{O})}^2 \biggr) \biggr] \notag \\
  \leqslant{} & \bigl(1 - 4 \kappa \tau \|P_h Q\|_{\gamma(\ell^2,L^2(\mathcal{O}))}^2\bigr)^{-J/2} \notag \\
  \leqslant{} & \exp\biggl(\frac{2\kappa T\|Q\|_{\gamma(\ell^2,L^2(\mathcal{O}))}^2}{1-4\kappa \tau \|Q\|_{\gamma(\ell^2,L^2(\mathcal{O}))}^2} \biggr)
  = \exp\biggl(\frac{\pi^2 T}{2(1 - \tau \pi^2)}  \biggr),
  \label{eq:exp_bound_2}
\end{align}
where we used the elementary inequality $(1-x)^{-J/2} = (1 + x/(1-x))^{J/2} < \exp(\tfrac{Jx}{2(1-x)})$ for $0 < x < 1$ and the equality $J\tau=T$.
Substituting \cref{eq:exp_bound_1} and \cref{eq:exp_bound_2} into \cref{eq:exp_bound} yields
\[
\mathbb{E} \biggl[ \exp\biggl( \kappa \tau \sum_{j=1}^J \|U_j\|_{H_{D}^{1,2}}^2 \biggr) \biggr]
\leqslant \frac{1}{2} \exp\bigl(2\kappa(1+\tau\pi^2)\|u_0\|_{L^2(\mathcal{O})}^2\bigr)
+ \frac{1}{2} \exp\biggl( \frac{\pi^2 T}{2(1-\tau\pi^2)} \biggr).
\]
Thus, the desired exponential moment bound \cref{eq:U-exp-moment} holds for
every $0 < \tau_0 < \frac{1}{\pi^2}$.

\qed


\end{document}